\documentclass[9pt,reqno,a4paper,oneside]{article}

\usepackage{breakurl}             % Not needed if you use pdflatex only.
\usepackage{underscore}           % Only needed if you use pdflatex.

\usepackage[english]{babel}

\usepackage[a4paper]{geometry}

\usepackage{amssymb}
\usepackage{amsthm}
\usepackage{amsmath}
\usepackage{textcmds}
\usepackage{listings}
\usepackage{comment}
\usepackage{tikz-cd} %commutative diagrams
\usepackage[ruled,vlined]{algorithm2e}
\usepackage{xcolor}
\usepackage{geometry}
\usepackage[colorlinks=true, allcolors=teal]{hyperref}
\usepackage[skip=10pt]{parskip}

\usepackage{subfig}
\usepackage{caption}

\usepackage{thmtools}
\usepackage{thm-restate}
\usepackage{tikz}

\usepackage{titlesec}

\makeatletter
\let\@internalcite\cite
\def\cite{\def\citeauthoryear##1##2{##1, ##2}\@internalcite}
\def\shortcite{\def\citeauthoryear##1{##2}\@internalcite}
\def\@biblabel#1{\def\citeauthoryear##1##2{##1, ##2}[#1]\hfill}
\makeatother

\newcommand{\R}{\mathbb{R}}

\newcommand{\N}{\mathbb{N}}

\newcommand{\E}{\mathbb{E}}

\newcommand{\M}{\mathcal{M}}
\newcommand{\A}{\mathcal{A}}
\DeclareMathOperator{\sign}{sign}
\newcommand\inp[2]{\langle #1, #2 \rangle}

\newcommand{\inv}{^{-1}}

\newcommand{\ord}{\mathcal{O}}

\newcommand{\thupper}{^{\text{th}}}

\newcommand{\p}[1]{\phi_{#1}}

\newcommand{\ddd}{\cdot\cdot\cdot}
\newtheorem{theorem}{Theorem}[section]

\newtheorem{prop}[theorem]{Proposition}

\theoremstyle{definition}
\newtheorem{definition}[theorem]{Definition}%[section]

\theoremstyle{definition}
\theoremstyle{definition}
\newtheorem{example}[theorem]{Example}%[section]
 
\theoremstyle{remark}
\newtheorem{remark}[theorem]{Remark}

\title{Diffusion Geometry}
\author{Iolo Jones \\ Durham University}
\date{July 2024}

\begin{document}

\maketitle

\begin{abstract}

We introduce diffusion geometry as a new framework for geometric and topological data analysis. Diffusion geometry uses the Bakry-Emery $\Gamma$-calculus of Markov diffusion operators to define objects from Riemannian geometry on a wide range of probability spaces. We construct statistical estimators for these objects from a sample of data, and so introduce a whole family of new methods for geometric data analysis and computational geometry. This includes vector fields and differential forms on the data, and many of the important operators in exterior calculus. Unlike existing methods like persistent homology and local principal component analysis, diffusion geometry is explicitly related to Riemannian geometry, and is significantly more robust to noise, significantly faster to compute, provides a richer topological description (like the cup product on cohomology), and is naturally vectorised for statistics and machine learning. We find that diffusion geometry outperforms multiparameter persistent homology as a biomarker for real and simulated tumour histology data and can robustly measure the manifold hypothesis by detecting singularities in manifold-like data.

\end{abstract}

\tableofcontents

\section{Introduction}

Many important problems in science and engineering are closely related to the study of \textit{shape}. 
This is especially true in physical sciences like biology, medicine, materials science, and geology, where the role or function of tissues, organs, materials, and crystals can be largely determined by their shape. 
The recent proliferation of better measures of this shape, such as more sensitive and accurate imagery, CT and MRI scanning, and microscopy, has led to a huge increase in the quality and quantity of \textit{geometric data}. The data have shape, and the shape has meaning. How can we measure it?

This work introduces \textit{diffusion geometry} as a new framework for geometric and topological data analysis. 
The main idea is to take constructions from the mathematical theory of Riemannian geometry (where notions like connectedness, length, perimeter, angles, and curvature are defined and well-understood) and define them on the probability space from which our data are sampled. 
We justify this with Figure \ref{fig:manifold_vs_prob}.

\begin{figure}[h!]
    \centering
    \includegraphics[width=0.62\textwidth]{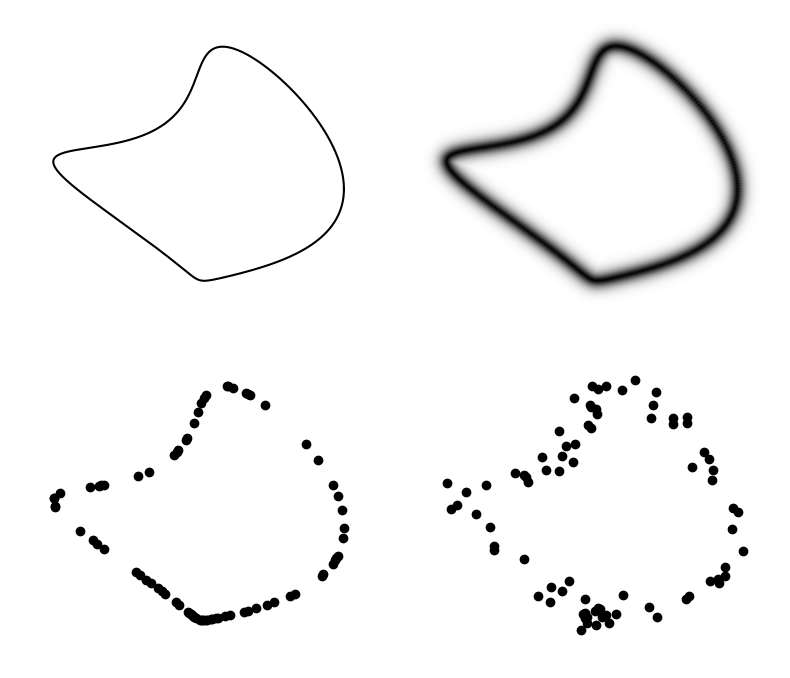}
    \caption{Why probability spaces?}
    \label{fig:manifold_vs_prob}
\end{figure}

In the top left of Figure \ref{fig:manifold_vs_prob} is a manifold, $\M$, whose shape is described by Riemannian geometry. Bottom left is a finite sample of points drawn from $\M$, in which we can clearly see much of the same geometry as $\M$. It stands to reason that we should be able to estimate these features statistically, given such a sample. Top right is a probability density function, $q$, taking large values near $\M$, and perhaps represents a \q{manifold with noise}. It is also clear that most of what can be said about the geometry of $\M$ can be said for the density $q$. Bottom right is a sample drawn from $q$, where again we recognise the same \q{geometry} as $q$, and hope to use this sample to construct estimators for it.

These observations lead to two questions:
\begin{enumerate}
    \item How much of the Riemannian geometry of a manifold is a special case of the \q{Riemannian geometry} of a probability space?
    \item How can the geometry of manifolds and probability spaces be estimated from finite samples of data?
\end{enumerate}
We propose diffusion geometry as a natural extension of Riemannian geometry to certain probability spaces. We show how to estimate it from data as a general framework for geometric and topological machine learning and statistics, as well as computational geometry and topology.
We now briefly highlight the main aspects of diffusion geometry (see Section \ref{section_comparison} for a full comparison with related work).

\subsection{Why define a new theory?}

Applied geometry and topology are often motivated by the idea that \q{data have shape}.
However, data come from probability spaces.
In this paper, we will argue that really it is the \textit{underlying probability space} that has shape.
We will use \textit{Markov diffusion operators}, which generalise the heat flow on a manifold, to develop a theory of Riemannian geometry on a broad class of probability spaces.

Crucially, this theory will lead to a natural statistical framework for estimating this geometry from data.
All the objects we define in the theory have corresponding algorithms that follow straight from the definitions.
While our theory does generalise a variety of existing approaches to metric and measure-theoretic geometry, the main motivation is that it leads \textit{directly} to computational tools for data.

\subsection{Machine learning and statistics}

There are centuries of work in differential and Riemannian geometry that can be exploited for geometric machine learning and statistics.
We can compute objects from this theory to create \textit{shape descriptors} for the data, known as \textit{feature vectors} in machine learning.
These encode a rich geometric representation (measuring things like curvature, connectedness, and the spectrum of the Hodge Laplacian), which can then be used for supervised and unsupervised learning.
As we will explore, these feature vectors are robust to noise and outliers, fast to compute, and explainable, so that the output of the final model can be interpreted and used to learn more about the data.

A fundamental observation of \textit{geometric deep learning} \cite{bronstein2021geometric} is that geometric machine learning models should be invariant under various group actions, such as translations and rotations of the input data.
Our feature vectors, and hence our models, will be invariant under translation, rotation, reflection, and, if required, also rescaling.

However, group invariance alone is not enough.
Geometry is also preserved by adding noise and outliers, and, perhaps, changing the sampling density of the data.
Models should therefore be both group invariant \textit{and} statistically invariant.
The diffusion geometry features are highly robust to noise, and, if required, can be made density-invariant as well.

We test diffusion geometry for supervised and unsupervised learning on tumour histology data and find that it outperforms existing methods.

\subsection{Computational and numerical geometry and topology}

When the data are drawn from a manifold, diffusion geometry agrees with the usual theory of Riemannian geometry.
As such, we can use it as a computational tool for investigating the properties of manifolds, and using optimisation to test examples and conjectures in Riemannian geometry.
This means we can view diffusion geometry as a novel tool for numerical geometry that, uniquely, only requires a sample of points from the manifold, which need not be uniform, and no triangulation.

Many applications of geometry and topology to data assume the \textit{manifold hypothesis}: that the data are drawn from a manifold.
We can also use diffusion geometry to test this hypothesis and develop tools for finding singularities in otherwise manifold-like data.

\subsection{How to read this paper}
The theory of diffusion geometry is developed in Sections
\begin{enumerate}
    \item[(\ref{section_markov})] where we show that Markov diffusion operators give a suitable generalisation of the Laplacian operator on a manifold,
    \item[(\ref{section_theory})] where we define an analogue of Riemannian geometry using this alternative Laplacian, and
    \item[(\ref{section_estimation})] where we show how to estimate this geometry from data.
\end{enumerate}

\textbf{If you are mainly interested in using diffusion geometry for data analysis} then read Sections
\begin{enumerate}
    \item[(\ref{section_ml})] where we test diffusion geometry as a tool for computational geometry and topology, machine learning, and statistics,
    \item[(\ref{section_complexity})] where we test the computational complexity of diffusion geometry, and
    \item[(\ref{section_comparison})] where we review existing methods and compare them with diffusion geometry.
\end{enumerate}

\section{Markov Diffusion Operators and Markov Triples}\label{section_markov}

Much of the Riemannian geometry of a manifold can be described in terms of its Laplacian operator $\Delta$, via the \q{carré du champ} identity
\begin{equation}\label{cdc identity}
\frac{1}{2} \big( f \Delta h + h \Delta f - \Delta (fh) \big) = g(df,dh),
\end{equation}
which gives a formula for the Riemannian metric of 1-forms $df$ and $dh$ in terms of $\Delta$. 
The main idea behind diffusion geometry is to replace the manifold Laplacian $\Delta$ with some other operator $L$ on a more general space so that (\ref{cdc identity}) can become a \textit{definition} for Riemannian geometry on that space. 
In this section, we will propose the use of \textit{Markov diffusion operators} as an appropriate generalisation of the Laplacian.

\subsection{Examples of Markov diffusion operators}

The theory of Markov diffusion operators is very well developed, and we will only state the essential properties here (refer to \cite{bakry2014analysis} for a comprehensive reference). As a motivating example, we consider the heat diffusion in the Euclidean space $\R^n$. The \textit{Gaussian kernel}
\begin{equation}
\label{heat kernel eq}
p_t(x,y) = \frac{1}{(4 \pi t)^{n/2}} \exp\Big(- \frac{\| x-y\|^2}{4t} \Big)
\end{equation}
can be used to define the \textit{heat diffusion operator} $P_t$ as
$$
P_t f(x) = \int p_t(x,y) f(y) dy.
$$
Setting $u(t,x) = P_t f(x)$ then solves the \textit{heat equation}
$$
\frac{\partial u}{\partial t} = \Delta u 
\qquad
u(0, x) = f(x).
$$
So if $f$ is an initial distribution of heat in space, then $P_t f$ is the heat distribution after $t$ time, and $\Delta$ measures the rate of diffusion. 
The operator $P_t$ satisfies
\begin{equation}\label{inf generator}
\lim_{t\rightarrow 0} \frac{P_t f - f}{t} = -\Delta f,
\end{equation}
in $L^2(\R^n)$, so we can interpret the Laplacian $\Delta$ as the derivative of the diffusion operator $P_t$ at $t=0$. 
In the language of Markov diffusion operators, we say that $(P_t)_{t\geq 0}$ is a \textit{Markov semigroup} and $\Delta$ is its \textit{infinitesimal generator}.
By \q{semigroup} we mean that $P_t$ is an operator acting on a function space, and $P_t \circ P_s = P_{t+s}$, and $P_0 = \text{Id}$. 
A semigroup has an infinitesimal generator $L$ defined as the limit in (\ref{inf generator}).
We say that $L$ \q{generates} $P_t$, meaning $P_t = \exp(-tL)$ in the sense of the spectral mapping theorem.
Equivalently, if $\phi$ is an eigenfunction of $L$ with eigenvalue $\lambda$, then it is also an eigenfunction of $P_t$ with eigenvalue $\exp(-t\lambda)$.
There are many important examples of semigroups.

\begin{example}
If $\M$ is a Riemannian manifold, then the Laplacian $\Delta$ defines a heat equation as above, and the associated heat diffusion operator $P_t$ is a semigroup. 
The infinitesimal generator is the Laplacian. 
This is the central example that motivates the rest of this work.
\end{example}

\begin{example}\label{eg weighted manifold}
If $\M$ is a Riemannian manifold with a density $q(x)$, known as a \textit{weighted manifold}, then the heat kernel can be used to define a weighted heat diffusion operator
$$
P_t f(x) = \int p_t(x,y) f(y) q(y) dy.
$$
The operator $P_t$ is also a semigroup, and its infinitesimal generator is given by
$$
\Delta_q = \Delta -  \frac{\Delta q}{q}.
$$
Weighted manifolds will also be an important source of motivation for this work, and are well explained in \cite{grigor2006heat}.
\end{example}

\begin{example}
If $(X_t)_{t\geq0}$ is a Markov process on a measurable space $E$, then
$$
P_t f(x) = \E(f(X_t) : X_0 = x)
$$
is a semigroup acting on some space of measurable functions. 
The two above examples are a special case of this one, corresponding to $X_t$ being a standard or weighted Brownian motion.
If $E$ is finite (or countable) then $P_t$ is called a \textit{Markov chain}.
\end{example}

These examples are all specifically \textit{Markov} semigroups: see Sections 1 and 2 of \cite{bakry2014analysis} for a formal definition and many other examples. 
The important point here is that the Laplacian on a manifold is a special example of an infinitesimal generator of a Markov semigroup.

\subsection{Markov triples}

Markov semigroups are often analysed indirectly through their associated \textit{carré du champ operators}, which are inspired by the carré du champ identity (\ref{cdc identity}).

\begin{definition}
If $L$ is the infinitesimal generator of a Markov semigroup defined on a function space $\A$, the bilinear map $\A \times \A \rightarrow \A$ given by
$$
\Gamma(f,h) = \frac{1}{2} \big( f L h + h L f - L (fh) \big)
$$
is called the \textbf{carré du champ operator}.
\end{definition}

So if $L = \Delta$ on a manifold, then $\Gamma(f,h) = g(df,dh)$. 
In the next section, we will use this property of the carré du champ $\Gamma$ to generalise manifold geometry: it will act in place of the metric where none exists.
The general setup we require is defined as follows (where a \q{good} measurable space is subject to loose but technical conditions explained in \cite{bakry2014analysis}, Section 1.1).

\begin{definition}\label{markov triple def}
A \textbf{Markov triple} $(E, \mu, \Gamma)$ consists of
\begin{enumerate}
    \item a \q{good} measurable space $(E, \mathcal{F})$ (where $E$ is a set with $\sigma$-algebra $\mathcal{F}$) with a $\sigma$-finite measure $\mu$,
    \item an algebra of real-valued bounded functions $\A$ that is dense in all the $L^p(E)$ spaces, and
    \item a symmetric bilinear map $\Gamma: \A \times \A \rightarrow \A$ (the carré du champ operator),
\end{enumerate}
such that
\begin{enumerate}
    \item[(a)] $\Gamma(f,f) \geq 0$ for all $f \in \A$,
    \item[(b)] the fundamental identity
    $$
    \int_E \Gamma(h, f^2) d\mu + 2\int_E h \Gamma(f,f) d\mu = 2 \int_E \Gamma(fh, f) d\mu
    $$
    holds for all $f,h \in \A$,
    \item[(c)] there exists an increasing sequence of functions in $\A$ converging $\mu$-almost everywhere to the constant function $1$.
\end{enumerate}
\end{definition}

Condition (a) means that $\Gamma$ is pointwise positive definite, so behaves like a metric.
We are only interested in probability spaces, so always normalise $\mu(E) = 1$.
The following canonical example justifies the rest of this work.

\begin{example}
If $\M$ is a compact manifold with Riemannian metric $g$ and induced Riemannian volume form $dx$, and $\Gamma(f,h) = g(df, dh)$ then $(\M, dx, \Gamma)$ is a Markov triple with $\A = C^\infty(\M)$. 
When we say that a Markov triple \q{is a manifold}, this is what we mean.
\end{example}

This is a special example of a very general class of Markov triples arising from Markov semigroups.

\begin{example}
Let $(X_t)_{t\geq0}$ be a Markov process on a measurable space $E$, where the Markov semigroup
$$
P_t f(x) = \E(f(X_t) : X_0 = x)
$$
acts on a space $\A$ of measurable functions satisfying the conditions in Definition \ref{markov triple def}.
Suppose that $\mu$ is an \textit{invariant measure} for $P_t$, meaning that $\int_E P_t(f) d\mu = \int_E f d\mu$ for all $f \in L^1(\mu)$, and let $\Gamma$ be the carré du champ of $P_t$.
Then $(E, \mu, \Gamma)$ is a Markov triple.
The previous example is a special case of this one, corresponding to $X_t$ being a Brownian motion on a manifold.
\end{example}

\begin{table}[h]
    \centering
    \begin{tabular}{c|c}
        \textbf{Manifolds} & \textbf{Markov triples}  \\
        \hline
        $\M$ & $E$ \\
        $dx$ & $\mu$ \\
        $C^\infty(\M)$ & $\A$ \\
        heat diffusion operator & Markov semigroup \\
        $\Delta$ & infinitessimal generator $L$ \\
        $g(df,dh)$ & $\Gamma(f,h)$ \\
    \end{tabular}
    \label{tab:my_label}
\end{table}

This class includes many classical examples (see \cite{bakry2014analysis} chapter 2).
The important idea is that Markov triples (corresponding to Markov semigroups) generalise Riemannian manifolds (corresponding to heat diffusion operators).
We can use them to define geometry on a broad class of measure spaces.
For example, a probability density function on $\R^n$ (such as in Figure \ref{fig:manifold_vs_prob}) defines a weighted heat diffusion (Example \ref{eg weighted manifold}) that defines the appropriate geometry.
We especially note the case where $E$ is finite, so $P_t$ is a Markov chain.

\begin{example}
Let $X$ be a finite set with $\sigma$-algebra given by all its subsets, and $\A$ be the space of real-valued functions on $X$ (so if $|X| = n$, we can identify $\A \cong \R^n$ with component-wise multiplication).
Let $P$ be a Markov chain on $X$ (i.e. a stochastic matrix), with stationary distribution $\mu$.
We can form a \textit{weighted graph Laplacian} from $P$ as $I - P$, and use it to define a carré du champ $\Gamma$.
Then $(X,\mu,\Gamma)$ is a Markov triple (see \cite{bakry2014analysis} Section 1.9.1).
To be explicit, if $L$ is the generator of $P_t$ and $x_i \in E$, we get the formula
$$
\Gamma(f,h)(x_i) = -\sum_{j=1}^n L(i,j)\big(f_i-f_j\big)\big(h_i-h_j\big).
$$
\end{example}

Suppose $M = (E, \mu, \Gamma)$ is a Markov triple arising from a Markov semigroup $P_t$, and $X \subset E$ is a finite sample of data.
As explained in Section \ref{section_estimation}, we can construct a finite Markov triple for $X$ that approximates $M$ by constructing a Markov chain on $X$ that approximates $P_t$.

We will consider the special class of Markov triples called \textit{diffusion} Markov triples, where the carré du champ satisfies an analogue of the Leibniz rule (when the Markov triple is a manifold, the diffusion property is equivalent to the Leibniz rule for $d$).

\begin{definition}\label{diffusion property}
Let $(E, \mu, \Gamma)$ be a Markov triple with function algebra $\A$.
Suppose $\A$ is closed under composition with smooth functions $\R^n \rightarrow \R$ that vanish at $0$, so if $f_1,...,f_n \in \A$ and $\phi:\R^n \rightarrow \R$ then $\phi(f_1,...,f_n) \in \A$.
The carré du champ operator $\Gamma$ has the \textbf{diffusion property} if
$$
\Gamma(\phi(f_1,...,f_n), h) = \sum_{i=1}^n \partial_i \phi(f_1,...,f_n) \Gamma(f_i,h)
$$
% $$
% \Gamma(f_1f_2,h) = f_1 \Gamma(f_2,h) + f_2 \Gamma(f_1,h)
% $$
for all $f_1,...,f_n,h \in \A$ and smooth $\phi:\R^n \rightarrow \R$ vanishing at 0.
% for all $f_1,f_2,h \in \A$. 
The Markov triple is then called a \textbf{diffusion Markov triple}.
\end{definition}

\section{Theory of Diffusion Geometry}\label{section_theory}

We can now construct a theory of Riemannian geometry using only the properties of diffusion Markov triples, which we will term \q{diffusion geometry}. 
The approach is essentially algebraic: we may not have the usual notions of tangent and cotangent bundles, sections, or connections, but we can instead use the various algebraic identities between them (like the carré du champ) as alternative definitions.
This section is expository, and all proofs are deferred to Appendix A (Section \ref{appendix proofs}).
We illustrate the definitions and concepts with examples of Markov triples corresponding to density functions on $\R^2$, which we approximate by finite data (this computation is explained in Section \ref{section_estimation}).
The Python code for all the examples is available at \url{https://github.com/Iolo-Jones/DiffusionGeometry}.

In this section, we will cover
\begin{itemize}
    \item[(\ref{subsec differential forms})] differential forms and the exterior algebra,
    \item[(\ref{subsec first order})] first-order calculus: vector fields and their action on functions,
    \item[(\ref{subsec second order})] second-order calculus: the Hessian (second derivative), covariant derivative, and Lie bracket,
    \item[(\ref{subsec ext derivative})] the exterior derivative and codifferential on forms,
    \item[(\ref{theory_cohomology_section})] differential algebra: differential operators, de Rham cohomology, and the Hodge Laplacian, and
    \item[(\ref{subsec third order})] third-order calculus: the second covariant derivative and curvature.
\end{itemize}

\subsection{Differential forms}
\label{subsec differential forms}

We start by defining differential forms on a diffusion Markov triple $M = (E, \mu, \Gamma)$, with function algebra $\A$.
While we do not have a cotangent bundle on Markov triples, we observe that, on a manifold $\M$,
$$
\Omega^1(\M) = \text{span}\{ f dh : f,h \in C^\infty(\M) \}.
$$
Using this, along with the fact that we expect $\Gamma(f,h) = g(df, dh)$, we make the following definition.

\begin{definition}
Define a (possibly degenerate) inner product $\inp{\cdot}{\cdot}$ on the tensor product $\A \otimes \A$ by
$$
\inp{f \otimes h}{f' \otimes h'} := \int ff' \Gamma(h, h') d\mu.
$$
Let $D^1 = \{\omega \in \A \otimes \A : \inp{\omega}{\omega} = 0\}$ be the vector subspace on which $\inp{\cdot}{\cdot}$ is degenerate. We define the \textbf{space of differential 1-forms} as the quotient space
$$
\Omega^1(M) := \frac{\A \otimes \A}{D^1}
$$
with the inner product $\inp{\cdot}{\cdot}$ and induced norm $\|\cdot\|$, and where the equivalence classes $[f \otimes h]$ are denoted $f dh$.
\end{definition}

Although this definition of forms is purely formal, we hope to recover enough of the geometry by imposing the right inner product structure on it.
In other words, we want to construct \q{manifold-free} Riemannian geometry using only the metric: even though the cotangent bundle does not exist (so we cannot define forms as its sections), the space of forms \textit{as an inner product space} is somehow more general.
We take a similar approach for the spaces of higher order $k$-forms. 
On a manifold $\M$, their metric is given by
$$
g(\alpha_1 \wedge \ddd \wedge \alpha_k, \beta_1 \wedge \ddd \wedge \beta_k) = \det(g(\alpha_i, \beta_j))
$$
which suggests the following definition.
\begin{definition}
Define a (possibly degenerate) inner product $\inp{\cdot}{\cdot}$ on $\A \otimes \bigwedge^k \A$ by
$$
\inp{f \otimes ( h_1 \wedge \ddd \wedge h_k)}{f' \otimes ( h_1' \wedge \ddd \wedge h_k')} := \int ff' \det(\Gamma(h_i, h_j')) d\mu.
$$
Let $D^k = \{\omega \in \A \otimes \bigwedge^k \A : \inp{\omega}{\omega} = 0\}$ be the vector subspace on which $\inp{\cdot}{\cdot}$ is degenerate. We define the \textbf{space of differential $k$-forms} as the quotient space
$$
\Omega^k(M) := \frac{\A \otimes \bigwedge^k \A}{D^k}
$$
with the inner product $\inp{\cdot}{\cdot}$ and induced norm $\|\cdot\|$, and where the equivalence classes $[f \otimes ( h_1 \wedge \ddd \wedge h_k)]$ are denoted $f dh_1 \wedge \ddd \wedge dh_k$.
\end{definition}

Notice that $\Omega^k(M)$ is an $\A$-module, and that $\Omega^0(M) = \A$ with the inner product from $L^2(M)$.
We can define operators between these spaces of forms by defining them on the space $\A \otimes \bigwedge^k \A$.
To check that they are well defined, we need to check that they are zero on forms of zero norm, and so descend to the quotient $\Omega^k(M)$.

\begin{figure}[!ht]
  \centering
  \captionsetup{width=0.66\linewidth}
  \subfloat[][$\alpha$]{\includegraphics[width=.48\textwidth]{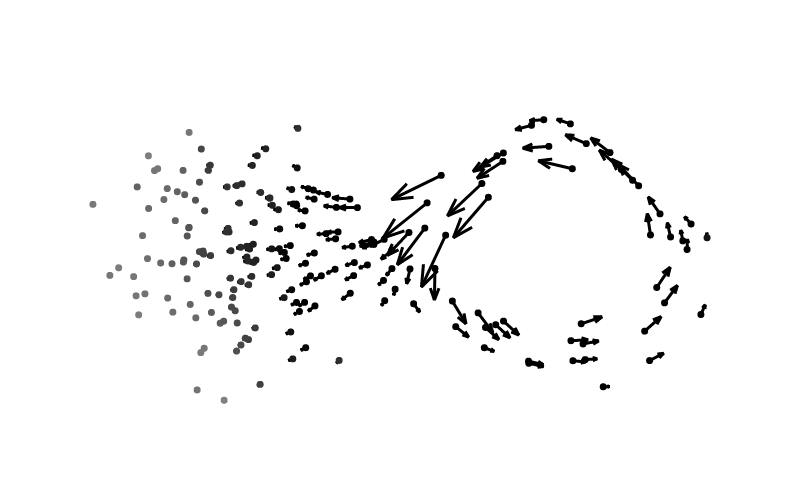}}\quad
  \subfloat[][$g(\alpha,\alpha)$]{\includegraphics[width=.48\textwidth]{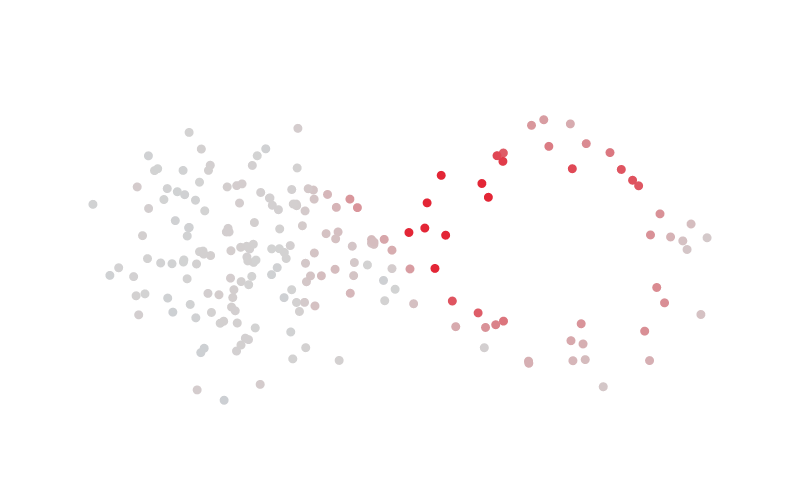}}\\
  \subfloat[][$\beta$]{\includegraphics[width=.48\textwidth]{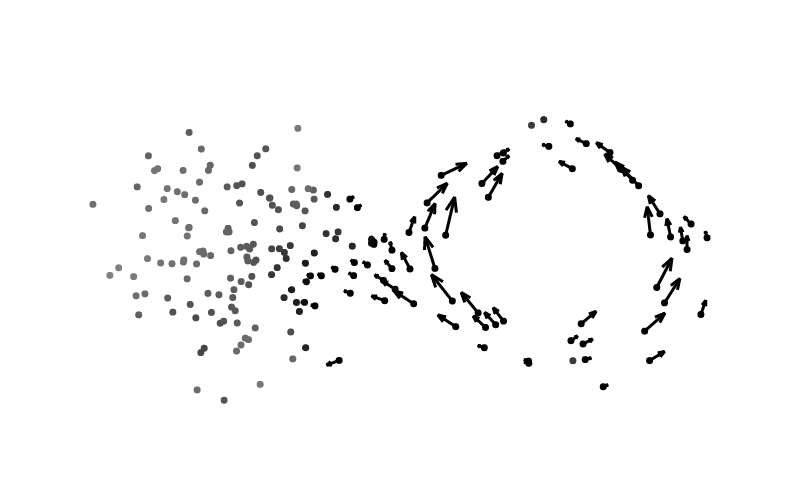}}\quad
  \subfloat[][$g(\alpha,\beta)$]{\includegraphics[width=.48\textwidth]{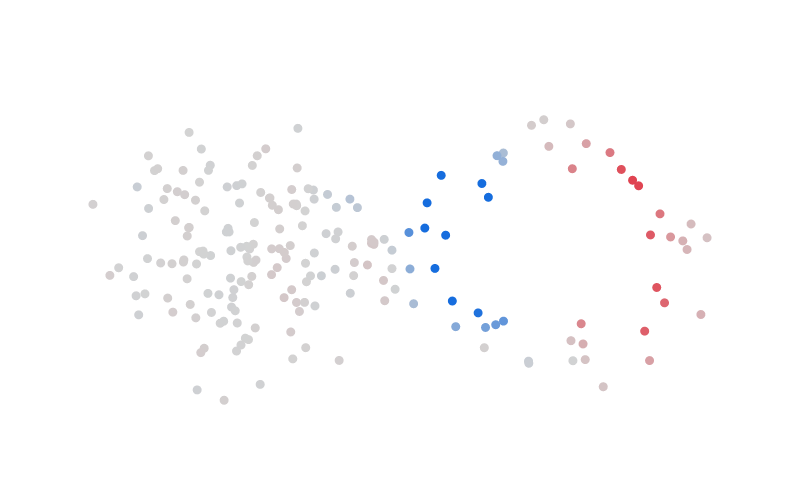}}\\
  \subfloat[][$\gamma$]{\includegraphics[width=.48\textwidth]{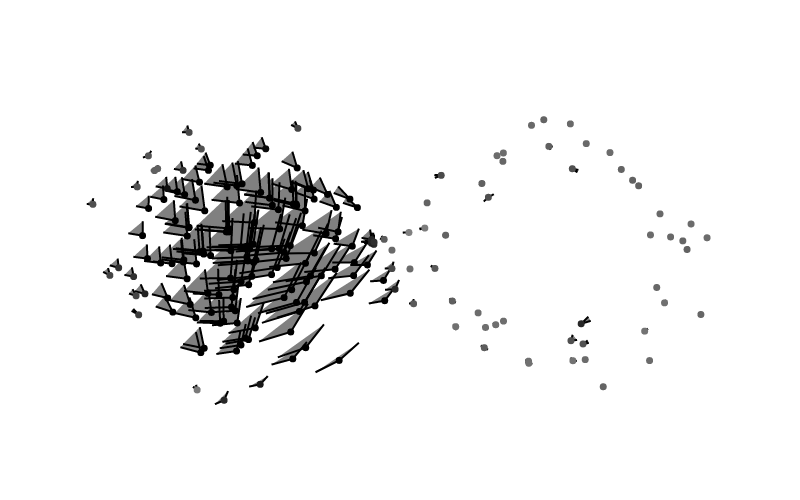}}\quad
  \subfloat[][$g(\gamma,\gamma)$]{\includegraphics[width=.48\textwidth]{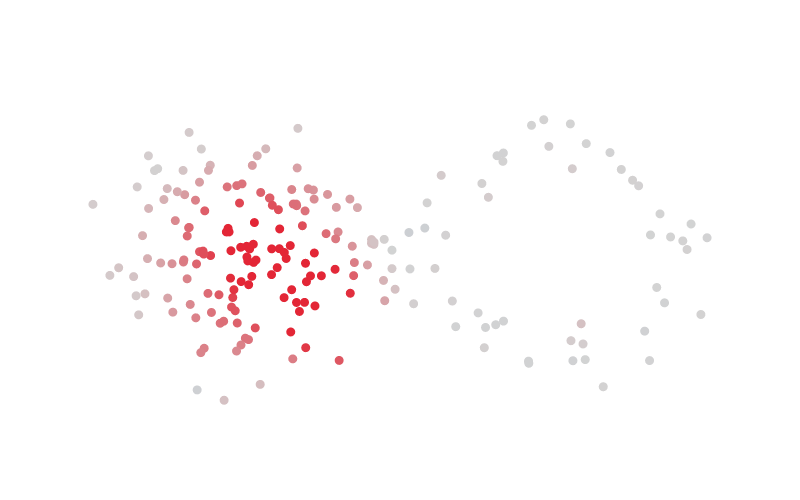}}
  \caption{\textbf{Forms and the metric.}
  The 1-forms $\alpha$ and $\beta$ are represented by their dual vector fields (a, c) (see Remark \ref{visualisation remark}), and the 2-form $\gamma$ is (non-uniquely) represented by its dual bi-vector field (e). 
  $g(\alpha,\alpha)$ and $g(\gamma,\gamma)$ measure the supports of $\alpha$ and $\gamma$ (b, f), and $g(\alpha,\beta)$ measures the alignment of $\alpha$ and $\beta$ (d). 
  Negative values are blue, and positive values are madder rose.
  }
  \label{fig:f2}
\end{figure}

We visualise forms in Figure \ref{fig:f2} using their dual vector fields (explained in Remark \ref{visualisation remark}), as well as the metric which we now define.

\begin{definition}
The \textbf{metric} is the bilinear function $g:\Omega^k(M) \times \Omega^k(M) \rightarrow \A$ given by
$$
g(f dh_1 \wedge \ddd \wedge dh_k, f' dh_1' \wedge \ddd \wedge dh_k') := ff' \det(\Gamma(h_i, h_j')).
$$
\end{definition}

Note that $g(\alpha, \alpha') \in \A$ because the determinant is a polynomial (and $\A$ is closed under multiplication), and the tensor product $\A \otimes \bigwedge^k \A$ comprises \textit{finite} linear combinations of the irreducible elements on which $g$ is defined, and so $g(\alpha, \alpha')$ is just a finite sum of functions in $\A$.
We then see that
$$
\inp{\alpha}{\beta} = \int g(\alpha, \beta) d\mu < \infty
$$
for all $\alpha, \beta \in \Omega^k(M)$, as functions in $\A$ have finite integral, so the inner product is finite.
We now verify that the metric satisfies all the properties we would like, so the notion of differential forms introduced above is well defined.

\begin{restatable}{reprop}{positivemetric}
\label{positive_metric}
The metric $g$ on $\Omega^k(M)$ is symmetric, bilinear, and positive semi-definite. In particular, it satisfies the Cauchy-Schwarz inequality pointwise, and the inner product $\inp{\cdot}{\cdot}$ and metric $g$ on $\Omega^k(M)$ are well defined.
\end{restatable}

% \footnote{The representation here is a crude one because $v\wedge w = (-w) \wedge v = w \wedge (-v) = (-v) \wedge (-w)$ so the orientation of the triangle we choose to plot is arbitrary.}

We abbreviate $1 dh_1 \wedge \ddd \wedge dh_k$ to $dh_1 \wedge \ddd \wedge dh_k$.
Using these definitions for differential forms, we can collect other objects from Riemannian geometry for our theory (see Figure \ref{fig:f3}).

\begin{definition}
The \textbf{wedge product} is the bilinear map $\wedge : \Omega^k(M) \times \Omega^l(M) \rightarrow \Omega^{k+l}(M)$ defined by
$$
(f dh_1 \wedge \ddd \wedge dh_k, f' dh_1' \wedge \ddd \wedge dh_l') \mapsto ff' dh_1 \wedge \ddd \wedge dh_k \wedge dh_1' \wedge \ddd \wedge dh_l'.
$$
\end{definition}

The wedge product induces an isometric isomorphism $\Omega^k(M) \cong \bigwedge^k_\A\Omega^1(M)$ as an $\A$-tensor product of $\A$-modules.

\begin{restatable}{reprop}{wedgewelldefined}
\label{wedge_well_defined}
If $\alpha \in \Omega^k(M)$ and $\beta \in \Omega^l(M)$ then $g(\alpha\wedge\beta, \alpha\wedge\beta) \leq g(\alpha,\alpha)g(\beta,\beta)$. In particular, the wedge product is well defined on $\Omega^k(M) \times \Omega^l(M)$, and is a bounded linear operator in each argument.
\end{restatable}

See Figures \ref{fig:f3} and \ref{fig:cohom product} for examples of the wedge product.
We also define the exterior derivative, but for now only on functions\footnote{The derivative of $k$-forms is given in Subsection \ref{subsec ext derivative}.}.

\begin{definition}
The \textbf{exterior derivative} is the linear map $d_0 : \A \rightarrow \Omega^{1}(M)$ defined by
$$
f \mapsto df,
$$
i.e. $f$ maps to the equivalence class $[1 \otimes f]$.
\end{definition}

We will usually suppress the subscript in $d_0$ and just write $d$ where there is no ambiguity.
We immediately see that these definitions agree with the notation, for example
$$
f dh_1 \wedge dh_2 = f \wedge \big( d_0(h_1) \wedge d_0(h_2) \big).
$$

\begin{figure}[!ht]
  \centering
  \captionsetup{width=0.66\linewidth}
  \subfloat[][$\alpha$]{\includegraphics[width=.48\textwidth]{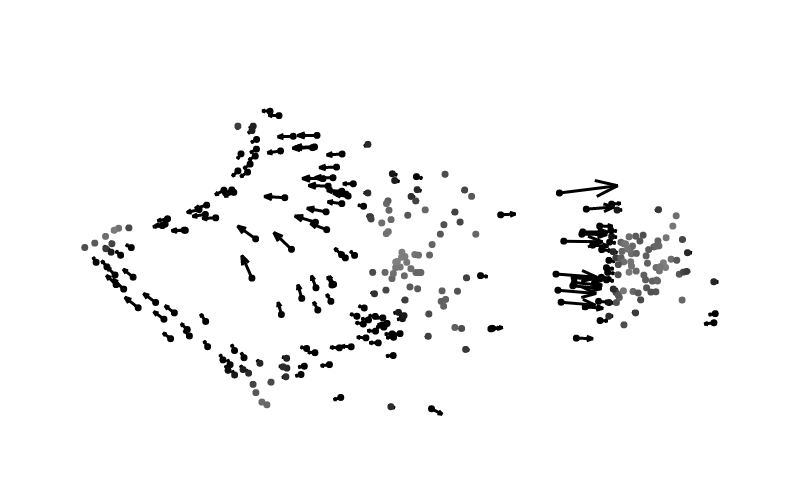}}\quad
  \subfloat[][$f$]{\includegraphics[width=.48\textwidth]{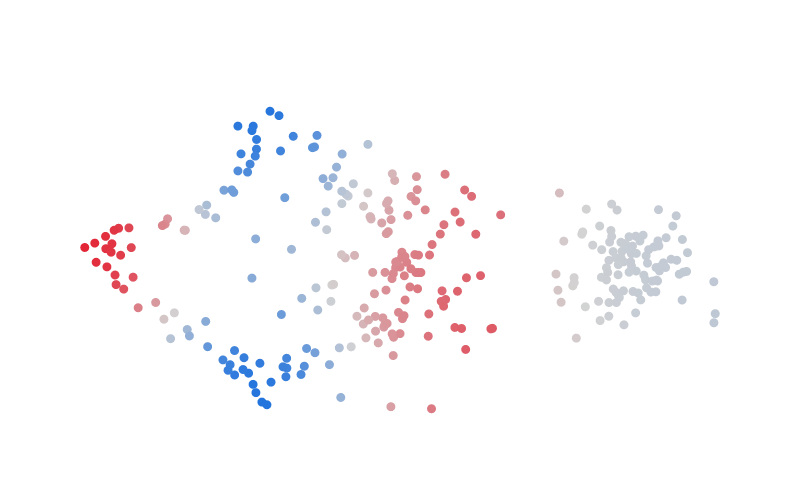}}\\
  \subfloat[][$f \alpha$]{\includegraphics[width=.48\textwidth]{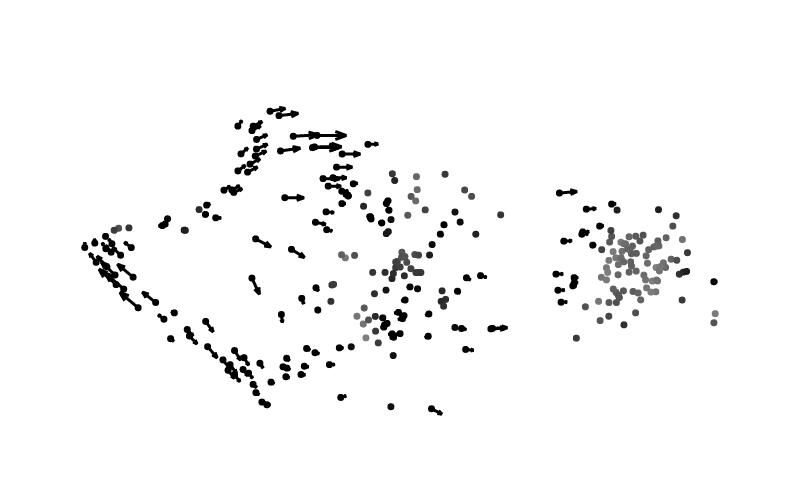}}\quad
  \subfloat[][$df$]{\includegraphics[width=.48\textwidth]{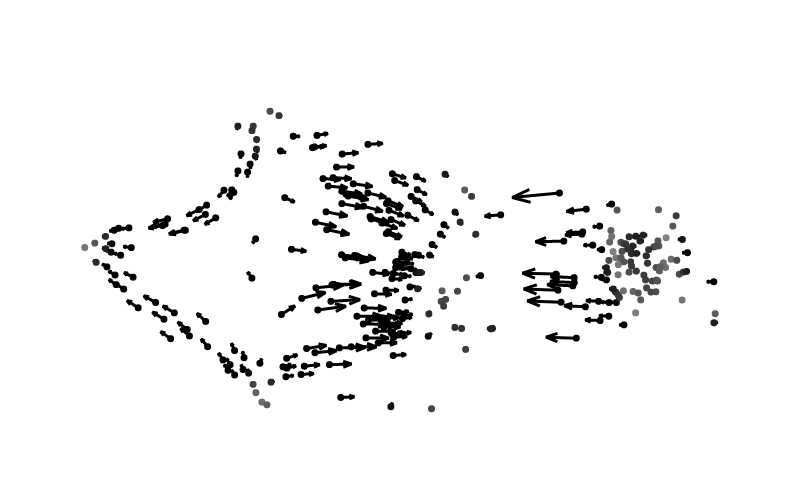}}
  \caption{\textbf{Wedge product and exterior derivative.}
  We can multiply a 1-form $\alpha$ (a) by a function $f$ (b).
  The product $f\alpha$ (c) restricts $\alpha$ to the support of $f$ (the left \q{component} of the data), and the negative values of $f$ (coloured blue) reverse the orientation of $\alpha$. $df$, visualised as its dual vector field $\nabla f$ (d), shows the direction of steepest increase of $f$.
  See Figure \ref{fig:cohom product} for a wedge product of two 1-forms.
  }
  \label{fig:f3}
\end{figure}

The fact that $M$ is a \textit{diffusion} Markov triple means the Leibniz rule holds for $d$, as an example of the chain rule.

\begin{restatable}[Calculus rules for $d_0$]{reprop}{calculusrulesd}
\label{calculus_rules_d0}
The following properties hold for $d = d_0$:
\begin{itemize}
    % \item Locality: if $f,g \in \A$ then $df = dg$ almost everywhere in $\{f=g\}$.
    \item Chain rule: if $f_1,...,f_k \in \A$ and $\phi:\R^k \rightarrow \R$ is smooth then
    $$
    d\big(\phi(f_1,...,f_k)\big) = \sum_{i=1}^k \partial_i\phi(f_1,...,f_k) df_i
    $$
    \item Leibniz rule: if $f,h \in \A$ then $d(fh) = fd(h) + hd(f)$.
\end{itemize}
\end{restatable}

\subsection{First-order calculus: vector fields and duality}
\label{subsec first order}

We now construct the first-order calculus in diffusion geometry, including vector fields and their duality with 1-forms.
We first consider the larger set of \textit{derivations}, which define a sort of directional derivative of functions.

\begin{definition}\label{derivation def}
A \textbf{derivation} is a linear map $X:\A \rightarrow \A$ which satisfies $X(fh) = fX(h) + hX(f)$ (which we call the Leibniz rule).
\end{definition}

Contraction with the metric $g$ gives a natural pairing between forms and derivations. 
If $\alpha \in \Omega^1(M)$ is a 1-form, then $\alpha^\sharp(f) := g(\alpha,df)$ defines a derivation (where the Leibniz rule for $\alpha^\sharp$ is inherited from $d$).
We define the \textit{vector fields} as the derivations in the image of this map (see Figure \ref{fig:f5}).

\begin{definition}
A \textbf{vector field} is a derivation $X$ that satisfies $X = \alpha^\sharp$ for some $\alpha \in \Omega^1(M)$.
$X$ is then called the \textbf{dual vector field} of $\alpha$.
We denote the space of vector fields by $\mathfrak{X}(M)$.
Each vector field $X$ has a \textbf{dual one form} $X^\flat$ defined by $g(X^\flat,df) = X(f)$
for all $f\in\A$. 
We extend the \textbf{metric} to vector fields as $g(X,Y) = g(X^\flat, Y^\flat)$.
\end{definition}

This definition of $X^\flat$ specifies it uniquely because the terms $fdh$ span $\Omega^1(M)$, and we can evaluate $\inp{X^\flat}{fdh} = \int fg(X^\flat,dh) d\mu$. 
It is clear to see that the map $\alpha \mapsto \alpha^\sharp$ is an isometric isomorphism between $\Omega^1(M)$ and $\mathfrak{X}(M)$, with inverse $X \mapsto X^\flat$ (both called \textit{musical isomorphisms}).
We can use this duality to define the usual action of 1-forms on vector fields.

\begin{definition}
A 1-form $\alpha$ defines an $\A$-linear map $\mathfrak{X}(M) \rightarrow \A$ by $\alpha(X) = g(\alpha, X^\flat)$.
\end{definition}

\begin{figure}[!h]
  \centering
  \captionsetup{width=0.66\linewidth}
  \subfloat[][$X$]{\includegraphics[width=.3\textwidth]{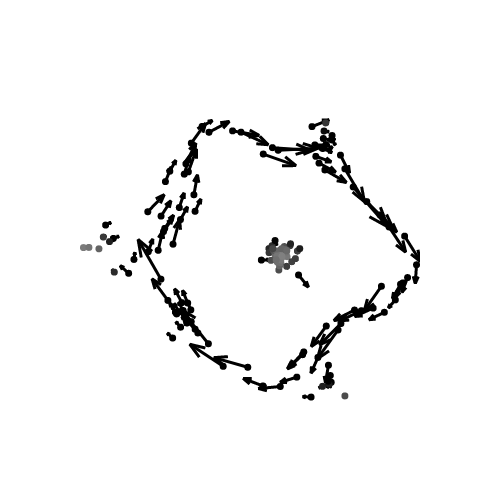}}
  \subfloat[][$f$]{\includegraphics[width=.3\textwidth]{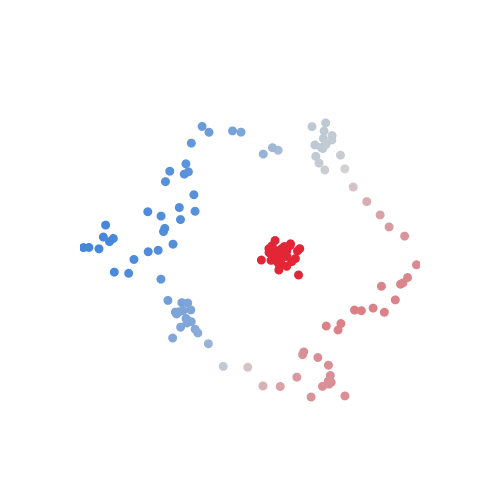}}
  \subfloat[][$X(f)$]{\includegraphics[width=.3\textwidth]{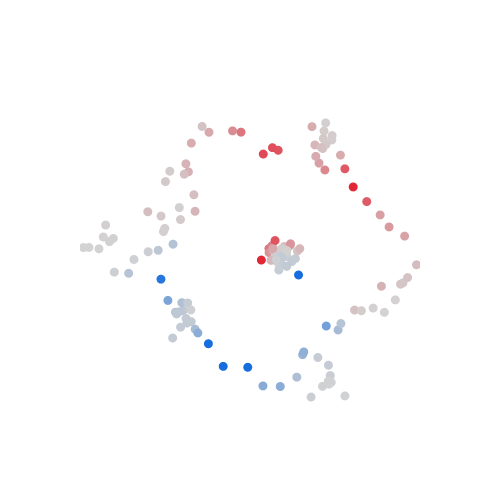}}\\
  \subfloat[][$h$]{\includegraphics[width=.3\textwidth]{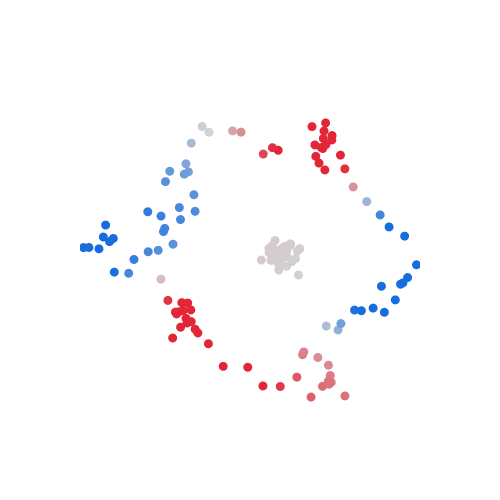}}
  \subfloat[][$X(h)$]{\includegraphics[width=.3\textwidth]{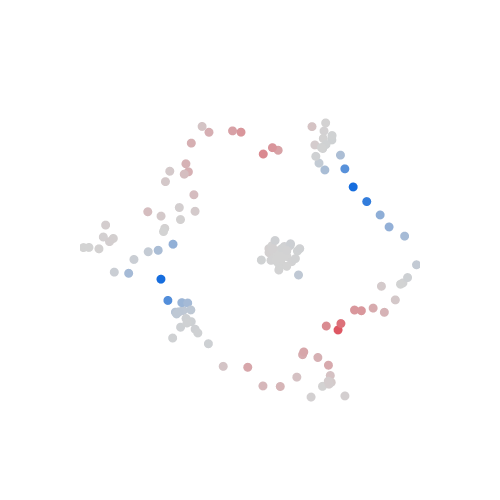}}
  % \subfloat[][$\ker(Id - X)$]{\includegraphics[width=.3\textwidth]{images/paper_images/f5/fig6.png}}\\  
  \caption{\textbf{Vector fields act on functions.}
  The vector field $X$ (a) acts on $f$ and $h$ (b, d) by differentiation clockwise (c, e). 
  % We can find functions that are fixed by $X$ by computing the spectrum of $Id - X$ and finding the zero eigenvalues (whose eigenfunctions are in $\ker(Id - X)$).
  }
  \label{fig:f5}
\end{figure}

\begin{remark}\label{visualisation remark}
We can visualise vector fields by their action on the coordinate functions of data.
In all the examples in this section, the data are drawn from probability distributions on $\R^2$, and so have two coordinate functions $x$ and $y$.
We can calculate the \q{direction} of a vector field $X$ at a point $p$, with position $\big(x(p),y(p)\big)$, as the vector $\big(X(x)(p),X(y)(p)\big)$.
We visualise 1-forms by their dual vector fields, i.e. $(fdh)^\sharp(x) = f\Gamma(h,x)$.
\end{remark}

We can also define the gradient operator by duality.

\begin{definition}
The \textbf{gradient} $\nabla:\A \rightarrow \mathfrak{X}(M)$ is the dual operator to the exterior derivative $d_0$, defined by $\nabla(f) = (df)^\sharp$.
\end{definition}

Given this definition of gradient, we see that we could offer precisely the same definition for $\mathfrak{X}(M)$ as of $\Omega^1(M)$, except with $d$ exchanged for $\nabla$ (hence the musical isomorphism).
The duality of forms and vector fields lets us define the interior product, which satisfies the usual properties.

\begin{definition}[Interior product]
If $X \in \mathfrak{X}(M)$, we define the \textbf{interior product} $i_X : \Omega^k(M) \rightarrow \Omega^{k-1}(M)$ inductively by
\begin{enumerate}
    \item if $\alpha \in \Omega^0(M)$, $i_X(\alpha) := 0$
    \item if $\alpha \in \Omega^1(M)$, $i_X(\alpha) := \alpha(X)$
    \item if $\alpha \in \Omega^k(M)$, $i_X(\alpha \wedge \beta) := i_X(\alpha) \wedge \beta + (-1)^k \alpha \wedge i_X(\beta)$.
\end{enumerate}
\end{definition}

\begin{restatable}{reprop}{interiorproductprop}
\label{interior product prop}
$i_X(\alpha)$ a well defined bilinear operator in $X$ and $\alpha$, and satisfies $i_Y i_X = -i_X i_Y$.
\end{restatable}

We are focusing, for simplicity, on a function space $\A$ which satisfies several nice properties.
The next section will require more delicate functional analysis, so we define the following \textit{complete} spaces of forms and vector fields.

\begin{definition}
We define $L^2\Omega^1(M)$ and $L^2\mathfrak{X}(M)$ to be the completions of $\Omega^1(M)$ and $\mathfrak{X}(M)$ respectively.
\end{definition}

\subsection{Second-order calculus: Hessian, covariant derivative, and Lie bracket}
\label{subsec second order}

We can now introduce the second-order objects from Riemannian geometry, and start by introducing the necessary tensor products of forms and vector fields.

\begin{definition}
We define the \textbf{space of (0,2)-tensors}
$$\Omega^1(M)^{\otimes2} := \Omega^1(M) \otimes_\A \Omega^1(M),$$
so elements of $\Omega^1(M)^{\otimes2}$ are finite sums $\sum_i h_i da_i \otimes db_i$. 
We extend the metric to these tensors by
$$
g(h_1 da_1 \otimes db_1, h_2 da_2 \otimes db_2) := h_1h_2 g(a_1,a_2) g(b_1,b_2).
$$
Equivalently, $g(\alpha_1\otimes\beta_1, \alpha_2\otimes\beta_2) = g(\alpha_1,\alpha_2) g(\beta_1,\beta_2)$.
The pointwise norm $\sqrt{g(A,A)}$ is known as the \textbf{(pointwise) Hilbert-Schmidt norm} of the tensor and is denoted $|A|_{HS}$.
This metric induces a norm by integration, and the \textbf{complete space of (0,2)-tensors}, denoted $L^2(\Omega^1(M)^{\otimes2})$, is the completion of $\Omega^1(M)^{\otimes2}$.
\end{definition}

\begin{definition}
We define the \textbf{space of (2,0)-tensors}
$$\mathfrak{X}(M)^{\otimes2} := \mathfrak{X}(M) \otimes_\A \mathfrak{X}(M),$$
so elements of $\mathfrak{X}(M)^{\otimes2}$ are finite sums $\sum_i h_i \nabla a_i \otimes \nabla b_i$. 
We extend the metric analogously to the (0,2)-tensors, and define the corresponding norm.
The \textbf{complete space of (2,0)-tensors}, denoted $L^2(\mathfrak{X}(M)^{\otimes2})$, is the completion of $\mathfrak{X}(M)^{\otimes2}$.
We obtain the usual musical isomorphisms $\sharp$ and $\flat$ by
$$
(f dh_1 \otimes dh_2)^\sharp = f \nabla h_1 \otimes \nabla h_2
\qquad
(f \nabla h_1 \otimes \nabla h_2)^\flat = f dh_1 \otimes dh_2,
$$
which are isometric isomorphisms $\Omega^1(M)^{\otimes2} \cong \mathfrak{X}(M)^{\otimes2}$ and $L^2(\Omega^1(M)^{\otimes2}) \cong L^2(\mathfrak{X}(M)^{\otimes2})$.
There is an action of $(0,2)$-tensors on $(2,0)$-tensors: if $A \in \Omega^1(M)^{\otimes2}$ or $A \in L^2(\Omega^1(M)^{\otimes2})$ and $\xi \in \mathfrak{X}(M)^{\otimes2}$ then 
$$
A(\xi) := g(A, \xi^\flat).
$$
If $A \in \Omega^1(M)^{\otimes2}$ or $A \in L^2(\Omega^1(M)^{\otimes2})$, we also define
$$
A(X,Y) := A(X\otimes Y)
$$
for $X,Y \in \mathfrak{X}(M)$.
\end{definition}

\subsubsection{Hessian}

We can use this tensor technology to define the second-order geometric objects, starting with the Hessian.
For functions $f: \R^n \rightarrow \R$, the Hessian is the matrix $\big(\frac{\partial^2 f}{\partial x_i \partial x_j}\big)$, and so generalises the \textit{second derivative} of a function (see Figure \ref{fig:f7}).
On manifolds, we get the same expression in geodesic normal coordinates.
The formula we use here as a definition is derived from standard properties of the manifold Hessian in Appendix A, Proposition \ref{manifold hessian}.

\begin{definition}\label{hessian def}
We define the \textbf{Sobolev space} $W^{2,2}(M) \subseteq \A$ as the set of $f \in \A$ for which there exists an $A \in L^2(\Omega^1(M)^{\otimes2})$ satisfying
$$
A(\nabla a,\nabla b) = \frac{1}{2}\big( \Gamma(a, \Gamma(f,b)) + \Gamma(b, \Gamma(f,a)) - \Gamma(f, \Gamma(a,b)) \big).
$$
for all $a,b\in\A$. This tensor, called the \textbf{Hessian}, is uniquely determined and is denoted $H(f)$.
\end{definition}

\begin{figure}[!ht]
  \centering
  \captionsetup{width=0.66\linewidth}
  \subfloat[][$f$]{\includegraphics[width=.3\textwidth]{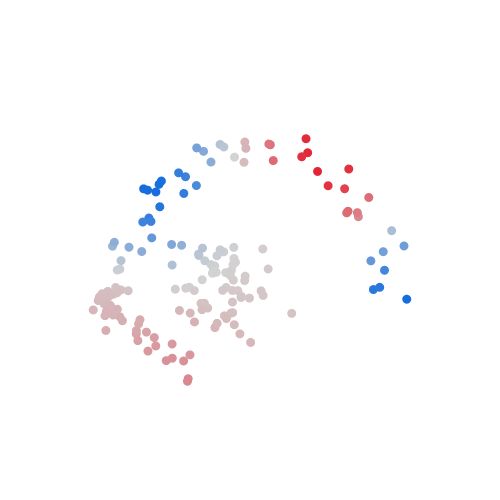}}
  \subfloat[][$X$]{\includegraphics[width=.3\textwidth]{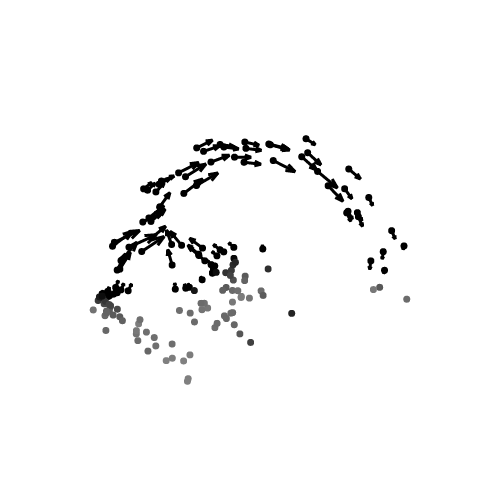}}
  \subfloat[][$Y$]{\includegraphics[width=.3\textwidth]{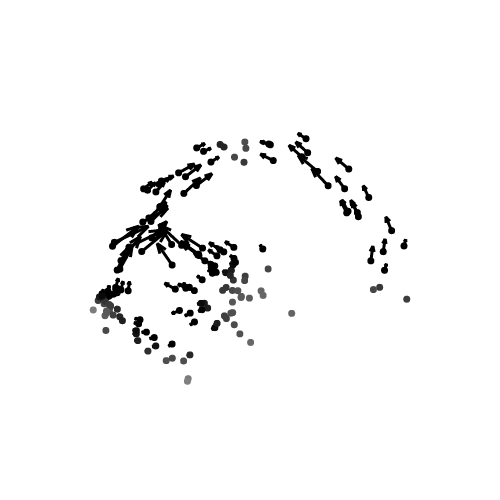}}\\
  \subfloat[][$g(df,df) = \Gamma(f,f)$]{\includegraphics[width=.3\textwidth]{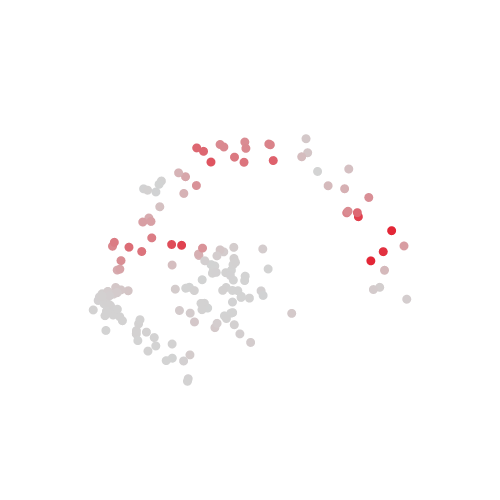}}
  \subfloat[][$H(f)(X,X)$]{\includegraphics[width=.3\textwidth]{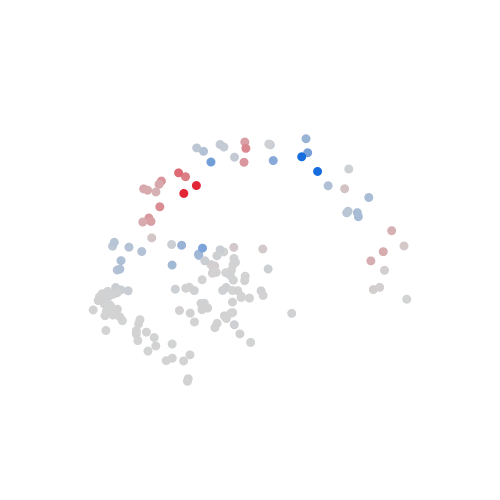}}
  \subfloat[][$H(f)(X,Y)$]{\includegraphics[width=.3\textwidth]{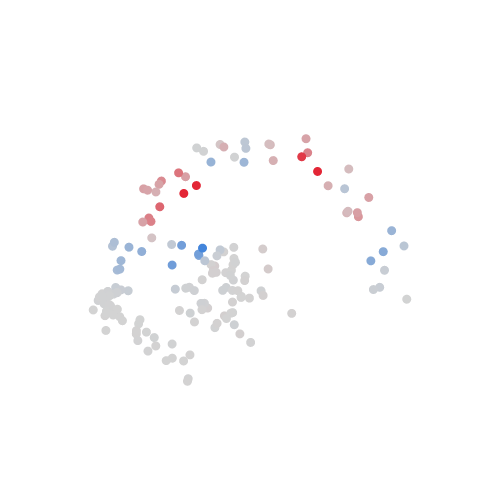}}\\
  \caption{\textbf{The Hessian (second derivative of functions).}
  A function $f$ has a maximum at the top right of the data and a minimum at the top left (a), which correspond to points where its derivative $df$ is zero (d). 
  Given a vector field $X$, which is supported near these local extrema (b), we can compute the Hessian $H(f)(X,X)$ to classify the extrema \textit{relative to the direction of $X$} (e). 
  The Hessian is negative at the maximum and positive at the minimum. 
  If we use a different vector field $Y$ (c) as the second argument, the Hessian changes sign in the areas where $X$ and $Y$ are unaligned (f).}
  \label{fig:f7}
\end{figure}

It is clear to see that the Hessian is a trilinear operator in $f$, $a$, and $b$, and is symmetric in $a$ and $b$.
The formula above dually defines $H(f)$ as an operator $\mathfrak{X}(M)^{\otimes 2}\rightarrow \A$ for any $f \in \A$, but this is not necessarily a \textit{bounded} operator. 
If it were bounded, the Reisz representation theorem would guarantee the existence of a $H(f) \in L^2(\Omega^1(M)^{\otimes2})$. 
Instead, we have to make do with $W^{2,2}(M)$, which may be a proper subset of $\A$.
In fact, in the most general setting, we cannot guarantee which, if any, functions from $\A$ are in $W^{2,2}(M)$.
Although, in the special case of RCD spaces\footnote{RCD spaces place a lower bound on the \q{Ricci curvature} of a metric measure space in a synthetic sense. A space being $RCD(K, \infty)$ is equivalent to the Ricci curvature being bounded below by $K$, and imposes strong regularity properties on the heat flow.}, the following result follows from a theorem of Gigli.

\begin{theorem}[Follows from 6.2.22. in $\protect\cite{gigli2020lectures}$]
\label{W22 A0 equivalence gigli}
If $M$ is an RCD$(K,\infty)$ space then $\A = W^{2,2}(M)$.
\end{theorem}

The Sobolev space $W^{2,2}(M)$ contains the functions whose Hessian exists as an $L^2(\Omega^1(M)^{\otimes2})$ tensor.
In the following, we would like to work with Hessians that are actually in $\Omega^1(M)^{\otimes2}$, and so define the following space of \q{twice differentiable} functions.

\begin{definition}
We define the \textbf{$\infty$-Sobolev space}
$$
W^{2,\infty}(M) := \{ f \in W^{2,2} : H(f) \in \Omega^1(M)^{\otimes2} \}.
$$
\end{definition}

In other words, $W^{2,\infty}(M)$ contains the functions whose Hessian exists and can be written as a \textit{finite} sum
$$
H(f) = \sum_i h_i da_i \otimes db_i,
$$
unlike the other functions in $W^{2,2}(M)$ whose Hessian may be a limit of such sums.

\begin{remark}
The fact that the formula for the Hessian defines $H(f)$ as an operator $\mathfrak{X}(M)^{\otimes 2}\rightarrow \A$ for any $f \in \A$, even if $H(f)$ is not a well-defined tensor in $L^2(\Omega^1(M)^{\otimes2})$, represents a strong theme in this work.
For many of the objects defined below, we can construct a more general \q{weak} version that acts on functions, but can only define a tensorial version on a delicately specified subspace.
\end{remark}

The other second-order differential objects can be built from the Hessian, and so, by considering just the functions from $W^{2,\infty}(M)$, we can guarantee that these other objects are well defined.
Specifically, we can now define a notion of \q{differentiable} vector fields.

\begin{definition}
We define the space of \textbf{differentiable vector fields}
$$
\mathfrak{X}^1(M) := \left\{\sum f_i \nabla h_i \in \mathfrak{X}(M) : h_i \in W^{2,\infty}(M) \text{ for all } i \right\}.
$$
\end{definition}

These vector fields are \q{differentiable} because the functions $h_i$ have already been differentiated once by $\nabla$, and so we also need their second derivative $H(h_i)$ to be a well behaved tensor.
We can derive the usual calculus rules for $H$.

\begin{restatable}{reprop}{hessianleibnizrule}
\label{hessian leibniz rule}
The Hessian satisfies the (second order) Leibniz rule
$$
H(fh) = f H(h) + h H(f) + df \otimes dh + dh \otimes df
$$
for all $f,h \in W^{2,2}(M)$.
In particular, $W^{2,2}(M)$ and $W^{2,\infty}(M)$ are closed under multiplication (and so are subalgebras of $\A$).
\end{restatable}

\begin{restatable}{reprop}{hessianproductrule}
\label{hessian product rule}
The Hessian satisfies the product rule for the carré du champ
$$
d \big( \Gamma(f_1, f_2) \big) = H(f_1)(\nabla f_2,\cdot) + H(f_2)(\nabla f_1,\cdot)
$$
for all $f_1,f_2 \in W^{2,2}(M)$.
\end{restatable}

\subsubsection{Covariant derivative}

We now define the covariant derivative, which is usually defined as a $(1,1)$ tensor field. 
Here, we instead define it as a $(0,2)$ field to make use of the space $\mathfrak{X}^1(M)$.
The formula we use comes from the fact that, on a manifold,
$$
\nabla_Y (f \nabla h) = df(Y) \nabla h + f \nabla_Y (\nabla h)
$$
and $g(\nabla_Y (\nabla h), Z) = H(h) (Y,Z)$ by definition, so
$$
g(\nabla_Y (f \nabla h), Z) = df(Y) dh(Z) + f H(h) (Y,Z).
$$

\begin{definition}\label{(0,2) nabla}
If $X = \sum f_i \nabla h_i \in \mathfrak{X}^1(M)$, we define $\nabla X \in \Omega^1(M)^{\otimes2}$ by
$$
\nabla X := \sum \big( d f_i \otimes d h_i + f_i H(h_i) \big),
$$
which we call the \textbf{covariant derivative} of $X$.
\end{definition}

We need $X \in \mathfrak{X}^1(M)$ to ensure that $H(h_i) \in \Omega^1(M)^{\otimes2}$ for all $i$.

\begin{restatable}{reprop}{covariantwelldefined}
\label{covariant well defined}
The covariant derivative $\nabla$ is a well defined map $\mathfrak{X}^1(M) \rightarrow \Omega^1(M)^{\otimes2}$.
\end{restatable}

We saw earlier that $H(f)$ is well defined as an operator $\mathfrak{X}(M)^{\otimes 2}\rightarrow \A$ for all $f \in \A$, even if that operator is not a tensor in $L^2(\Omega^1(M)^{\otimes2})$.
By extension, the formula for $\nabla X$ above defines a \q{weak} derivative $\nabla X: \mathfrak{X}(M)^{\otimes 2} \rightarrow \A$ by
$$
\nabla X (Y,Z) = \sum \big( Y(f_i) Z(h_i) + f_i H(h_i)(Y,Z) \big)
$$
for all $X = \sum f_i \nabla h_i \in \mathfrak{X}(M)$, although we cannot guarantee that this operator defines a tensor in $L^2(\Omega^1(M)^{\otimes2})$ unless $X \in \mathfrak{X}^1(M)$.
We can now define the $(1,1)$ tensor field version of the covariant derivative (see Figure \ref{fig:f6} (d)).

\begin{definition}\label{(1,1) nabla}
If $X = \sum f_i \nabla h_i \in \mathfrak{X}^1(M)$, we define $\nabla_Y X \in \mathfrak{X}(M)$ by
$$
\nabla_Y X := \sum \big( Y (f_i) \nabla h_i + f_i [H(h_i)(Y, \cdot)]^\sharp \big)
$$
for all $Y \in \mathfrak{X}(M)$.
If $f\in\A$, we also define $\nabla_Y(f) = Y(f)$.
\end{definition}

Notice that $g(\nabla_Y X, Z) = \nabla X (Y,Z)$ for all $Z \in \mathfrak{X}(M)$, so $X \mapsto \nabla_Y X$ is well defined by Proposition \ref{covariant well defined}.
As with the Hessian and $\nabla$, we can define a more general \q{weak} derivative $\nabla_Y X$ as a derivation $\A \rightarrow \A$ by
$$
\nabla_Y X (f) = g(\nabla_Y X, \nabla f) = \nabla X (Y,\nabla f)
$$
for all $X \in \mathfrak{X}(M)$, although this may not define a vector field $\nabla_Y X \in \mathfrak{X}(M)$ unless $X \in \mathfrak{X}^1(M)$.
We check that this definition for $\nabla$ satisfies the standard properties.

\begin{restatable}{reprop}{nablaaffineconnection}
\label{nabla affine connection}
The covariant derivative $\nabla$ is an affine connection, meaning $\nabla X(Y,Z)$ is $\A$-linear in $Y$ and $Z$, linear in $X$, and satisfies the Leibniz rule
$$
\nabla(fX) = df \otimes X^\flat + f\nabla X.
$$
In particular, $\nabla_Y X$ is $\A$-linear in $Y$, and $\nabla_Y (fX) = Y(f) X + f \nabla_Y X$.
\end{restatable}

\begin{restatable}{reprop}{nablapreservesmetric}
\label{nabla preserves metric}
The covariant derivative $\nabla$ is compatible with the metric, meaning 
$$
\nabla_X g(Y,Z) = g(\nabla_X Y, Z) + g(Y, \nabla_X Z).
$$
\end{restatable}
\subsubsection{Lie bracket}

Another notion of \q{differentiating a vector field} comes from the Lie bracket (see Figure \ref{fig:f6} (c)).
The final property needed for the covariant derivative to be a Levi-Civita connection (along with being an affine connection: see Proposition \ref{nabla affine connection}) is that it is also \textit{torsion-free}, a property involving the Lie bracket.
Here, we will take this property as a definition.

\begin{definition}\label{Lie bracket def}
If $X,Y \in \mathfrak{X}^1(M)$ then we define their \textbf{Lie bracket} $[X,Y] \in \mathfrak{X}(M)$ by
$$
[X,Y] := \nabla_X Y - \nabla_Y X.
$$
\end{definition}

The Lie bracket is usually defined as the commutator of $X$ and $Y$ as operators $XY - YX$, although we cannot always guarantee\footnote{The commutator $XY - YX$ is always a derivation, but that is not enough to be a vector field.} that $XY - YX \in \mathfrak{X}(M)$.
However, on $\mathfrak{X}^1(M)$ where the Lie bracket is a well-defined vector field, it does agree with the commutator.

\begin{restatable}{reprop}{liebracketcommutator}
\label{lie bracket commutator}
We have $[X,Y] = XY - YX$ for all $X,Y \in \mathfrak{X}^1(M)$.
\end{restatable}

\begin{figure}[!ht]
  \centering
  \captionsetup{width=0.66\linewidth}
  \subfloat[][$X$]{\includegraphics[width=.48\textwidth]{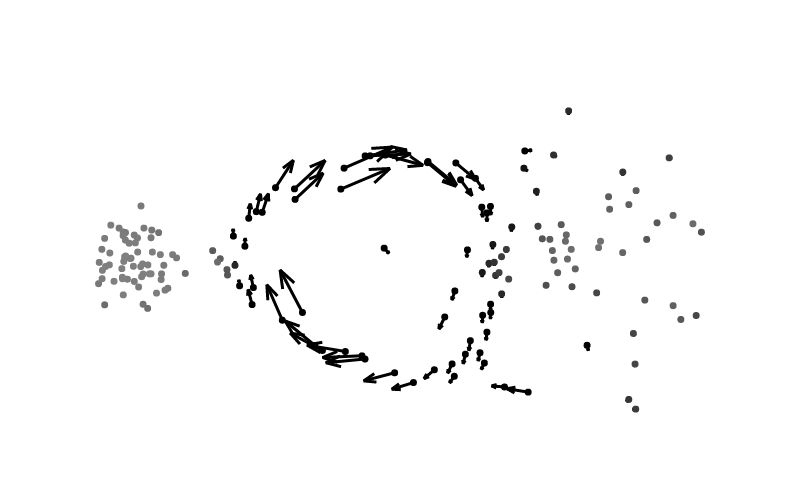}}\quad
  \subfloat[][$Y$]{\includegraphics[width=.48\textwidth]{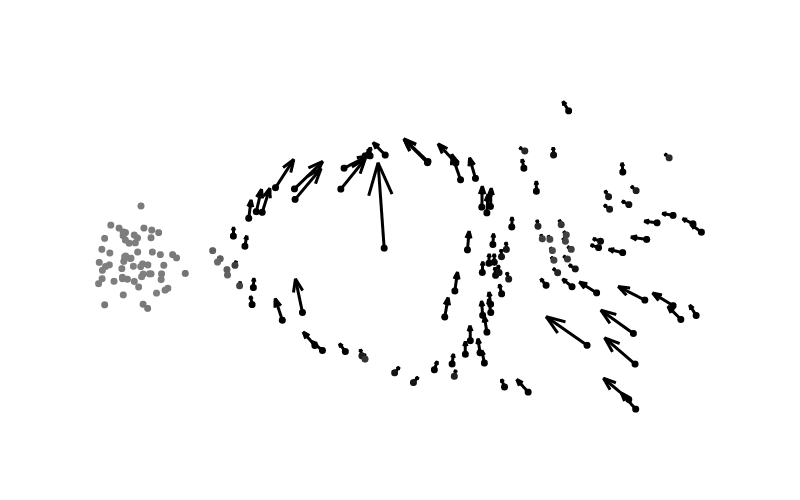}}\\
  \subfloat[][${[X,Y]}$]{\includegraphics[width=.48\textwidth]{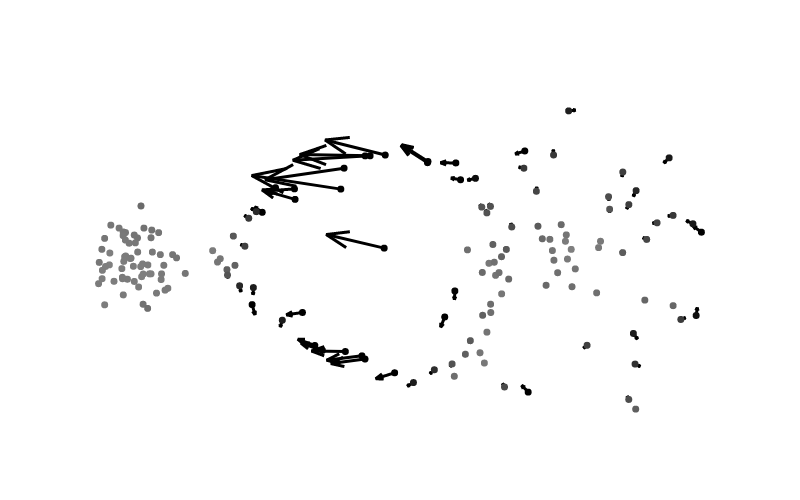}}\quad
  \subfloat[][$\nabla_X Y$]{\includegraphics[width=.48\textwidth]{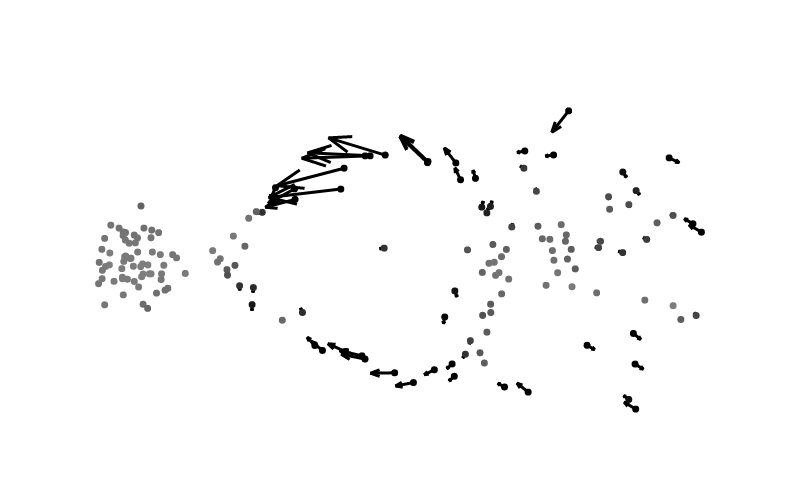}}
  \caption{\textbf{Lie bracket and covariant derivative.} 
  Given two vector fields $X$ and $Y$ (a, b), both $[X,Y]$ and $\nabla_X Y$ define a notion of \q{differentiating Y with respect to X} (c, d).
  They measure the \q{acceleration} or change in \textit{relative} velocity of $Y$ in the direction of $X$, and produce similar pictures. 
  Following the cycle of $X$ clockwise, the flow of $Y$ on the left and right sides is constant relative to $X$, and so $[X,Y]$ and $\nabla_X Y$ are zero here. 
  But at the top and bottom $Y$ changes orientation (relative to $X$): it accelerates at the bottom of the cycle and decelerates over the top, and $[X,Y]$ and $\nabla_X Y$ increase in the direction of this force.}
  \label{fig:f6}
\end{figure}

Notice, again, that we can define a \q{weak} Lie bracket as a derivation $[X,Y] : \A \rightarrow \A$ by $f \mapsto X(Y(f)) - Y(X(f))$, although this does not necessarily define a vector field without stronger assumptions.
We can also define the action of a form on the Lie bracket in this weak sense, by (in the example of 1-forms),
\begin{equation}
\label{weak lie bracket eq}
fdh([X,Y]) = f\big(X(Y(h)) - Y(X(h)) \big).
\end{equation}
This \q{weak} notion of action will be crucial in Proposition \ref{ext derivative invariant formula} below.
We find that the usual product rule holds for the Lie bracket.

\begin{restatable}{reprop}{liebracketleibniz}
\label{lie bracket leibniz}
We have $[X,fY] = X(f)Y + f[X,Y]$ and $[fX,Y] = - Y(f)X + f[X,Y]$.
\end{restatable}

The above properties of $\nabla$ and the Lie bracket lead, by a standard argument, to the \textit{Koszul formula}
\begin{equation*}
    \begin{split}
        g(\nabla_X Y,Z) = \nabla Y (X,Z) = \frac{1}{2}\Big(
        & X(g(Y,Z)) + Y(g(Z,X)) - Z(g(X,Y)) \\
        &+ g([X,Y],Z) - g([Y,Z],X) + g([Z,X],Y) \Big),
    \end{split}
\end{equation*}
which shows that $\nabla$ is the unique affine linear connection (in the sense of Proposition \ref{nabla affine connection}) which is torsion-free (satisfies Definition \ref{Lie bracket def}) and preserves the metric (Proposition \ref{nabla preserves metric}), and so is the \textit{Levi-Civita connection}.

We can use the Lie bracket to derive the following formula to evaluate the Hessian.

\begin{restatable}{reprop}{hessianevalformulae}
\label{hessian eval formulae}
The Hessian satisfies
$$
H(f)(X,Y) = \frac{1}{2}\Big(g(X, [Y,\nabla f]) + g(Y,[X,\nabla f]) + \Gamma(f, g(X,Y)) \Big).
$$
\end{restatable}

\subsubsection{Covariant derivative of forms}

We can use these conditions, as well as duality, to extend the definition of $\nabla_X$ to other tensors.
We first define the following notions of \q{differentiable} forms analogous to $\mathfrak{X}^1(M)$.

\begin{definition}
We define the \textbf{space of differentiable $k$-forms}
$$
\Omega^{k,1}(M) := \left\{\sum f_i dh^i_1 \wedge \dots \wedge dh^i_k \in \Omega^k(M) : h^i_j \in W^{2,\infty}(M) \text{ for all } i,j\right\}.
$$
\end{definition}

So $\Omega^{1,1}(M)$ is exactly the $\flat$-dual of $\mathfrak{X}^1(M)$.
We can define the covariant derivative on 1-forms by this duality.

\begin{definition}
If $\alpha \in \Omega^{1,1}(M)$, then define $\nabla_X \alpha := \big(\nabla_X(\alpha^\sharp)\big)^\flat = \nabla(\alpha^\sharp)(X,\cdot)$.
\end{definition}

We can derive the usual formula for the action of $\nabla_X \alpha$ on a vector field $Y$, using the fact that $\nabla_X$ preserves the metric, as
\begin{equation*}
\begin{split}
\nabla_X \alpha (Y)
&= g(\nabla_X(\alpha^\sharp), Y) \\
&= \nabla_X g(\alpha^\sharp, Y) - g(\alpha^\sharp, \nabla_X Y) \\
&= X(\alpha(Y)) - \alpha(\nabla_X Y).
\end{split}
\end{equation*}
This action formula defines a \q{weak} notion of $\nabla_X \alpha$ as an operator $\mathfrak{X}(M)\rightarrow\A$ for any $\alpha \in \Omega^1(M)$, but $\nabla_X \alpha$ is not necessarily a 1-form in $\Omega^1(M)$ unless $\alpha \in \Omega^{1,1}(M)$.
We can further extend $\nabla_X$ to tensor products in a way that satisfies the Leibniz rule, so in general
$$
\nabla_X(\xi \otimes \zeta) = \nabla_X(\xi) \otimes \zeta + \xi \otimes \nabla_X(\zeta).
$$
In particular, we obtain the following definition.

\begin{definition}
$\nabla_X$ is defined inductively on $\alpha, \beta \in \Omega^{k,1}(M)$ by
$$
\nabla_X(\alpha \wedge \beta) := \nabla_X(\alpha) \wedge \beta + \alpha \wedge \nabla_X(\beta).
$$
\end{definition}

It is straightforward to check that $\nabla_X \alpha$ is still linear in $\alpha$ and $\A$-linear in $X$, and satisfies the Leibniz rule in $\alpha$.

\subsection{Exterior derivative and codifferential}
\label{subsec ext derivative}

We can now introduce the exterior derivative on higher-order forms.
The derivative is a fundamental object in geometry, although we have waited this long to define it to employ the Lie bracket in the following proofs.

\begin{definition}
The \textbf{exterior derivative} is the linear map $d_k : \Omega^k(M) \rightarrow \Omega^{k+1}(M)$ defined by
$$
f dh_1\wedge\ddd\wedge dh_k \mapsto df \wedge dh_1\wedge\ddd\wedge dh_k.
$$
\end{definition}

\begin{restatable}{reprop}{extderivativeinvariantformula}
\label{ext derivative invariant formula}
If $X_0,...,X_k \in \mathfrak{X}(M)$ and $\alpha \in \Omega^k(M)$ then $d_k \alpha$ satisfies
$$
d_k \alpha (X_0,...,X_k) = \sum_{i=0}^k (-1)^i X_i(\alpha(...,\hat{X_i},...)) + \sum_{i<j} (-1)^{i+j} \alpha([X_i, X_j],...,\hat{X_i},...,\hat{X_j},...),
$$
where $[X_i, X_j]$ is defined in the weak sense\footnote{See equation \ref{weak lie bracket eq}.}.
In particular, if $\|\alpha\| = 0$ then $\|d_k \alpha\| = 0$, so $d_k$ is a well defined map.
\end{restatable}

As with $d_0$, we will usually drop the subscript.
We immediately see that $d^2 = 0$, because
$$
d^2(fdh_1\wedge\dots\wedge dh_k) = d(1) \wedge df \wedge dh_1\wedge\dots\wedge dh_k
$$
and $d(1) = 0$.
We also find that the exterior derivative satisfies the Leibniz rule.

\begin{restatable}{reprop}{leibnizforextderivative}
\label{leibniz for ext derivative}
If $\alpha \in \Omega^k(M)$ and $\beta \in \Omega^l(M)$ then
$$
d(\alpha \wedge \beta) = d\alpha \wedge \beta + (-1)^k \alpha \wedge d\beta.
$$
\end{restatable}

We can also introduce the dual operator to $d$, known as the \textit{codifferential} $\partial$, which we define as follows.

\begin{definition}
The \textbf{codifferential} is the linear map $\partial_k : \Omega^{k,1}(M) \rightarrow \Omega^{k-1}(M)$ defined by
\begin{equation*}
\begin{split}
f dh_1 \wedge \dots \wedge dh_k
&\mapsto \sum_{i=1}^k (-1)^i \big( \Gamma(f, h_i) - f L(h_i) \big) dh_1 \wedge \dots \wedge \hat{dh_i} \wedge \dots \wedge dh_k \\
&\qquad + \sum_{i<j} (-1)^{i+j} f [\nabla h_i, \nabla h_j]^\flat \wedge dh_1 \wedge \dots \wedge \hat{dh_i} \wedge \dots \wedge \hat{dh_j} \wedge \dots \wedge dh_k.
\end{split}
\end{equation*}
\end{definition}

\begin{restatable}{reprop}{codiffwelldefined}
\label{codiff well defined}
The codifferential satisfies $\inp{\partial_k \alpha}{\beta} = \inp{\alpha}{d_{k-1} \beta}$ for all $\beta$, and so, in particular, is a well defined map.
\end{restatable}

We took $\Omega^{k,1}(M)$ as the domain for $\partial_k$ because we needed $[\nabla h_i, \nabla h_j]^\flat$ to be a well defined form.
We can, if needed, extend this domain by the adjoint relationship with $d_{k-1}$.

\begin{definition}
The \textbf{domain of the codifferential} $\Omega^k_\partial(M)$ is defined as the set of all $\alpha \in \Omega^k(M)$ for which there exists some $\eta \in \Omega^{k-1}(M)$ satisfying
$$\inp{\eta}{\beta} = \inp{\alpha}{d_{k-1} \beta}$$
for all $\beta \in \Omega^{k-1}(M)$.
This $\eta$ is then uniquely determined and is called $\partial_k \alpha$.
We equip $\Omega^k_\partial(M)$ with the Sobolev inner product
$$\inp{\alpha}{\beta}_\partial = \inp{\alpha}{\beta} + \inp{\partial\alpha}{\partial\beta}$$
and let $L^2\Omega^k_\partial(M)$ be the completion of  $\Omega^k_\partial(M)$ with respect to $\|\cdot\|_\partial$.
\end{definition}

We can, as usual, also define a \q{weak} codifferential $\partial_k \alpha: \mathfrak{X}^{k-1}(M) \rightarrow \A$ for all $\alpha \in \Omega^k(M)$, using the weak Lie bracket
$$
[\nabla h_i, \nabla h_j]^\flat(X) = \Gamma(h_i, X(h_j)) - \Gamma(h_j, X(h_i)).
$$

\subsection{Differential operators, cohomology and the Hodge Laplacian}
\label{theory_cohomology_section}

We can use the standard language of differential operators from commutative algebra to describe operators in diffusion geometry. 

\begin{definition}
A linear operator $X:\A \rightarrow \A$ is
\begin{itemize}
    \item a \textbf{first-order differential operator} if its carré du champ
    $$
    \Gamma_X(f,h) = \frac{1}{2}\big(X(fh) - fX(h) - hX(f)\big)
    $$
    vanishes (i.e. it is a derivation), and
    \item a \textbf{$\mathbf{n^{th}}$-order differential operator} if its carré du champ is a differential operator of order  $(n-1)$ in both $f$ and $h$.
\end{itemize}
\end{definition}

Vector fields are first-order differential operators.
Notice that the Markov generator $L$ is a second-order operator precisely when the carré du champ $\Gamma$ satisfies the diffusion property (Definition \ref{diffusion property}). The Hessian $f \mapsto H(f)(X,Y)$ is also a second-order operator because we can show
$$
\Gamma_H(f,h)(X,Y) = \frac{1}{2}\big( X(f)Y(h) + X(h)Y(f) \big),
$$
for all $X$ and $Y$, which satisfies the Leibniz rule in both $f$ and $h$.
We can also define differential algebra for the Markov triple $M$.

\begin{definition}
We call $\Omega^\bullet(M) = \bigoplus_k \Omega^k(M)$ the \textbf{algebra of differential forms}.
The facts that $d^2 = 0$ and $d$ satisfies the Leibniz rule (Proposition \ref{leibniz for ext derivative}) mean that $(\Omega^\bullet(M), \wedge, d)$ is a \textit{commutative differential graded algebra}.
We can interpret $\Omega^\bullet(M)$ as a \textit{cochain complex}
$$
\begin{tikzcd}[column sep=small]
\Omega^0(M) \arrow[r, "d"]
& \Omega^1(M) \arrow[r, "d"]
& \ddd \arrow[r, "d"]
& \Omega^{k-1}(M) \arrow[r, "d"]
& \Omega^k(M) \arrow[r, "d"]
& \Omega^{k+1}(M) \arrow[r, "d"]
& \ddd
\end{tikzcd}
$$
and define the \textbf{de Rham cohomology groups}
$$
H^k(M) := \frac{\ker(d_k)}{\text{im} (d_{k-1})}.
$$
The wedge product $\wedge$ descends to a well defined \textbf{cup product} on $H^k(M)$ by the Leibniz rule, which we denote $\cup$, and makes $(H^\bullet(M), \cup)$ into a \textit{commutative graded algebra}.
The algebra structure makes $(H^\bullet(M), \cup)$ a strictly stronger piece of data than $H^\bullet(M)$ alone: see Figure \ref{fig:cohom product} for an example of two spaces that share the same homology but have different cohomology, so can be distinguished by the cup product.
\end{definition}

A fundamental link between the geometry and topology of a manifold is established by the Hodge theorem, which uses the \textit{Hodge Laplacian}.
Before defining this, we need to introduce one more space of forms.

\begin{definition}
We define the \textbf{space of coefficient-differentiable $k$-forms}
$$
\Omega^{k,1}_+(M) = \left\{\sum f_i dh^i_1 \wedge \dots \wedge dh^i_k \in \Omega^k(M) : f_i, h^i_j \in W^{2,\infty}(M) \text{ for all } i,j\right\}.
$$
\end{definition}

Recall that forms in $\Omega^{k,1}(M)$ do not also require the coefficient functions $f_i$ to be in $W^{2,\infty}(M)$.
This additional assumption in $\Omega^{k,1}_+(M)$ means we can define the following.

\begin{definition}
The \textbf{Hodge Laplacian} is the linear map $\Delta_k : \Omega^{k,1}_+(M) \rightarrow \Omega^k(M)$ defined by
$$
\Delta_k := d_{k-1}\partial_k + \partial_{k+1}d_k.
$$
\end{definition}

This definition makes sense for $\alpha \in \Omega^{k,1}_+(M)$, because $\Omega^{k,1}_+(M) \subseteq \Omega^{k,1}(M)$, so $\partial\alpha$ is well defined, and $d\alpha \in \Omega^{k+1,1}(M)$, so $\partial d\alpha$ is also well defined.
As with the codifferential, the domain of $\Delta_k$ may actually be slightly larger, and so we make the following definition.

\begin{definition}
The \textbf{domain of the Hodge Laplacian} $\Omega^k_\Delta(M)$ is defined as the set of all $\alpha \in \Omega^k_\partial(M)$ for which there exists some $\eta \in \Omega^k(M)$ satisfying
$$\inp{\eta}{\beta} = \inp{d\alpha}{d\beta} + \inp{d\partial\alpha}{\beta}$$
for all $\beta \in \Omega^k(M)$.
This $\eta$ is then uniquely determined and is called $\Delta_k \alpha$.
We equip $\Omega^k_\Delta(M)$ with the Sobolev inner product
$$\inp{\alpha}{\beta}_\Delta = \inp{\alpha}{\beta} + \inp{d\alpha}{d\beta} + \inp{\partial\alpha}{\partial\beta}$$
and let $L^2\Omega^k_\Delta(M)$ be the $\|\cdot\|_\Delta$-completion of $\Omega^k_\Delta(M)$.
\end{definition}

For clarity, the hierarchy of spaces we have now defined is
$$
\Omega^{k,1}_+(M) \subseteq \Omega^k_\Delta(M) \subseteq \Omega^k_\partial(M) \subseteq \Omega^k(M).
$$

The Hodge theorem states that, if $\M$ is a manifold (or weighted manifold), then $H^k(\M) \cong \ker \Delta_k$.
It may be the case that $H^k(M) \cong \ker \Delta_k$ on more general Markov triples than manifolds, but we do not address that question here (see \cite{gigli2018nonsmooth} for a discussion of Hodge theory on $L^2$ spaces of forms on $RCD(k,\infty)$ spaces).
Nonetheless, this result will motivate our practical approach to measuring the \q{shape} of a dataset with the spectrum of the Hodge Laplacian (see Figure \ref{fig:f4}).

\begin{figure}[!ht]
  \centering
  \captionsetup{width=0.66\linewidth}
  \subfloat[][0.09]{\includegraphics[width=.3\textwidth]{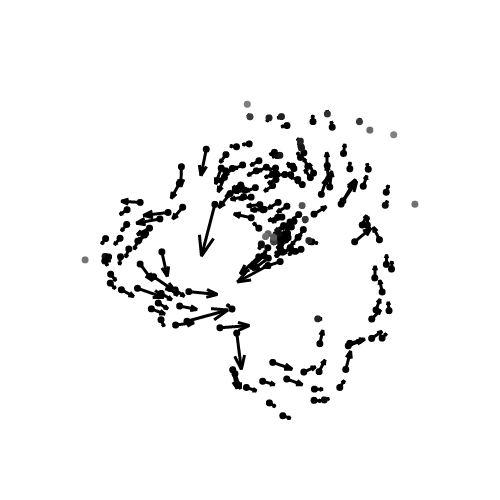}}
  \subfloat[][0.44]{\includegraphics[width=.3\textwidth]{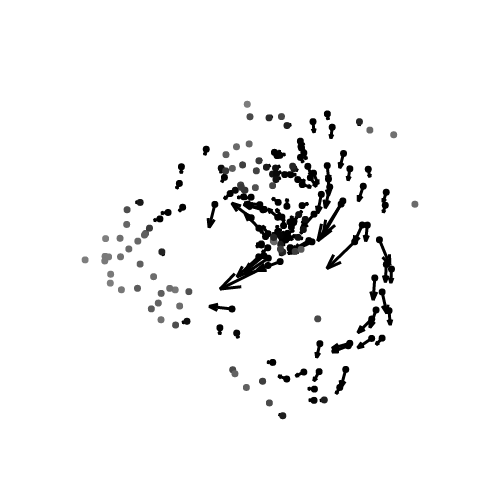}}
  \subfloat[][0.60]{\includegraphics[width=.3\textwidth]{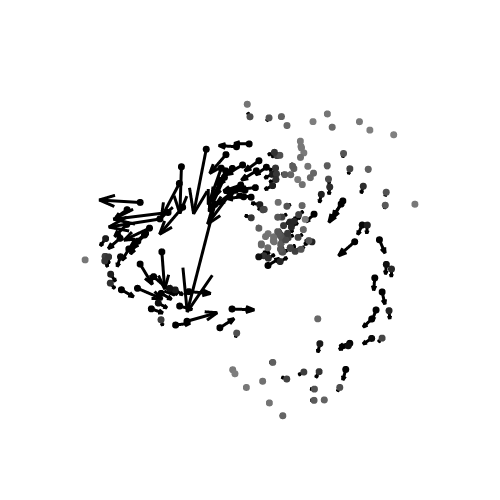}}\\
  \subfloat[][0.67]{\includegraphics[width=.3\textwidth]{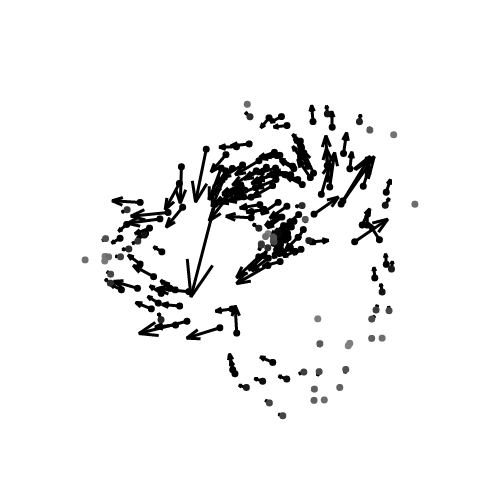}}
  \subfloat[][1.63]{\includegraphics[width=.3\textwidth]{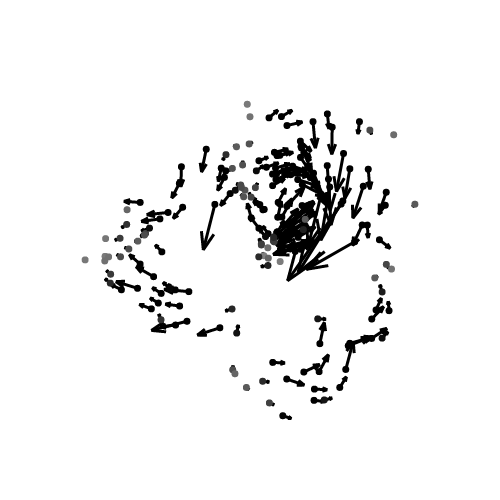}}
  \subfloat[][2.33]{\includegraphics[width=.3\textwidth]{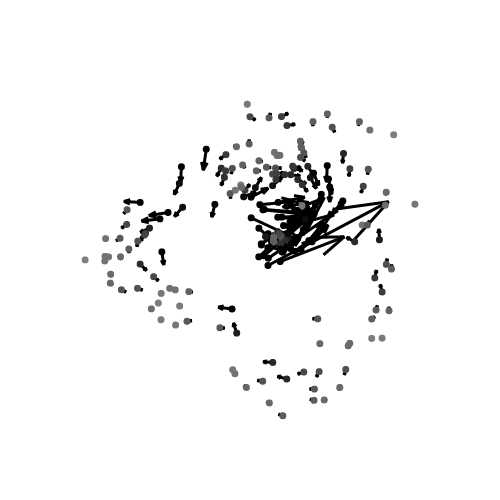}}\\
  \caption{\textbf{The spectrum of the 1-Hodge Laplacian.}
  Eigenforms are labelled with their eigenvalues (a - f). 
  The first eigenvalue is very close to zero (0.09) and, up to the influence of noise and artefacts of the discretisation, represents the large \q{hole} in the data as an element of its $H^1$ cohomology. 
  The eigenform's dual vector field (a) represents a rotating flow around that hole. 
  The second eigenvalue is bigger (0.44), and measures a much less prominent hole at the top as indicated by its corresponding eigenform (b).}
  \label{fig:f4}
\end{figure}

\subsection{Third order calculus: second covariant derivative and curvature}
\label{subsec third order}

We are also interested in the curvature of functions, tensors, and spaces, and this analysis requires the \textit{iterated} covariant derivative. 
To simplify the exposition, and to work in further generality, we consider the covariant derivative $\nabla_X \xi$ where $\xi \in F$ and $F$ may stand for $\A$, $\mathfrak{X}^1(M)$, or $\Omega^{k,1}(M)$ (or other tensors).
We can then regard $\nabla_\bullet \xi : \mathfrak{X}(M) \rightarrow F$ as an \textit{$F$-valued 1-form}, i.e. an element of $\Omega^1(M) \otimes F$.
To avoid confusion with the gradient $\nabla: \A \rightarrow \mathfrak{X}(M)$, we will always write $\nabla_\bullet$ to mean the covariant derivative in this sense. 
So
$$
\nabla_\bullet: F \rightarrow \Omega^1(M) \otimes F.
$$
We have already introduced two examples of this $\nabla_\bullet$:
\begin{enumerate}
    \item $F = \mathfrak{X}^1(M)$, where $\nabla_\bullet$ is the (1,1) tensor of Definition \ref{(1,1) nabla}, and
    \item $F = \A$, where $\nabla_X f = X(f) = df(X)$, so $\nabla_\bullet f = df$ and $\nabla_\bullet = d:\A \rightarrow \Omega^1(M)$.
\end{enumerate}

Just like the extension of $\nabla_\bullet$ from vector fields to 1-forms, we will extend $\nabla_\bullet$ to $\Omega^1(M) \otimes F$ by preserving the Leibniz rule, so we expect
$$
\nabla_\bullet(\alpha\otimes\xi) = \nabla_\bullet(\alpha)\otimes\xi + \alpha\otimes\nabla_\bullet(\xi).
$$
For example, for this to hold would mean that the \textit{second covariant derivative} $\nabla^2_{X,Y}$ must satisfy
$$
\nabla_X(\nabla_Y\xi) = \nabla_{\nabla_X Y}\xi + \nabla^2_{X,Y} \xi,
$$
although this requires vector fields with even more regularity than $\mathfrak{X}^1(M)$.

\begin{definition}
We define the \textbf{space of twice-differentiable vector fields}
$$
\mathfrak{X}^2(M) := \{X \in \mathfrak{X}^1(M) : \nabla X \in \Omega^1(M) \otimes \Omega^{1,1}(M)\},
$$
or, equivalently, $X \in \mathfrak{X}^2(M)$ if $\nabla_Y X \in \mathfrak{X}^1(M)$ for all $Y \in \mathfrak{X}(M)$.
\end{definition}

These vector fields are \q{twice-differentiable} because we can take their covariant derivative twice.
If $Z \in \mathfrak{X}^2(M)$ then $\nabla_X Z \in \mathfrak{X}^1(M)$ for all $X \in \mathfrak{X}(M)$, so $\nabla_Y \nabla_X Z \in \mathfrak{X}(M)$ is well defined for all $Y \in \mathfrak{X}(M)$.
We also define the corresponding spaces of twice-differentiable $k$-forms.

\begin{definition}
We define the \textbf{space of twice-differentiable $k$-forms} as
$$
\Omega^{k,2}(M) := \{\alpha \in \Omega^{k,1}(M) : \nabla_Y \alpha \in \Omega^{k,1}(M) \text{ for all } Y \in \mathfrak{X}(M)\}.
$$
In particular, $\Omega^{1,2}(M) = \mathfrak{X}^2(M)^\flat$.
\end{definition}

As before, we can take the covariant derivative twice: if $\alpha \in \Omega^{k,2}(M)$, then $\nabla_X \alpha \in \Omega^{k,1}(M)$ and $\nabla_Y \nabla_X \alpha \in \Omega^k(M)$.
We can now define the second covariant derivative on $\A$, $\mathfrak{X}^2(M)$, and $\Omega^{k,2}(M)$.

\begin{definition}
If $X \in \mathfrak{X}(M)$, $Y \in \mathfrak{X}^1(M)$, and $F = \A$, $\mathfrak{X}^2(M)$, or $\Omega^{k,2}(M)$, the \textbf{second covariant derivative} on $F$ is given by $\nabla^2_{X,Y} := \nabla_X\nabla_Y - \nabla_{\nabla_X Y}$. So if
\begin{enumerate}
    \item $F = \A$, then $\nabla^2_{X,Y} : \A \rightarrow \A$,
    \item $F = \mathfrak{X}^2(M)$, then $\nabla^2_{X,Y} : \mathfrak{X}^2(M) \rightarrow \mathfrak{X}(M)$, and
    \item $F = \Omega^{k,2}(M)$, then $\nabla^2_{X,Y} : \Omega^{k,2}(M) \rightarrow \Omega^k(M)$.
\end{enumerate}
\end{definition}

Note that we need $Y \in \mathfrak{X}^1(M)$ to ensure that $\nabla_X Y \in \mathfrak{X}(M)$.
Also, we can see that the regularity issues that led us to define $\mathfrak{X}^1(M)$ and $\mathfrak{X}^2(M)$ do not arise for functions in $\A$, because
$$
\nabla^2_{X,Y}(f) = XY(f) - \nabla_X Y(f)
$$
is defined for all $f \in \A$.
% In fact, we can use the \q{weak} version of $\nabla_X Y$ to define $\nabla_X Y (f)$ even when $Y \in \mathfrak{X}(M)$.
In the proof of Proposition \ref{lie bracket commutator} we show that that
$$
XY(f) = \nabla_X Y (f) + H(f)(X, Y),
$$
so $\nabla^2_{X,Y}(f) = H(f)(X,Y)$.
Although we only technically proved Proposition \ref{lie bracket commutator} in the case that $X,Y \in \mathfrak{X}^1(M)$, the same argument applies to the \q{weak} formulations of operators, with which we are working here.
So in fact $\nabla^2_{X,Y}(f) = H(f)(X,Y)$ for all $X,Y \in \mathfrak{X}(M)$.

Finally, we can use the covariant derivative to introduce the Riemann curvature operator.

\begin{definition}
The \textbf{Riemann curvature operator} is the map $R: \mathfrak{X}^1(M) \times \mathfrak{X}^1(M) \times \mathfrak{X}^2(M) \rightarrow \mathfrak{X}(M)$ given by
$$
R(X,Y)Z := \nabla_X \nabla_Y Z - \nabla_Y \nabla_X Z -\nabla_{[X,Y]} Z.
$$
\end{definition}

We needed $X,Y \in \mathfrak{X}^1(M)$ to ensure $[X,Y] \in \mathfrak{X}(M)$, and $Z \in \mathfrak{X}^2(M)$ to ensure $\nabla_X \nabla_Y Z \in \mathfrak{X}(M)$.
The torsion-free property of $\nabla_X$ means that $R(X,Y) = \nabla^2_{X,Y} - \nabla^2_{Y,X} : \mathfrak{X}^2(M) \rightarrow \mathfrak{X}(M)$.

\section{Computing Diffusion Geometry from Data}\label{section_estimation}

We can separate the question of computing diffusion geometry into two parts: how to compute the infinitesimal generator (i.e. the Laplacian), and how to compute the rest of the geometry given that operator.

\subsection{Computing the infinitesimal generator}

Estimating the Laplacian of a manifold from a point cloud of samples from it is a very well-developed topic.
It has been a central question in \textit{manifold learning} because the eigenfunctions of the Laplacian are good candidates for coordinate functions for dimensionality reduction.
The approach we favour here, known as \textit{diffusion maps} \cite{COIFMAN20065} (Algorithm \ref{alg:DM}), exploits the relationship between the Laplacian and the heat kernel laid out in Section \ref{section_markov} and particularly equation \ref{inf generator}.
They show that, if $\M$ is a manifold isometrically embedded in $\R^d$, and $p_t(x,y)$ is the heat kernel \textit{in the ambient space} (equation \ref{heat kernel eq}), then $p_t$ approximates the heat kernel in $\M$. So if 
$$
P_t f(x) = \int_\M p_t(x,y) f(y) dy,
$$
then
$$
\lim_{t\rightarrow 0} \frac{f - P_t f}{t} = \Delta f
$$
pointwise (where $\Delta$ is the Laplacian on $\M$).
Given a finite sample $(x_i)_{i\leq n}$ from $\M$, then by the law of large numbers we can approximate
$$
P_t f(x_i) \approx \frac{1}{n}\sum_{j=1}^n p_t(x_i,x_j) f(x_j),
$$
like a Monte Carlo integration.
So, if the function $f$ is represented by a vector of length $n$ (encoding its value on the $n$ points), then multiplication by the heat kernel matrix $(p_t(x_i,x_j))_{ij}$ represents the action of the heat diffusion operator.
Diagonalising this matrix then yields estimators for the eigenfunctions of the Laplacian.

When $\M$ has a density $q$ (i.e. it is a weighted manifold), the authors of \cite{COIFMAN20065} show that the heat diffusion operator
$$
P_t f(x) = \int_\M p_t(x,y) f(y) q(y) dy,
$$
is generated by
\begin{equation}\label{eq generator weighted}
\lim_{t\rightarrow 0} \frac{f - P_t f}{t} = \frac{\Delta(fq)}{q} - \frac{\Delta(q)}{q}f
\end{equation}
which is equivalent (by conjugation) to the weighted Laplacian in Example \ref{eg weighted manifold}.
This is relevant when the data are sampled non-uniformly from $\M$, with probability mass function $q$, because
$$
\E \Big(\frac{1}{n}\sum_j p_t(x_i,x_j) f(x_j) \Big) = \int_\M p_t(x_i,y) f(y) q(y) dy,
$$
so the heat kernel matrix will instead yield eigenfunctions of this $q$-weighted operator (\ref{eq generator weighted}).
They also introduce a simple 1-parameter renormalisation of the heat kernel, which controls the role of this density $q$.
We may use this in practice, but refer to \cite{COIFMAN20065} for the details.
See Algorithm \ref{alg:DM} for the basic implementation of diffusion maps.

\begin{algorithm}
\caption{Diffusion maps estimation of the infinitessimal generator}\label{alg:DM}
\KwData{$n\times n$ distance matrix $D$, bandwidth $\epsilon$}
\KwResult{Eigenvalues $\lambda_i$ and eigenfunctions $\p{i}$ of the infinitessimal generator, empirical density $D$}
$P_{ij} \gets \exp(-D_{ij}^2/\epsilon$)\;
$D_i \gets \sum_j P_{ij}$\;
$D \gets \text{diagonal}(D)$\;
$\Lambda, \Phi \gets \text{solutions of the generalised eigenproblem } P\Phi = D \Phi \Lambda$\;
$\Phi \gets (1 - \Phi) / \epsilon$
\end{algorithm}

In the following, we assume that our data are drawn from some smooth distribution on $
\R^d$, or on an embedded submanifold of $\R^d$, and so the diffusion maps algorithm will estimate the infinitesimal generator (\ref{eq generator weighted}) (where the Laplacian $\Delta$ in (\ref{eq generator weighted}) is either that of $\R^d$ or the submanifold). 
However, there are many other pertinent approaches, such as the combinatorial Laplacian on a graph or simplicial complex, or the generator of some correlation matrix between the data.
Working with the graph Laplacian would lead to an alternative construction of graph theory, which we will explore in future work.

\subsection{Computing diffusion geometry with the generator}

Now suppose we have an infinitesimal generator $L$, which acts on an $n$-dimensional function algebra $\A$.
We will usually suppose that $\A\cong\R^n$ represents functions on $n$ points or vertices, and so the basis vector $e_i$ represents a function being 1 on the $i\thupper$ vertex and 0 on the others.
A natural way to represent forms is suggested by the definition, which we recall (for 1-forms) to be
$$
\Omega^1(M) = \frac{\A \otimes \A}{D^1}.
$$
We should 
\begin{enumerate}
    \item choose a basis $\{f_i : i \in I\}$ for the function algebra $\A$, so $\{f_i \otimes f_j : i,j \in I\}$ gives a basis for $\A \otimes \A$ (we could use the standard basis $\{e_i\}$ or, as discussed below, the eigenfunctions of $L$),
    \item compute the inner product (Gram) matrix $G$ for $\{f_i \otimes f_j : i,j \in I\}$, and
    \item quotient $\A \otimes \A$ by its kernel (i.e. diagonalise $G$ and project onto the orthonormal basis of eigenvectors with positive eigenvalue).
\end{enumerate}
In other words, the eigenvectors of $G$ give an orthonormal basis for $\Omega^1(M)$.
The formulae given in Section \ref{section_theory} can then be used to construct all the other objects in the theory, given the operator $L$. 
For example, we can compute the carré du champ
$$
\Gamma(f_i,f_j) = \frac{1}{2} \big( f_i L f_j + f_j L f_i - L (f_if_j) \big)
$$
and metric
$$
g(f_i df_j, f_k df_l) = f_if_k \Gamma(f_j,f_l)
$$
in the $f_i \otimes f_j$ basis for $\A \otimes \A$, and then project into the orthonormal basis for $\Omega^1(M)$.
We take the analogous approach for higher-order forms.
All the operators defined above are then matrices and tensors and can be processed using highly optimised software for multilinear algebra.
We outline this approach for 1-forms in Algorithm \ref{alg:1forms}.

\subsubsection{Well-posed Galerkin scheme}

We want to show that this finite-dimensional space of forms computed from data is an estimate for the space of forms on the underlying probability space.
In other words, just like in the finite-element method, we need to justify that these discrete representations of objects in diffusion geometry will converge to the true objects in the large-data limit.

Suppose that, instead of projecting onto eigenvectors of the Gram matrix, we had simply applied Gram-Schmidt.
This may result in a different choice of orthonormal basis but is otherwise identical to the process above.
So to represent 1-forms we can
\begin{enumerate}
    \item choose a basis $\{f_i : i \in I\}$ for $\A$, so the span of $\{b_{ij} = f_i \otimes f_j : i,j \in I\}$ projects onto a dense subspace of $L^2(\Omega^1(M))$, and
    \item apply the Gram-Schmidt orthonormalisation to the $b_{ij}$, removing linearly dependent elements, to obtain an orthonormal basis $\{c_k : k \in \N\}$ for $L^2(\Omega^1(M))$.
\end{enumerate}
In other words, \textit{because the basis is orthonormal}, we have
$$
\lim_{K\rightarrow\infty} \Big\|\sum_{k = 1}^K \inp{\alpha}{c_k} c_k - \alpha \Big\| = 0
$$
for all $\alpha \in \Omega^1(M)$.
This means that the representations of forms and operators in the finite \textit{truncated basis} $\{c_k : k \leq K\}$ are guaranteed to converge to the correct objects as $K \rightarrow \infty$.
The same approach applies to all the $\Omega^k(M)$ and justifies our computational framework for $L^2$-convergent forms.

If we want to represent operators on forms with higher regularity than $\Omega^k(M)$, such as the Hodge Laplacian $\Delta$, we need to work in the appropriate Sobolev spaces, like $\Omega^k_\Delta(M)$.
In these cases, we can apply the same method as above but pick an orthonormal basis for the Sobolev inner product instead, which in $\Omega^k_\Delta(M)$ is
$$
\inp{\alpha}{\beta}_\Delta = \inp{\alpha}{\beta} + \inp{d\alpha}{d\beta} + \inp{\partial\alpha}{\partial\beta}.
$$
If we had only used the $L^2$ inner product, there is a risk that, as $K\rightarrow\infty$, the truncated representations converge to something with no well-defined Hodge Laplacian.
By using the Sobolov inner product, we guarantee that the truncated sequence converges in $\Omega^k_\Delta(M)$ and so $\Delta$ is guaranteed to exist.

\subsubsection{Solving the eigenproblem for the Hodge Laplacian}

We can now solve the eigenproblem for $\Delta$ in $\Omega^k_\Delta(M)$ using the weak formulation: rather than solve $\Delta\alpha = \lambda\alpha$ directly, we solve
\begin{equation}\label{galerkin eq}
\inp{d\alpha}{d\beta} + \inp{\partial\alpha}{\partial\beta} + \epsilon\inp{\alpha}{\beta}_\Delta = \lambda\inp{\alpha}{\beta}
\end{equation}
for all $\beta$, where $\epsilon$ is a small regularisation parameter.
The bilinear form on the left of this equation is then \textit{coercive}, because 
$$
\inp{d\alpha}{d\alpha} + \inp{\partial\alpha}{\partial\alpha} + \epsilon\inp{\alpha}{\alpha}_\Delta \geq \epsilon\|\alpha\|_\Delta^2
$$
and \textit{bounded}, because 
\begin{equation*}
\begin{split}
\inp{d\alpha}{d\alpha} + \inp{\partial\alpha}{\partial\alpha} + \epsilon\inp{\alpha}{\alpha}_\Delta
&= (1+\epsilon)\|\alpha\|_\Delta^2 - \|\alpha\|^2 \\
&\leq (1+\epsilon)\|\alpha\|_\Delta^2
\end{split}
\end{equation*}
and the form defines an inner product so Cauchy-Schwarz implies
\begin{equation*}
\begin{split}
\inp{d\alpha}{d\beta} + \inp{\partial\alpha}{\partial\beta} + \epsilon\inp{\alpha}{\beta}_\Delta
&\leq \big(\inp{d\alpha}{d\alpha} + \inp{\partial\alpha}{\partial\alpha} + \epsilon\inp{\alpha}{\alpha}_\Delta\big)^\frac{1}{2} \big(\inp{d\beta}{d\beta} + \inp{\partial\beta}{\partial\beta} + \epsilon\inp{\beta}{\beta}_\Delta\big)^\frac{1}{2} \\
&\leq (1+\epsilon)\|\alpha\|_\Delta \|\beta\|_\Delta.
\end{split}
\end{equation*}
These two properties mean that the Galerkin equation \ref{galerkin eq} is well-posed and has a unique solution in $\Omega^k_\Delta(M)$.
The finite truncated basis $\{c_k : k \leq K\}$ is orthonormal, so the solution we obtain for each $K$ will converge to the true solution as $K \rightarrow \infty$.
This formulation is fully symmetric, so is efficient to compute (see Section \ref{appendix hodge laplacian}).

We can solve equations involving other operators, like the Hessian or codifferential, by defining the corresponding Sobolev inner products on their domains, picking an orthonormal basis for that inner product, and solving the equations in the truncated basis.

\subsubsection{Choosing the basis}\label{subsection choosing basis}

The steps above give a general computational framework for diffusion geometry, although working with higher-order forms and tensors may become computationally expensive as the dimension of the function space, $n$, increases.
Usually (and in the \textit{diffusion maps} setting) $n$ is the number of data, so this computational complexity will become prohibitive for large data sets.
We can, instead, project the functions into a smaller space and reduce the cost.

If $L$ is an $n \times n$ matrix, we can diagonalise $L$ and use the span of its first $n_0$ eigenfunctions as a compressed representation of $\A$. 
This is equivalent to the band-limited Fourier space of functions when $L$ is the Laplacian on a manifold.
This space is a natural choice for compression, because, among all sets of $n_0$ functions in $\A$, the first $n_0$ eigenfunctions of $L$ have the lowest \textit{Dirichlet energy} as defined by the energy functional
$$
\mathcal{E}(f) := \int fL(f) d\mu = \int \Gamma(f,f) d\mu.
$$
Functions with low Dirichlet energy are smoother and less oscillatory, so we can capture the finer details of functions by increasing $n_0$.

We apply the same principle to representing higher-order forms.
For example, the 2-forms are spanned by the terms $\p{i}d\p{j}\wedge d\p{k}$, so if we use all $n$ functions $\p{i},\p{j},\p{k}$ then the 2-forms are (at most) $n^3$-dimensional.
Instead, we limit $i \leq n_1$ and $j,k \leq n_2$ for some $n_1,n_2 \leq n$.
These \textit{approximate} 2-forms are then (at most) $n_1n_2^2$-dimensional.
In general, the $k$-forms are $n_1n_2^k$-dimensional, which lets us explicitly trade off computational complexity against precision.
% It is important to note here that, on a manifold, we only need $n_2$ to be large enough that $\p{1},...,\p{n_2}$ span the cotangent space at every point.
We test the consequences of this in Section \ref{section_complexity}.

Working with these eigenfunctions also leads to some simplification in the formulae. Let $\p{i}$ be the $i\thupper$ eigenfunction of $L$, with eigenvalue $\lambda_i$, and define the \textit{structure constants}
$
c_{ijk} := \int \p{i}\p{j}\p{k} d\mu,
$
so that $\p{i}\p{j} = \sum_k c_{ijk}\p{k}$.
Then, for example, we can simplify
$$
\Gamma(\p{i},\p{j}) = \frac{1}{2} \sum_k (\lambda_i + \lambda_j - \lambda_k) c_{ijk} \p{k}.
$$
Formulae for computing all the other objects in Section \ref{section_theory} follow straight from their definitions, and are deferred to Appendix B (Section \ref{appendix implementation}). 
The Python code is available at \url{https://github.com/Iolo-Jones/DiffusionGeometry}.
We outline this approach to computing 1-forms in Algorithm \ref{alg:1forms}, using formulae referenced from Appendix B (Section \ref{appendix implementation}).

\begin{algorithm}
\caption{Computing 1-forms in diffusion geometry}\label{alg:1forms}
\KwData{Eigenvalues $\lambda_i$ and eigenfunctions $\p{i}$ of the infinitessimal generator, empirical density $D$, parameters $n_0, n_1, n_2$}
\KwResult{Projection matrix $P$ onto an orthonormal basis for the 1-forms}
$c_{ijk} \gets \sum_{s=1}^{n_0} (\p{i})_s(\p{j})_s(\p{k})_s D_s$ 
\hfill $\vartriangleright$ structure constants (\ref{structure constants appendix}), so

\hfill $\p{i}\p{j} = \sum_{k=1}^{n_0} c_{ijk}\p{k}$

$\Gamma_{ijs} \gets \frac{1}{2} \sum_{k=1}^{n_0} (\lambda_i + \lambda_j - \lambda_k) c_{ijk} (\p{k})_s$ 
\hfill $\vartriangleright$ carré du champ (\ref{cdc appendix}), so

\hfill $\Gamma(\p{i},\p{j}) = \sum_{k=1}^{n_0} \Gamma_{ijk}\p{k}$

$G_{ijkl} \gets \sum_{s=1}^{n_0} c_{iks} \Gamma_{jls}$
\hfill $\vartriangleright$ inner products as a 4-tensor (\ref{gram matrix forms appendix})

$G \gets G_{(ij)(kl)}$
\hfill $\vartriangleright$ Gram matrix: inner products 

\hfill reshaped to a 2-tensor

$L, U \gets \text{solutions of the eigenproblem } GU = UL$ 

\qquad\qquad in increasing eigenvalue order

$\nu \gets \text{index of the smallest positive eigenvalue in }L$

$P \gets (U_{ij}/ \sqrt{L_i} : j \geq \nu)$
\hfill $\vartriangleright$ projection matrix onto an orthonormal \\
\hfill basis for the 1-forms (i.e. $P^T G P = I_\nu$)

\end{algorithm}

This approach to representing and computing forms as tensor products of eigenfunctions is based on Berry and Giannakis' \q{spectral exterior calculus} (SEC) on Riemannian manifolds \cite{berry2020spectral}. The authors used the fact that simplifications like the above exist to derive formulae for computing the Hodge Laplacian on 1-forms. In the special case of Riemannian manifolds, the computation of diffusion geometry can be seen as an extension of the SEC programme to other objects in Riemannian geometry, as well as higher-order forms.

These Laplacian eigenfunctions are, in the above sense, the most natural choice for representing objects in diffusion geometry but computing them has $\ord(n^3)$ complexity (from diagonalising the $n\times n$ matrix in Algorithm \ref{alg:DM}), which may still be prohibitively expensive in some cases.
We could choose other bases to allow faster approximate computation of diffusion geometry, but this will be explored in future work.

\section{Computational Geometry and Topology, Machine Learning, and Statistics}\label{section_ml}

Diffusion geometry provides a very general framework for explainable geometric machine learning and statistics, where we can interpret all the objects defined above as representing \textit{geometric features} of the data.
It also provides a broad framework for computational geometry and topology, where the data are assumed to lie on a manifold and we want to compute various properties.

In this section, we give some simple demonstrations of diffusion geometry as a tool for
\begin{enumerate}
    \item computational geometry, as a measure for identifying intersections and tangent vectors in manifold-like data,
    \item computational topology, for computing the cohomology of manifolds,
    \item unsupervised learning, as a biomarker for the infiltration of immune cells into a tumour, and
    \item supervised learning, as a feature vector for classifying different types of immune cells in tumours by their spatial distribution.
\end{enumerate}

\subsection{Computational geometry and topology}

The geometry of manifolds is a common source of inspiration for computational geometry, where the data set is often called the \q{data manifold}.
This is usually just a metaphor, but the explicit assumption that data are actually drawn from a manifold is called the \q{manifold hypothesis}.
In this setting, diffusion geometry agrees with the usual Riemannian geometry so can be used to estimate properties of that manifold from a finite sample.
When the data are drawn from a space that looks almost everywhere like a manifold, we can still compute the same objects and use these to \textit{test} the manifold hypothesis.

\begin{figure}[!ht]
  \centering
  \captionsetup{width=0.66\linewidth}
  % \subfloat[][Two circles intersecting]{\includegraphics[width=.48\textwidth]{images/paper_images/manifold_hypothesis/fig2.png}}\quad
  % \subfloat[][Two circles intersecting (noisy)]{\includegraphics[width=.48\textwidth]{images/paper_images/manifold_hypothesis/fig2_n.png}}\\
  \subfloat[][Two circles and a curve intersecting]{\includegraphics[width=.48\textwidth]{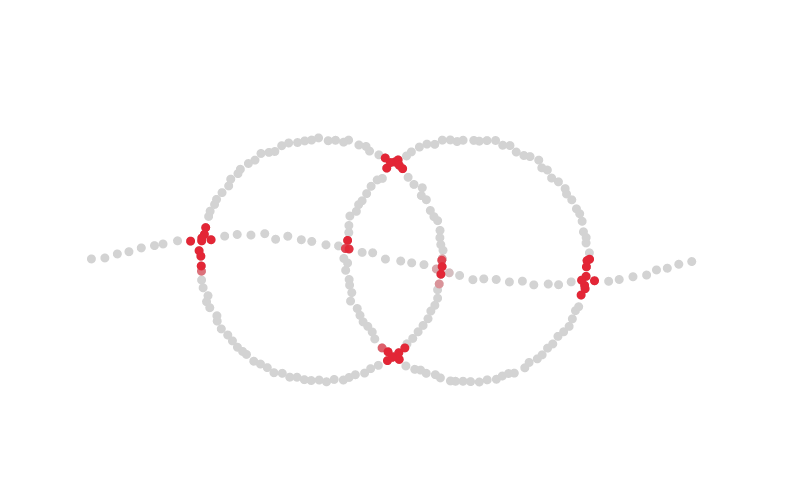}}\quad
  \subfloat[][Two circles and a curve intersecting (noisy)]{\includegraphics[width=.48\textwidth]{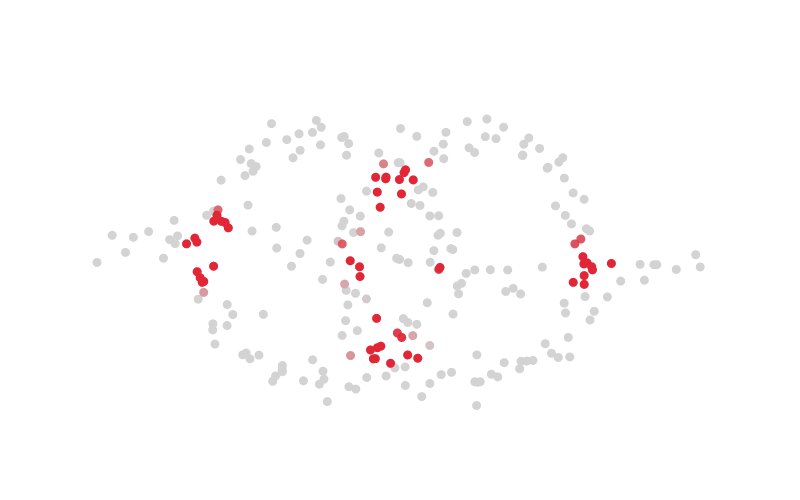}}\\
  \subfloat[][Two spheres and a circle intersecting]{\includegraphics[width=.48\textwidth]{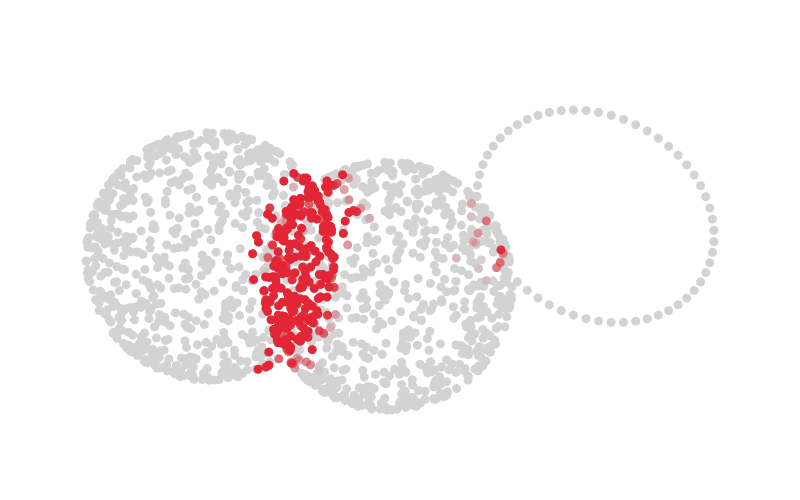}}\quad
  \subfloat[][Two spheres and a circle intersecting (noisy)]{\includegraphics[width=.48\textwidth]{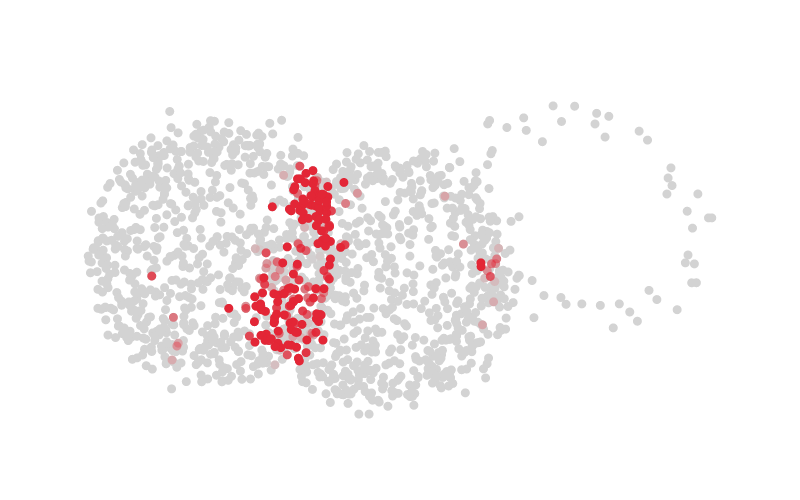}}\\
  \caption{\textbf{Testing the manifold hypothesis in 2D (a, b) and 3D (c, d).} 
  The first eigenvalue of the metric measures the degeneracy of the tangent space. 
  It is large (grey) where the data looks like a manifold, and dips lower (rose) at the singularities where the tangent space degenerates. 
  This measure is very robust to noise (b, d).}
  \label{fig:manifold_hypothesis_1}
\end{figure}

\subsubsection{Geometry of manifolds and cell complexes}

On a Riemannian manifold $\M$ of dimension $d$, each point has a \textit{tangent space} $T_x\M$, and the \textit{Riemannian metric} is an inner product $T_x\M \times T_x\M \rightarrow \R$.
The tangent space $T_x\M$ is $d$-dimensional, and so the metric has exactly $d$ positive eigenvalues, whose eigenvectors form an orthonormal basis for $T_x\M$.
Given a sample of data from $\M$ we can compute the metric at $x$ (i.e. a square symmetric matrix at each point $x$) and find these eigenvalues.
The $k\thupper$ largest eigenvalue is, therefore, a proxy statistic for \q{at least $k$-dimensionality at $x$}.

Suppose now that the data are drawn from a space that is not quite a manifold (such as the intersection of manifolds or a cell complex).
The tangent space is well defined everywhere that the space looks like a manifold, but degenerates at the points where the manifold hypothesis fails.
We can measure this local failure with the largest eigenvalue of the metric, which should be positive where the manifold hypothesis holds, but dip down towards zero at the degenerate points.
We evaluate this in Figure \ref{fig:manifold_hypothesis_1}.
The results are comparable to existing approaches for singularity detection (such as \cite{lim2023hades, stolz2020geometric, von2023topological}), but are significantly more robust to noise and outliers.

Where the manifold hypothesis holds, i.e. the largest eigenvalue of the metric is positive, we can also expect its eigenvectors to span the tangent space.
We test this in Figure \ref{fig:tangent_bundle_1}.
Unlike the commonly used local principal component analysis \cite{kambhatla1997dimension}, this approach is very robust to noise and outliers, and may yet perform better with further smoothing processes, which we will explore in future work.

\begin{figure}[!ht]
  \centering
  \captionsetup{width=0.66\linewidth}
  \subfloat[][]{\includegraphics[height=.35\textwidth]{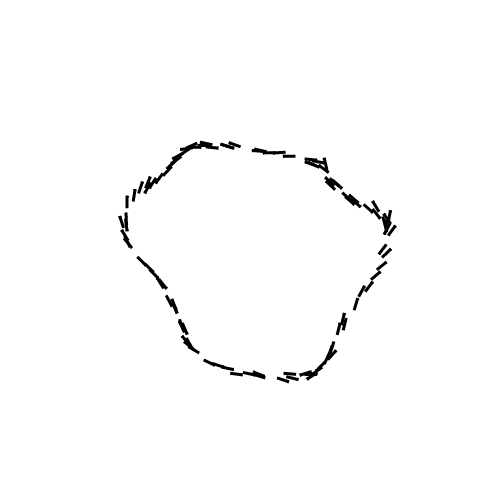}}
  \subfloat[][]{\includegraphics[height=.35\textwidth]{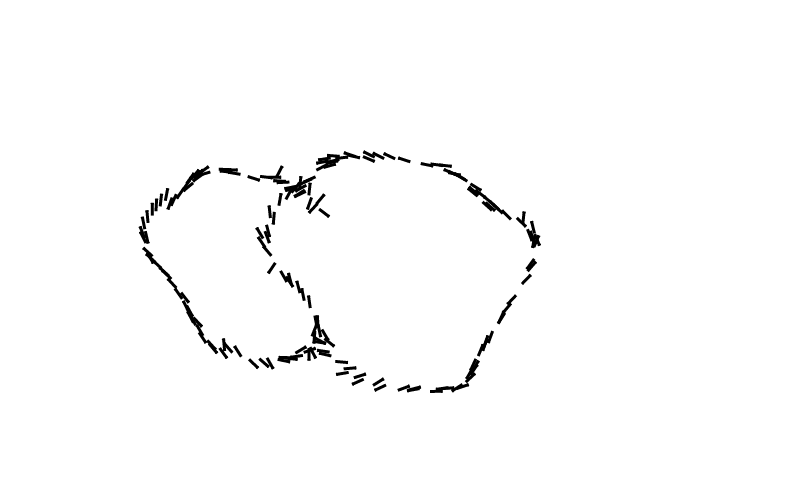}}\\
  \subfloat[][]{\includegraphics[height=.35\textwidth]{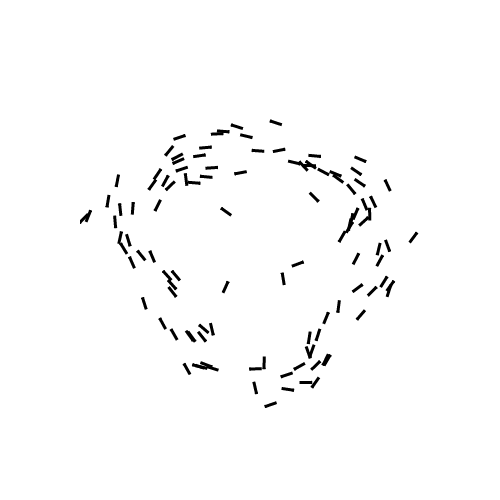}}
  \subfloat[][]{\includegraphics[height=.35\textwidth]{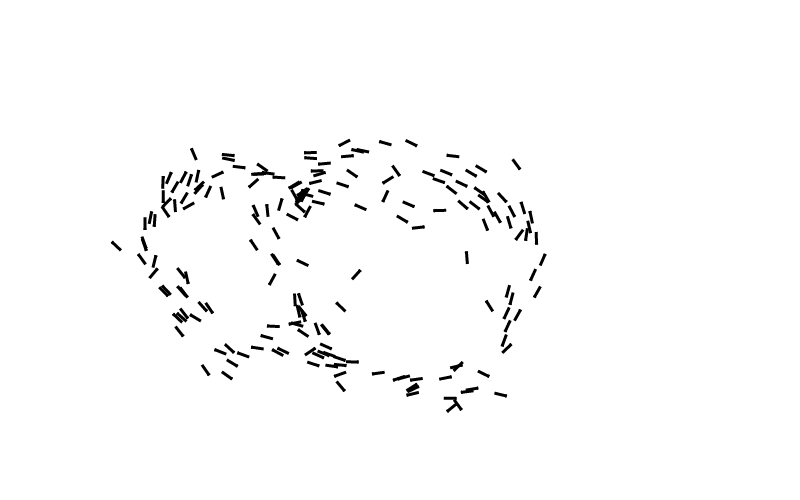}}
  \caption{\textbf{Fitting tangent lines to data.} 
  The first eigenvector of the metric represents the tangent space on a one-dimensional manifold (a). 
  Where the manifold hypothesis fails (b), the tangent vectors try to align at the intersections. 
  This process is very robust to noise (c, d).}
  \label{fig:tangent_bundle_1}
\end{figure}

\subsubsection{Topology}

Homology and cohomology are important topological properties of a space, which can be used to tell spaces apart when they are topologically distinct.
The dimensions of the homology (or cohomology) groups $H_i$ are called the \textit{Betti numbers} $\beta_i$ of the space, and measure the number of topological features in each dimension.
For example, $\beta_0$ is the number of connected components, $\beta_1$ is the number of holes or loops in the space, and $\beta_2$ is the number of enclosed voids.

\begin{figure}[!ht]
  \centering
  \captionsetup{width=0.66\linewidth}
  \subfloat[][$\alpha_1$, $\|\alpha_1\| = 1$]{\includegraphics[width=.48\textwidth]{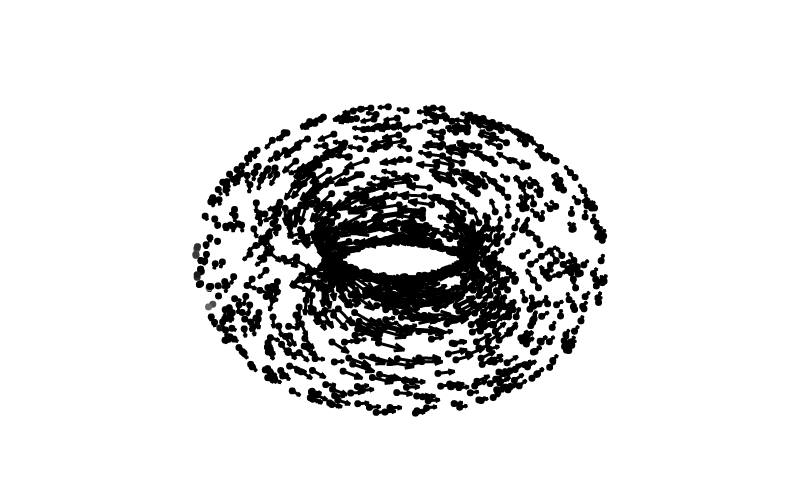}}\quad
  \subfloat[][$\beta_1$, $\|\beta_1\| = 1$]{\includegraphics[width=.48\textwidth]{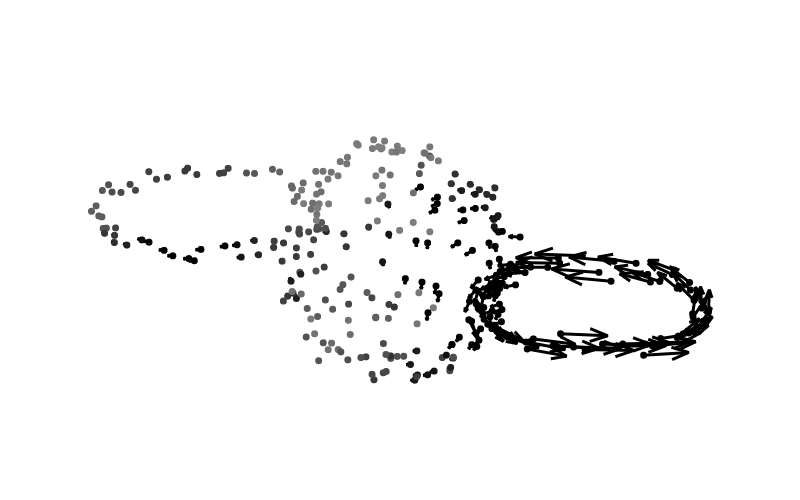}}\\
  \subfloat[][$\alpha_2$, $\|\alpha_2\| = 1$]{\includegraphics[width=.48\textwidth]{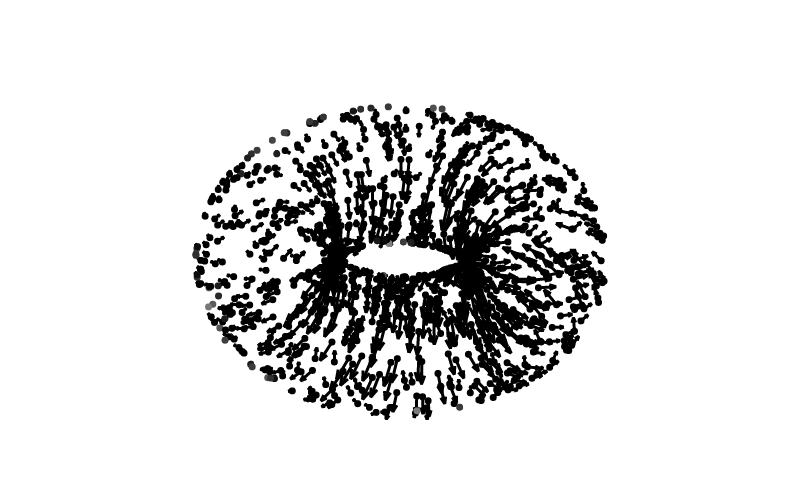}}\quad
  \subfloat[][$\beta_2$, $\|\beta_2\| = 1$]{\includegraphics[width=.48\textwidth]{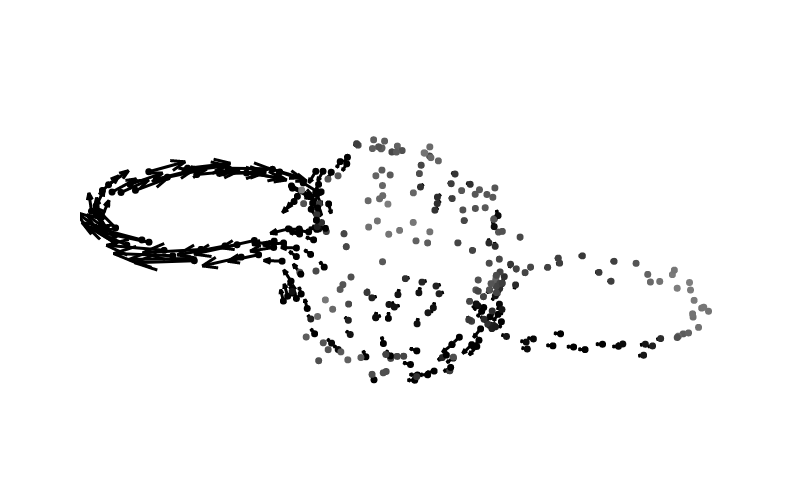}}\\  
  \subfloat[][$\alpha_1\wedge\alpha_2$, $\|\alpha_1\wedge\alpha_2\| = 0.112$]{\includegraphics[width=.48\textwidth]{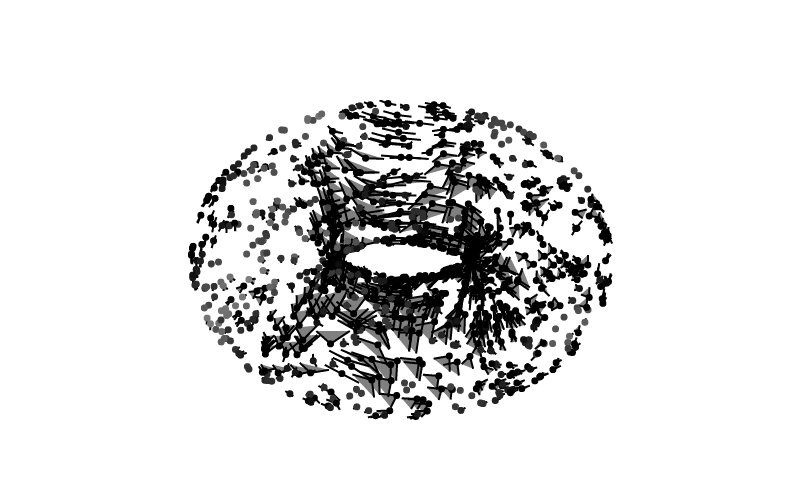}}\quad
  \subfloat[][$\beta_1\wedge\beta_2$, $\|\beta_1\wedge\beta_2\| = 0.003$]{\includegraphics[width=.48\textwidth]{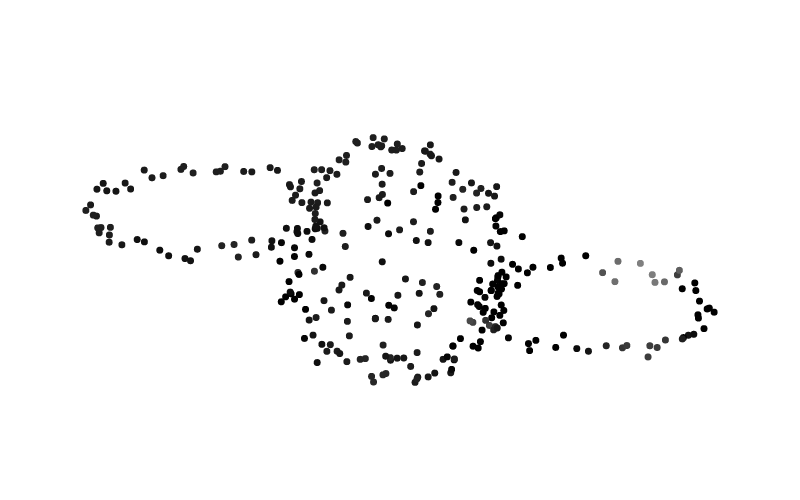}}\\
  \caption{\textbf{The wedge product on cohomology.}
  A torus (left column) and a sphere with two adjoined circles (right column) share the same Betti numbers (one component, two holes, one void) and so have the same \textit{homology}.
  However, \textit{cohomology} has a product that can tell the two spaces apart.
  Both spaces have two harmonic 1-forms (a, c) and (b, d), which measure the holes.
  However, their product is nonzero on the torus (e) but zero on the sphere-and-circles (f).}
  \label{fig:cohom product}
\end{figure}

As discussed in Section \ref{theory_cohomology_section}, we can compute the cohomology of a manifold by finding differential forms in the kernel of the Hodge Laplacian, which are known as \textit{harmonic} forms (by analogy with harmonic functions in the kernel of the Laplacian).
Unlike homology, cohomology also has a product structure (i.e. it is a ring), so we can multiply elements of cohomology together.
This makes cohomology a strictly finer invariant than homology: two spaces may share the same Betti numbers (so their homology and cohomology groups are the same size), but the \textit{products} of elements in the cohomology groups may differ.
They are indistinguishable by homology, but cohomology can tell them apart.

We can use diffusion geometry as a computational tool for cohomology by finding eigenforms of the Hodge Laplacian, whose zero (or approximately zero) eigenvalues correspond to harmonic forms.
We can then compute their wedge products (the product on cohomology) to obtain more topological information about the space.
One of the most popular tools for computational geometry and topology is \textit{persistent homology} \cite{robins1999towards, edelsbrunner2002topological}, which computes the homology (or cohomology) of data across a range of scales and represents it in a \textit{persistence diagram}.
However, even though persistent cohomology does have a product, it cannot be straightforwardly represented in the diagram\footnote{The product of points in the persistence diagram is, in general, a linear combination of other points.}.
By contrast, the wedge product of harmonic forms can be easily described in diffusion geometry.
We test this in Figure \ref{fig:cohom product} on two spaces that have the same Betti numbers, but distinct cohomology.

Recall that diffusion maps approximates the weighted Laplacian in Example \ref{eg weighted manifold}.
So, when the data are drawn from a smooth density on a submanifold of $\R^n$, we can expect diffusion geometry to estimate the Riemannian geometry of that submanifold with a metric weighted by the density.
In this case, the classical Hodge theorem applies, and so we expect to recover the cohomology of the submanifold.
In general, we can say that the cohomology groups we compute should approximate the \textit{cohomology of the support of the density}.

\subsection{Feature vectors for geometric machine learning and statistics}
\label{subsection feature vectors}

Machine learning and statistics usually require data in the form of \textit{feature vectors} that encode the necessary or interesting information.
Given complex geometric data, such as a point cloud, it is a fundamental problem to produce feature vectors that give a rich geometric description of the data with a large amount of information, are robust to noise and outliers, are fast to compute, and are explainable, so that the output of the final model can be interpreted and used to learn more about the data.
% \begin{enumerate}
%     \item give a rich geometric description of the data with a large amount of information,
%     \item are robust to noise and outliers,
%     \item are fast to compute, and
%     \item are explainable, so that the output of the final model can be interpreted and used to learn more about the data.
% \end{enumerate}

A fundamental observation of geometric deep learning \cite{bronstein2021geometric} is that geometric machine learning models must be invariant under various group actions, such as translations and rotations of the input data.
Diffusion maps, and hence diffusion geometry, works with the pairwise distances of the data (through the heat kernel), and so is naturally invariant under all isometries: translation, rotation, and reflection.
It can also be made scale-invariant through an appropriate choice of the heat kernel bandwidth.

However, group invariance alone is not enough.
Geometry is also preserved by adding noise and outliers, and, perhaps, changing the sampling density of the data.
Models should therefore be both group invariant \textit{and} statistically invariant.
The heat kernel is highly robust to noise, as we will explore, and the density-renormalisation of diffusion maps means that diffusion geometry models can be made \textit{density-invariant} as well.

We can use diffusion geometry to produce feature vectors by

\begin{enumerate}
    \item computing eigenfunctions $\phi_i$ and eigenforms $\alpha_i$ of the Laplacian and Hodge Laplacian (which encode rich geometric and topological information about the space: see Subsection \ref{theory_cohomology_section} and Figure \ref{fig:f4}),
    \item applying different combinations of the geometric operators defined in Section \ref{section_theory} to produce more functions and forms, e.g. $d\phi_2$, $H(\phi_1)(\nabla_{\alpha_1}\alpha_2, \alpha_3)$,
    \item reducing these functions and forms to single numbers by taking inner products, e.g. $\inp{\alpha_3}{d\phi_2}$, and
    \item stacking these numbers into a long vector.
\end{enumerate}

These vectors are explicitly computing (generalisations of) well-understood objects from Riemannian geometry, so give a rich geometric description and are explainable.
In the computational framework described in Section \ref{section_estimation} (and tested later in Section \ref{section_complexity}), they are also fast to compute and inherit very strong noise-robustness from the heat kernel used in the diffusion maps algorithm.
We can then use these vectors for unsupervised learning (dimensionality reduction or clustering) and supervised learning (as features for regression).

\subsubsection{Unsupervised learning: biomarkers for tumour dynamics}

We consider the problem of finding biomarkers for processes in the tumour microenvironment to demonstrate diffusion geometry's effectiveness for unsupervised representation learning.

\begin{figure}[!ht]
  \centering
  \captionsetup{width=0.66\linewidth}
  \includegraphics[width=\textwidth]{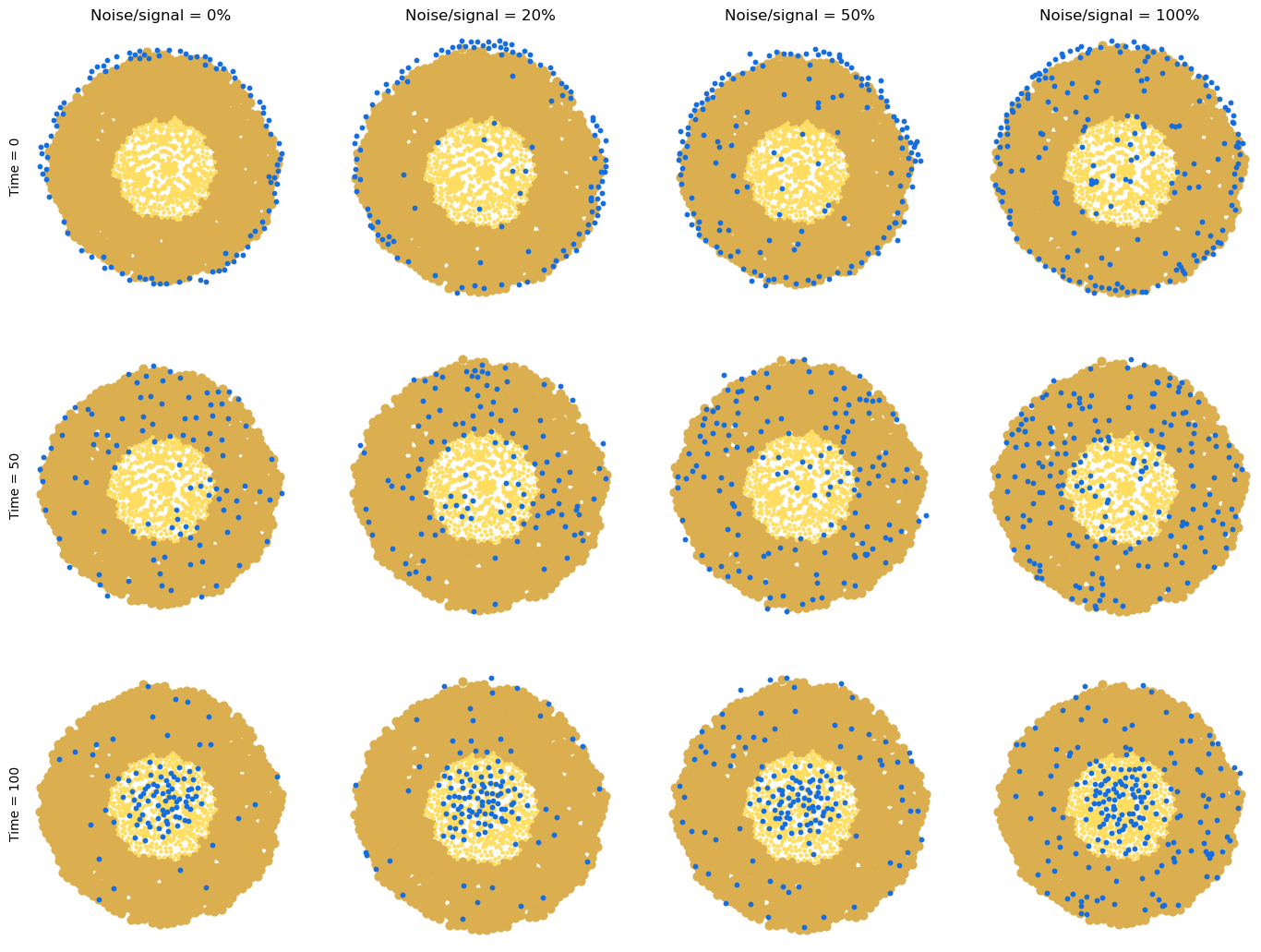}
  \caption{\textbf{Agent-based model simulation of immune cells (blue) infiltrating a tumour (yellow).}
  The cells start on the outside edge (first row) and end in the centre (third row).
  We consider data with different levels of background noise (columns), where increasing numbers of tumour cells are misclassified as immune cells.}
  \label{fig:ABM data}
\end{figure}

Figure \ref{fig:ABM data} shows agent-based model (ABM) data from \cite{Vipond}, which simulates the infiltration of immune cells into the centre of a tumour.
Cells start outside the tumour and end up in the middle.
The problem the authors of \cite{Vipond} consider is to find a statistic (called a \textit{biomarker} in this context) that measures the stage of this process and lets us track the infiltration of cells over time.
The ABMs are initialised with different \textit{chemotactic gradients} $\chi$ which change the rate of infiltration: when $\chi = 0$ the immune cells move slowly and when $\chi = 10$ they are roughly twice as fast.
A good biomarker will measure the stage of the process while clearly separating the different chemotaxis parameters.
Real data from biological and medical imaging are frequently corrupted with noise and outliers, and this is modelled by misclassifying some of the background tumour cells as immune cells.
The biomarker must still perform well even under the addition of this kind of noise.

As a benchmark, the authors of \cite{Vipond} consider biomarkers derived from persistent homology.
They compute Vietoris-Rips persistence diagrams for each collection of immune cells at each point in time with \textit{Ripser} \cite{bauer2021ripser}.
There are several standard statistics from these diagrams, and \cite{Vipond} uses the longest persistent bar in $H_1$, which measures the radius of the largest \q{hole} in the data.
We additionally consider the \q{total persistence} in $H_1$ (the sum of all the $H_1$ bars), which measures the overall amount of \q{holes} in the data\footnote{Trying other standard vectorisations of the persistence diagrams did not yield significantly different pictures than these two.}.

The results are shown in the first two rows of Figure \ref{fig:biomarkers}, where each biomarker is plotted against time.
Each ABM simulation is run five times, and we plot the means and standard deviation bars for each chemotaxis parameter and at each noise level.
The longest bar in $H_1$ (first row) shows a clear sigmoid transition curve which tracks the transition of cells from the outside to the inside of the tumour, although there is no significant separation between classes, particularly from time 60 onwards.
Conversely, total $H_1$ persistence gives better separation in the latter half, but is generally much harder to interpret.
Both persistent homology biomarkers are significantly distorted by even a small amount of noise and do not yield meaningful results for data with a noise/signal level over $50\%$.

\begin{figure}[!ht]
  \centering
  \captionsetup{width=0.66\linewidth}
  \includegraphics[width=\textwidth]{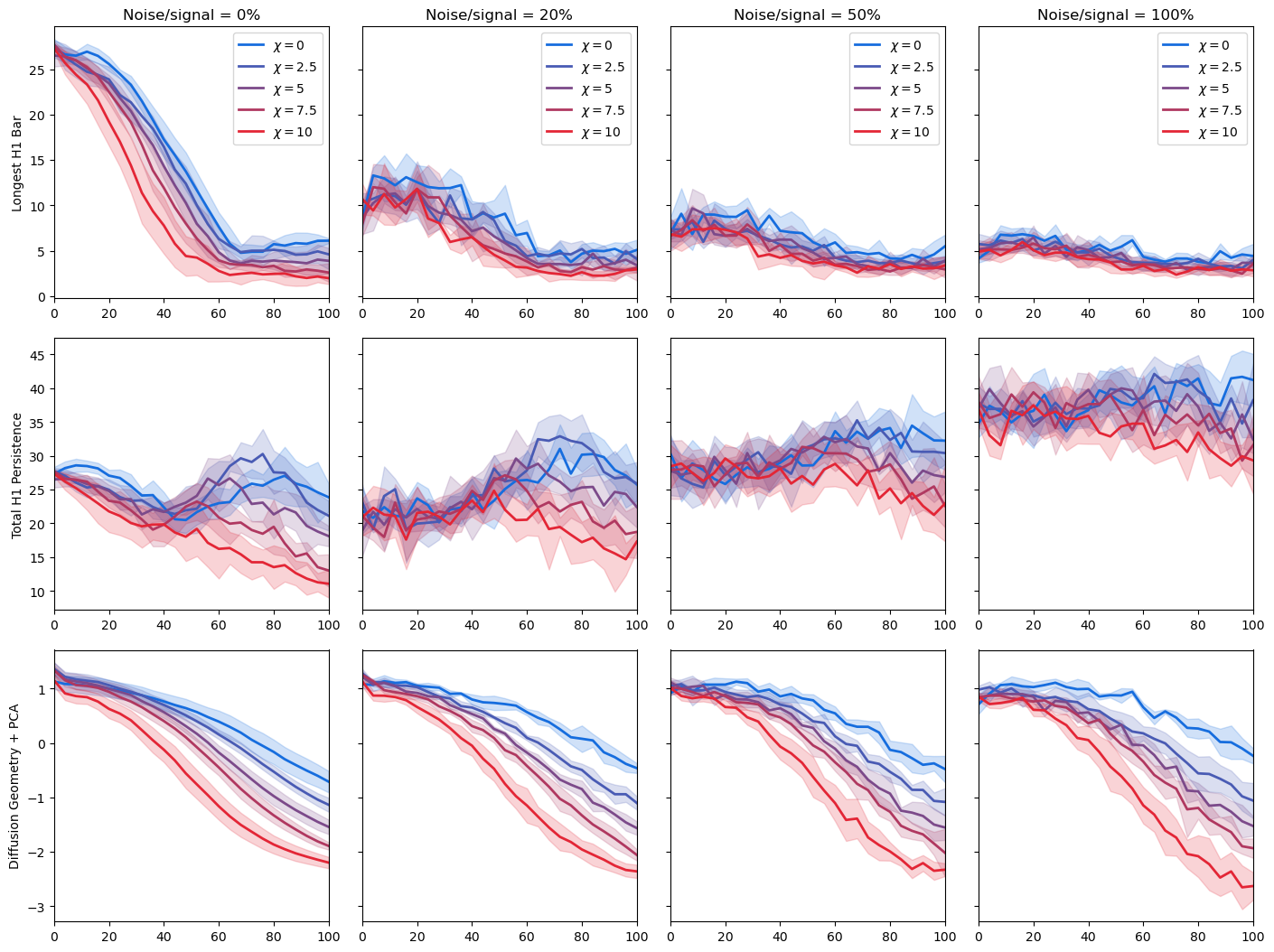}
  \caption{\textbf{Biomarkers for immune cell infiltration.} 
  Three different biomarkers (rows) are plotted over time for the four different noise levels (columns).
  The colours denote the different chemotaxis parameters.
  The persistent homology biomarkers (first two rows) are less descriptive than diffusion geometry (third row), with less chemotaxis class separation, and are significantly less robust to noise.
  }
  \label{fig:biomarkers}
\end{figure}

We compare these with a diffusion geometry biomarker.
The feature vectors described at the start of Section \ref{subsection feature vectors} encode the geometry of a point cloud in a very high-dimensional space, and we compute them for each simulation at each point in time.
We can reduce these high-dimensional vectors to a single dimension with principal component analysis (PCA) for use as a biomarker.

The results are shown in the third row of Figure \ref{fig:biomarkers}.
The diffusion geometry biomarker produces a very clear sigmoid curve as the cells transition from the outside to the inside of the tumour, with very clear chemotaxis parameter separation.
The different curves are also the same shape, except for an elongation that increases with the chemotaxis parameter, suggesting the correct intuition that the underlying process is the same but happening at different speeds.
As such, the diffusion geometry biomarker is strong enough to hint at the underlying biology without prior knowledge of the mechanism.
Crucially, it is also extremely robust to noise: even with $100\%$ noise it outperforms persistent homology with $0\%$ noise in all the above senses.

PCA assigns a weight to each of the features in the feature vector, and we can \textit{sparsify} this weight vector by setting all but the largest few entries to zero.
This produces a PCA feature that depends on only a few geometric features, which we can then interpret.
For example, we find that one of the features here is approximately equal to
$$
\Big[\lambda^0_1
- e^{-\lambda^0_1}
- e^{-10\lambda^0_1}
- e^{-50\lambda^0_1}
\Big]
+ \lambda^0_2
+ 
\Big[\lambda^1_1
- e^{-\lambda^1_1}
- e^{-10\lambda^1_1}
\Big]
$$
where $\lambda^0_i$ is the $i^{th}$ eigenvalue of the Laplacian and $\lambda^1_i$ is the $i^{th}$ eigenvalue of the Hodge Laplacian on 1-forms.
If $f$ is a function then its \textit{Dirichlet energy} $\|df\|^2$ measures how variable it is.
The eigenfunctions $\phi_1$ and $\phi_2$ always increase along the length of the data (see, for example, Figure \ref{fig:f5} (b)), and, since they are eigenfunctions of the Laplacian, the eigenvalues $\lambda_i$ are exactly the Dirichlet energies $\|d\phi_i\|^2$.
So for $\lambda_1$ and $\lambda_2$ to be small means that the data is more connected and compact.
We can also use intuition from Hodge theory (e.g. Figure \ref{fig:f4}) to interpret the Hodge Laplacian eigenvalue $\lambda^1_1$ as measuring the prominence of the \q{hole} in the middle of the data\footnote{We include some exponential eigenvalues $e^{-t\lambda_i^0}$ and $e^{-t\lambda_i^1}$ as additional data.
These are the eigenvalues of the corresponding diffusion operators $\exp(-tL)$ and $\exp(-t\Delta_1)$.}.
As such, we can interpret the decrease in this diffusion geometry feature as measuring the data becoming increasingly connected and compact, and decreasingly hollow.

\subsubsection{Supervised learning: classifying immune cell types}

We also test diffusion geometry feature vectors as a representation for supervised learning.
In the same paper \cite{Vipond}, the authors consider real histology images of slices of head and neck tumours. They use a semiautomated procedure \cite{bull2020combining} to identify the locations of three different types of immune cells: CD8, CD68, and FoxP3, and use persistent homology methods to classify the different cell types based on their spatial distribution.
Automatic cell identification is always vulnerable to the sort of misclassification noise described above, and so the authors use multiparameter persistent homology (MPH) landscapes \cite{vipond2020multiparameter} to mitigate the poor robustness of 1-parameter persistent homology to noise and outliers.

The MPH landscapes are used as feature vectors for linear discriminant analysis.
Conversely, we use the diffusion geometry features described above with logistic regression as a classifier and obtain similar or better results across all categories.
% The results are in Table \ref{tab:cell classification}.

\begin{table}[!ht]
    \centering
    \captionsetup{width=0.66\linewidth}
    \begin{tabular}{c|cc}
         & MPH landscapes & Diffusion geometry  \\
        \hline
        CD8 vs FoxP3 & 74.7\% & 88.2 $\pm$ 5.9\% \\
        CD8 vs CD68 & 65.3\% & 62.4 $\pm$ 8.2\% \\
        FoxP3 vs CD68 & 86.3\% & 90.1 $\pm$ 5.4\% \\
    \end{tabular}
    % \caption{Diffusion geometry vs multiparameter persistent homology as feature vectors for supervised learning.}
    \label{tab:cell classification}
\end{table}

\subsubsection{Other types of machine learning problem}

Diffusion geometry gives a unified framework for \q{strongly typed} machine learning problems on geometric data.
The features we compute are naturally graded into data in the form of
\begin{enumerate}
    \item[(-1)] numbers, for summaries of the whole point cloud,
    \item[(0)] functions, for segmentation-type problems,
    \item[(1)] vector fields/ 1-forms, for dynamical system and time series problems,
    \item[($k$)] $k$-forms for other analysis.
\end{enumerate}
Many of the operators defined in Section \ref{section_theory}, such as the differential and codifferential, and interior and wedge products, give the tools for moving things up and down this hierarchy.
This is important because different geometric machine learning problems are also located in this hierarchy: point cloud classification requires summary features, segmentation requires functional features, and dynamical systems are vector fields and require vector field and 1-form features.
In the feature vectors described above we just consider the first of these types, computing inner products between all the functions and forms to collapse all the data to the bottom level.
However, given a higher-order problem, we can select higher-order features and use them instead, as we will explore in future work.

\section{Computational Complexity}\label{section_complexity}

The computational framework outlined in Section \ref{section_estimation} gives us explicit control of the computational complexity (subsection \ref{subsection choosing basis}).
The only unavoidable cost is diagonalising the diffusion maps Laplacian, which has $\ord(n^3)$ complexity (where $n$ is the number of data).
After that, the functions and higher-order $k$-forms are represented by spaces of dimension $n_0$ and $n_1n_2^k$ respectively, where $n_0,n_1,n_2 \leq n$.
Crucially, we can vary $n_1$ and $n_2$ to explicitly trade off computational complexity against precision.

\begin{figure}[!ht]
  \centering
  \captionsetup{width=0.66\linewidth}
  \includegraphics[width=\textwidth]{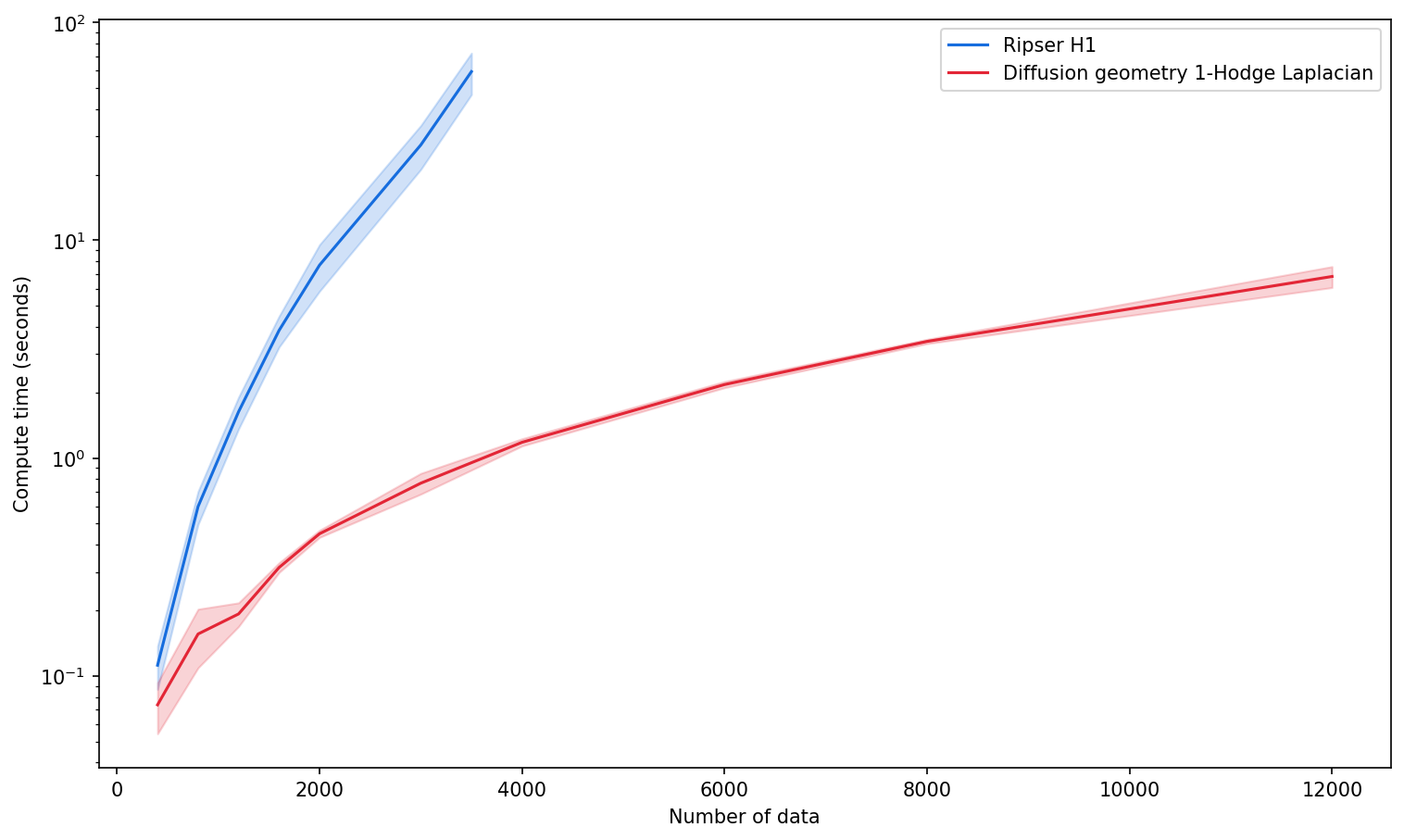}
  \caption{\textbf{Computational complexity of diffusion geometry and persistent homology.}
  Different numbers of data are sampled from a torus.
  We compute $H_1$ persistent homology with Ripser and diagonalise the Hodge Laplacian on 1-forms (the equivalent problem in diffusion geometry).
  We test each size 20 times and plot the mean and standard deviation bars on a logarithmic scale.
  Computing diffusion geometry for 12,000 points takes 6.8 seconds, while the equivalent computation of PH would take over two hours.
  }
  \label{fig:complexity}
\end{figure}

To test the consequences of this, we compare the computation of $H_1$ Vietoris-Rips persistent homology (PH) with diagonalising the Hodge Laplacian on 1-forms (the equivalent problem in diffusion geometry).
Computing $H_1$ persistent homology has $\ord(n^6)$ complexity, and diagonalising the Hodge Laplacian, with fixed $n_0 = 35$, $n_1 = 10$, $n_2 = 4$, has just $\ord(n^3)$ complexity.
We compute PH with Ripser \cite{bauer2021ripser}, which is very highly optimised.
We compute diffusion geometry with a naive implementation in NumPy \cite{harris2020array}, which is not yet optimised.
Data are sampled from a torus in $\R^3$ (the same as in Figure \ref{fig:cohom product}), and the computation is on a standard 2020 M1 Macbook Pro.
We plot the mean and standard deviation bars from 20 runs on a logarithmic scale in Figure \ref{fig:complexity}.
Computing eigenforms of the 1-Hodge Laplacian in diffusion geometry for 12,000 points takes 6.8 seconds, while the equivalent computation of PH in Ripser would take over two hours.

\section{Comparison with Related Work}\label{section_comparison}

\subsection{Theory}

The Bakry-Emery $\Gamma$-calculus discussed here was originally developed in \cite{bakry1985seminaire} to measure the subtle relationship between the geometry and the probability of diffusion operators.
Dirichlet forms like $\Gamma$ and the $\Gamma_2$ operator (which we do not discuss here) give explicit control of curvature and dimension and have a powerful regularising effect on diffusion, leading to a family of functional inequalities \cite{bakry2014analysis, ambrosio2015bakry}.
In a similar vein, the existing work on $\Gamma$-calculus explores the \q{geometry} of Markov diffusion operators as a tool for understanding their probabilistic aspects \cite{sturm2018ricci, arnold2001convex}.
The \q{nonsmooth differential geometry} programme of Gigli, Ambrosio, Savaré and others \cite{ambrosio2014calculus, gigli2018nonsmooth, gigli2020lectures, ambrosio2018calculus, ambrosio2014metric} is spiritually similar to this and explicitly constructs gradients on metric measure spaces using Sobolev weak derivatives.
The approach to \q{Hodge theory on metric spaces} in \cite{bartholdi2012hodge} gives another definition of exterior calculus on metric spaces, but is more topological and is akin to simplicial cohomology.

In this work, we have used the $\Gamma$-calculus to \textit{define} geometry in the most general setting, and our theory agrees with Riemannian geometry, nonsmooth differential geometry, and the constructions in $\Gamma$-calculus, where they apply.
Crucially, by working exclusively with $\Gamma$-calculus we obtain a natural computational model for the whole theory.

\subsection{Computation}

Our estimation of the Laplacian from data uses diffusion maps \cite{COIFMAN20065}, one of many such methods in the richly developed field of \textit{manifold learning}, and we may explore others in the future.

There are several existing approaches for computing objects from differential geometry.
Many assume the manifold hypothesis (that the data are drawn from a manifold), and so can use methods like local principal component analysis (PCA) \cite{kambhatla1997dimension} that depend heavily on that assumption.
This allows the computation of, for example, the Hessian \cite{donoho2003hessian}, the connection Laplacian \cite{singer2012vector}, and curvature \cite{yang2006robust}, but these methods break down when local PCA does (when the manifold hypothesis fails or there is too much noise).
\q{Finite-element exterior calculus} uses an explicit discretisation of a manifold and gives an overarching theory for computation \cite{arnold2006finite, arnold2010finite} in this setting.
A significant influence on this work was Berry and Giannakis' \q{spectral exterior calculus} \cite{berry2020spectral} which derived formulae for the Hodge Laplacian of 1-forms on a manifold in terms of the eigenfunctions of the Laplacian.
We use the same eigenfunction expressions as a computational framework for diffusion geometry.

There are fewer computational geometry frameworks that do not assume the manifold hypothesis, and they generally use simplicial complexes instead, with simplicial cochains playing the role of differential forms.
An overarching theory is given by \q{discrete exterior calculus} \cite{desbrun2005discrete}, and includes many of the objects we define here.
However, explicitly constructing a simplicial complex from a sample of data (such as by connecting nearest neighbours) is problematic, and generally leads to poor noise-robustness.
This is partially addressed by persistent homology (see below) but at the expense of much of the geometry.

Diffusion geometry computes objects from Riemannian geometry, so is a rich geometric measure of the data, but without assuming the manifold hypothesis, so is naturally adapted to the sorts of data encountered in real problems.

\subsection{Data analysis}

Perhaps the most popular tool for topological data analysis is persistent homology (PH) \cite{robins1999towards, edelsbrunner2002topological, zomorodian2004computing}, which tracks the changes in homology on a filtered simplicial complex.
The most popular filtrations use a varying radius around each data point to connect them up, and so PH measures the change in homology across scales.
However, these radius filtrations are not robust to noise (e.g. Figure \ref{fig:biomarkers} and \cite{turkevs2021noise}), and are not statistical estimators of well defined geometric properties of the underlying probability space\footnote{In the large-data limit, PH of radius filtrations converges to the PH of the support of the distribution, which is density-independent.}.
An alternative to radius filtration is the sublevel set filtration of the underlying probability density function, which can be estimated with a kernel \cite{bobrowski2017topological}. 
These are generally harder to compute \cite{shinconfidence} and are less widely applied because they do not measure scale.
Radius and density can be combined in a bifiltration, which results in multiparameter persistent homology (MPH) \cite{carlsson2007theory}. 
However, computing MPH is very computationally expensive \cite{otter2017roadmap, carlsson2009computing} and it has no complete \q{discrete invariants} (like a barcode for 1-parameter PH), and so is hard to quantify statistically \cite{harrington2019stratifying, blumberg2022stability}.

Diffusion geometry is highly robust to noise (Figures \ref{fig:biomarkers}, \ref{fig:manifold_hypothesis_1}, and \ref{fig:tangent_bundle_1}) and fast to compute (Figure \ref{fig:complexity}).

\section{Conclusions}

This work introduces \textit{diffusion geometry} as a framework for geometric and topological data analysis. 

We constructed a theory of Riemannian geometry on measure spaces using the Bakry-Emery $\Gamma$-calculus, which leads to a natural model for computation.
We compute the diffusion geometry of point cloud data with the \textit{diffusion maps} Laplacian, leading to computationally inexpensive and highly robust estimators for the geometry of the underlying probability distribution.

We view diffusion geometry as a statistical tool for geometric and topological data analysis, and computational geometry and topology.
We find that it outperforms existing methods in a handful of real and synthetic examples.
Diffusion geometry has broad potential for future development.

\section*{Acknowledgements}
I am extremely grateful to my supervisor Jeff Giansiracusa for pointing me in the direction that eventually became this project, and his continual support throughout.
I would also like to thank my supervisor Yue Ren, along with Fernando Galaz Garcia, Andrew Krause, David Lanners, Uzu Lim, Jerry Swan, Kelly Maggs, and Thea Stevens for their generous comments, suggestions, and advice, and Josh Bull for his permission to reproduce the ABM simulations in Section \ref{section_ml}.
This work was carried out as part of the Centre for TDA, supported by EPSRC grant EP/R018472/1. 
\newpage

\section{Appendix A: Proofs}\label{appendix proofs}

We now justify the definitions offered in Section \ref{section_theory}, and prove the results stated there.
In the following, we will make use of the following standard properties of Markov diffusion triples, which we reference from \cite{bakry2014analysis}:
\begin{enumerate}
    \item Positivity (1.4.2): $\Gamma(f,f) \geq 0$, which implies
    \item Cauchy-Schwarz inequality (1.4.3): $\Gamma(f,h)^2 \leq \Gamma(f,f) \Gamma(h,h)$.
    % \item General diffusion property (3.1.2): If $\Psi:\R^k \rightarrow \R$ is smooth and $\Psi(0) = 0$, then
    % $$
    % \Gamma(\Psi(f_1,...,f_k),h) = \sum_{i=1}^k \partial_i \Psi(f_1,...,f_k)\Gamma(f_i,h)
    % $$
    % for all $f_1,...,f_k,h \in \A$.
    \item General diffusion property (3.1.2): for all $f_1,...,f_k \in \A$ and a smooth function $\phi:\R^k\rightarrow\R$ which vanishes at 0, $\phi(f_1,...,f_k) \in \A$ and
    $$
    \Gamma(\phi(f_1,...,f_k),h) = \sum_{i=1}^k \partial_i\phi(f_1,...,f_k) \Gamma(f_i,h)
    $$
    for all $h \in \A$.
    \item Gradient bound (3.1.5): for each $f \in \A$ there is a finite constant $C(f)$ such that
    $$
    \Big| \int \Gamma(f,h)  d\mu \Big| \leq C(f) \|h\|_2
    $$
    for all $g \in \A$.
\end{enumerate}
The general diffusion property implies the standard one by setting $\phi(f_1,f_2) = f_1f_2$.
The gradient bound condition means that $\Gamma$ uniquely specifies an infinitesimal generator $L$ by the integration-by-parts formula
$$
\int h L(f) d\mu = \int \Gamma(f,h) d\mu.
$$

\subsection{Differential forms}

We first justify the definition of differential forms: the inner product is positive semi-definite on $\A \otimes \bigwedge^k \A$ and so descends to a positive definite inner product when we take the quotient by its kernel.

\positivemetric*

\begin{proof}
It is straightforward to check that $g$ is symmetric and bilinear. To show that it is also positive semi-definite, we will take some $\alpha \in \A \otimes \bigwedge^k \A$ and evaluate $g(\alpha,\alpha)(x)$ at some particular point $x$.
Since elements of the tensor product $\A \otimes \bigwedge^k \A$ are \textit{finite} linear combinations of the irreducible elements, we can write
$$
\alpha = \sum_{i=1}^N f_0^i \otimes ( f_1^i \wedge \ddd \wedge f_k^i)
$$
for some finite $N$ depending on $\alpha$. To avoid any infinite sums and the question of convergence, we define $A_\alpha$ to be the (at most) $N(k+1)$-dimensional subspace of $\A$ spanned by
$$
\{f_j^i : i = 1,...,N,\ j = 0,...,k \}.
$$
The carre du champ $\Gamma(\cdot,\cdot)(x)$ evaluated at $x$ defines a symmetric, positive semi-definite, bilinear form on $A_\alpha$, so let us take a basis for the kernel of $\Gamma(\cdot,\cdot)(x)$, and extend it to a basis for the rest of $A_\alpha$. 
We can then apply the Gram-Schmidt orthonormalisation to the positive definite basis elements to produce a basis $e_i$ for $A_\alpha$ where
\[   
\Gamma(e_i,e_j)(x) = 
     \begin{cases}
       0 & \Gamma(e_i,e_i)(x)=0  \text{ or } \Gamma(e_j,e_j)(x)=0\\
       \delta_{ij} &\quad\text{otherwise.} \\ 
     \end{cases}
\]
The $\Gamma(e_i,e_j)(x)=0$ case follows from the Cauchy-Schwarz inequality for $\Gamma$. 
So, after a change of basis, we can express $\alpha$ as
$$
\alpha = \sum_{i=1}^{N(k+1)} h^i \otimes (e_1^i \wedge \ddd \wedge e_k^i)
$$
where each $e_l^i$ is an element of the basis we just constructed, and $h^i$ is a function in $A_\alpha$. 
Using the fact that $f_1^i \wedge \ddd \wedge f_k^i$ is an element of the exterior algebra $\bigwedge^k \A$, we can ensure that the ordering of basis elements $e_l^i$ by $l$ respects the order of the basis. 
By keeping $h^i$ as an arbitrary function, and not an element of the basis, we can also assume without loss of generality that
$$
(e_1^i,...,e_k^i) \neq (e_1^j,...,e_k^j)
$$
for all $i \neq j$ (since otherwise we can combine their coefficient functions $h^i$). 
We can now write
$$
g(\alpha,\alpha)(x) = \sum_{i,j} h^i(x) h^j(x) \det\big[\big(\Gamma(e_l^i, e_m^j)(x)\big)_{l,m}\big].
$$
To simplify the sum, we first remove the terms which contain a basis function in the kernel of $\Gamma(\cdot,\cdot)(x)$. 
If the index set for $i$ is $I = \{1,...,N(k+1)\}$, we can split up $I = I_0 \sqcup I_+$ where
$$
I_0 = \{i \in I : \Gamma(e_l^i,e_l^i)(x)=0 \text{ for some } l = 1,...,k \}
$$
are the indices of those forms containing a function in the kernel of $\Gamma(\cdot,\cdot)(x)$ and
$$
I_+ = \{i \in I : \Gamma(e_l^i,e_l^i)(x)>0 \text{ for all } l = 1,...,k \}
$$
are the indices of those forms entirely comprising functions on which $\Gamma(\cdot,\cdot)(x)$ is positive. 
If $i \in I_0$ (or $j \in I_0$), then
$$
\Gamma(e^i_l, e^j_m)(x)^2 \leq \Gamma(e^i_l, e^i_l)(x)\Gamma(e^j_m, e^j_m)(x) = 0
$$
for some $l$ (or $m$) and all $m$ (or $l$), and so the matrix $\big(\Gamma(e^i_l, e^j_m)(x)\big)_{l,m}$ must contain a row (or column) of zeros, and so has zero determinant. This means that
\[   
\det\big[\big(\Gamma(e_l^i, e_m^j)(x)\big)_{l,m}\big] = 
     \begin{cases}
       0 & i \in I_0 \text{ or } j \in I_0\\
       \delta_{ij} &\quad i,j \in I_+ \\ 
     \end{cases}
\]
and so we can simplify
$$
g(\alpha,\alpha)(x) = \sum_{i \in I_+} h^i(x)^2 \geq 0.
$$
This holds for all $x$, so $g(\alpha,\alpha) \geq 0$ for all $\alpha \in \A \otimes \bigwedge^k \A$, which implies the Cauchy-Schwarz inequality for $g$.
It then follows that $\inp{\cdot}{\cdot} = \int g(\cdot, \cdot) d\mu$ is a symmetric, positive semi-definite, bilinear form on $\A \otimes \bigwedge^k \A$, and so descends to a positive definite inner product on $\Omega^k(M)$.
We also see that $g$ descends to a well defined map on the quotient, since $\|\alpha\| = 0$ if and only if $g(\alpha,\alpha) = 0$ almost everywhere, and so
$$
g(\alpha, \beta)^2 \leq g(\alpha,\alpha) g(\beta,\beta) = 0
$$
for all $\beta \in \Omega^k(M)$.
\end{proof}

\subsection{Wedge product and exterior derivative}

We verify that the wedge product is well defined and that the usual calculus rules hold for $d_0$.

\wedgewelldefined*
\begin{proof}
We will work with the same machinery as in the proof of Proposition \ref{positive_metric}. 
This time we consider two forms $\alpha$ and $\beta$, so let $A = A_\alpha \cup A_\beta$ be the finite subspace of $\A$ spanned by the functions that feature in the expansion of $\alpha$ and $\beta$. 
As before, let $\{e_i\}$ be a basis for $A$ which is orthonormal on the positive definite subspace of $\Gamma(\cdot,\cdot)(x)$. We can then expand 
$$
\alpha = \sum_i f^i \otimes (e_1^i \wedge \ddd \wedge e_k^i)
$$
and
$$
\beta = \sum_j h^j \otimes (e_1^j \wedge \ddd \wedge e_l^j)
$$
where each $e_p^i$ and $e_p^j$ is an element of the basis, the ordering of basis elements $e_p^i$ and $e_p^j$ by $p$ respects the order of the basis, and
$$
(e_1^i,...,e_k^i) \neq (e_1^{i'},...,e_k^{i'})
\qquad
(e_1^j,...,e_l^j) \neq (e_1^{j'},...,e_l^{j'})
$$
for all $i \neq i'$ and $j \neq j'$. We can then evaluate $g(\alpha\wedge\beta,\alpha\wedge\beta)(x)$ as
$$
\sum_{i,i',j,j'} f^i(x) f^{i'}(x) h^j(x) h^{j'}(x) \det\big[\big(\Gamma((e^i\wedge e^j)_p, (e^{i'}\wedge e^{j'})_q)(x)\big)_{p,q}\big],
$$
where by $(e^i\wedge e^j)_p$ we mean the $p\thupper$ element of $(e_1^i,...,e_k^i,e_1^j,...,e_l^j)$.
As before, we simplify this sum by removing the zero terms, which can now appear in two ways. 
If the index sets for $i$ and $j$ are $I$ and $J$, we decompose $I = I_0 \sqcup I_+$ and $J = J_0 \sqcup J_+$ just like in Proposition \ref{positive_metric}. 
Now suppose that $i,j,i',j' \in I_+$, and neither $e^i$ and $e^j$ nor $e^{i'}$ and $e^{j'}$ share a common basis function. 
Let $\pi$ and $\pi'$ be the permutations of $e^i\wedge e^j$ and $e^{i'}\wedge e^{j'}$ into their basis orders, which are unique since the functions are all distinct. 
Then
\begin{equation*}
    \begin{split}
        \det\big[\big(\Gamma((e^i\wedge e^j)_p, (e^{i'}\wedge e^{j'})_q)(x)\big)_{p,q}\big]
        &= \delta_{ii'}\delta_{jj'} (-1)^{\sign(\pi)} (-1)^{\sign(\pi')} \\
        &= \delta_{ii'}\delta_{jj'} \big((-1)^{\sign(\pi)}\big)^2 \\
        &= \delta_{ii'}\delta_{jj'}
    \end{split}
\end{equation*}
since, if $i=i'$ and $j = j'$, then $\pi = \pi'$. 
If either $e^i$ and $e^j$ or $e^{i'}$ and $e^{j'}$ do share a common function, suppose without loss of generality that it is shared by both $e^i$ and $e^j$. 
Let $\pi$ be a permutation of $e^i\wedge e^j$ into its basis order. 
The repeated function means there exists a transposition $\Tilde{\pi}$ that fixes $\pi(e^i\wedge e^j)$, so
\begin{equation*}
    \begin{split}
        &\det\big[\big(\Gamma((e^i\wedge e^j)_p, (e^{i'}\wedge e^{j'})_q)(x)\big)_{p,q}\big]\\
        =\ & (-1)^{\sign(\pi)}\det\big[\big(\Gamma((\pi(e^i\wedge e^j))_p, (e^{i'}\wedge e^{j'})_q)(x)\big)_{p,q}\big] \\
        =\ & (-1)^{\sign(\Tilde{\pi})}(-1)^{\sign(\pi)}\det\big[\big(\Gamma((\Tilde{\pi}(\pi(e^i\wedge e^j)))_p, (e^{i'}\wedge e^{j'})_q)(x)\big)_{p,q}\big] \\
        =\ & -(-1)^{\sign(\pi)}\det\big[\big(\Gamma((\pi(e^i\wedge e^j))_p, (e^{i'}\wedge e^{j'})_q)(x)\big)_{p,q}\big] \\
        =\ & -\det\big[\big(\Gamma((e^i\wedge e^j)_p, (e^{i'}\wedge e^{j'})_q)(x)\big)_{p,q}\big]
    \end{split}
\end{equation*}
and hence
$$
\det\big[\big(\Gamma((e^i\wedge e^j)_p, (e^{i'}\wedge e^{j'})_q)(x)\big)_{p,q}\big] = 0.
$$
Let $S_0$ be the subset of indices $(i,j) \in I_+ \times J_+$ such that $e^i$ and $e^j$ share a common basis function, and let $S_+ = (I_+\times J_+) \setminus S_0$. 
Then, using the same argument as in Proposition \ref{positive_metric} for the indices in $I_0$ and $J_0$, we can conclude that
\[   
\det\big[\big(\Gamma((e^i\wedge e^j)_p, (e^{i'}\wedge e^{j'})_q)(x)\big)_{p,q}\big] = 
     \begin{cases}
       0 & i \in I_0 \text{ or } j \in I_0 \text{ or } i' \in I_0 \text{ or } j' \in I_0\\
       0 &\quad (i,j) \in S_0 \text{ or } (i',j') \in S_0 \\ 
       \delta_{ii'}\delta_{jj'} &\quad (i,j) \in S_+ \text{ and } (i',j') \in S_+ \\ 
     \end{cases}
\]
and simplify
\begin{equation*}
    \begin{split}
        g(\alpha\wedge\beta,\alpha\wedge\beta)(x)
        &= \sum_{(i,j) \in S_+} \big(f^i(x) h^j(x) \big)^2 \\
        &\leq \sum_{i \in I_+, j \in J_+} \big(f^i(x) h^j(x) \big)^2 \\
        &= \sum_{i \in I_+} f^i(x)^2 \sum_{j \in J_+} h^j(x)^2 \\
        &= g(\alpha,\alpha)(x)g(\beta,\beta)(x).
    \end{split}
\end{equation*}
Notice that, as on a manifold, there is equality at $x$ if and only if $S_0 = \varnothing$, i.e. the spans of the 1-forms comprising $\alpha$ and $\beta$ are orthogonal.
We then have, for fixed $\beta$,
$$
\int g(\alpha\wedge\beta,\alpha\wedge\beta) d\mu
\leq \int g(\alpha,\alpha)g(\beta,\beta) d\mu
\leq \sup(g(\beta,\beta)) \int g(\alpha,\alpha) d\mu,
$$
so $\|\alpha\wedge\beta\| \leq \sqrt{\sup(g(\beta,\beta))} \|\alpha\|$, and likewise for $
\beta$. 
In particular, if $\|\alpha\| = 0$ or $\|\beta\| = 0$ then $\|\alpha\wedge\beta\| = 0$, so the wedge product is well defined on $\Omega^k(M) \times \Omega^l(M)$, and is a bounded linear operator in each argument.
\end{proof}

\calculusrulesd*
\begin{proof}
The general diffusion property implies that
\begin{equation*}
\begin{split}
\inp{d\phi(f_1,...,f_k)}{h'dh}
&= \int h'\Gamma(\phi(f_1,...,f_k),h) d\mu \\
&= \sum_{i=1}^k \int \partial_i\phi(f_1,...,f_k) h' \Gamma(f_i,h) d\mu \\
&= \sum_{i=1}^k \inp{\partial_i\phi(f_1,...,f_k) df_i}{h'dh}.
\end{split}
\end{equation*}
The chain rule then follows from the non-degeneracy of the inner product and the fact that terms $h'dh$ span $\Omega^1(M)$. The Leibniz rule follows by setting $\phi(f_1,f_2) = f_1f_2$.
\end{proof}

\subsection{First-order calculus: vector fields and duality}

The interior product is also well defined.

\interiorproductprop*
\begin{proof}
It is clear that condition (3) in the definition is compatible with (1) and (2), and that $i_\alpha(\beta)$ is linear in $\alpha$ and $\beta$ for $k = 0,1$. 
Now take $\beta = f_0 df_1 \wedge \ddd \wedge df_k \in \Omega^k(M)$. 
By definition, we must have
$$
i_\alpha(\beta) = i_\alpha(f_0 df_1 \wedge \ddd \wedge df_{k-1}) \wedge df_k - (-1)^k g(\alpha,df_k) f_0 df_1 \wedge \ddd \wedge df_{k-1}.
$$
Since $f_0 df_1 \wedge \ddd \wedge df_{k-1} \in \Omega^{k-1}(M)$, this expression gives an inductive definition for $i_\alpha(\beta)$, which we extend linearly to all $\beta \in \Omega^k(M)$. Linearity in $\alpha$ also follows by the same induction on $k$.

Notice that the antisymmetry property holds vacuously for $k = 0, 1$, since $i_\gamma i_\alpha(\beta) = i_\alpha i_\gamma (\beta) = 0$, but in general
\begin{equation*}
\begin{split}
i_\gamma i_\alpha(\beta)
&= i_\gamma i_\alpha(f_0 df_1 \wedge \ddd \wedge df_{k-1}) \wedge df_k \\
& \qquad + (-1)^k g(\gamma,df_k) i_\alpha(f_0 df_1 \wedge \ddd \wedge df_{k-1}) \\
& \qquad - (-1)^k g(\alpha,df_k) i_\gamma \big( f_0 df_1 \wedge \ddd \wedge df_{k-1} \big).
\end{split}
\end{equation*}
We can inductively assume that $i_\alpha i_\gamma$ is antisymmetric on $\Omega^{k-1}(M)$, and obtain
\begin{equation*}
\begin{split}
i_\alpha i_\gamma(\beta)
&= -i_\gamma i_\alpha(f_0 df_1 \wedge \ddd \wedge df_{k-1}) \wedge df_k \\
& \qquad + (-1)^k g(\alpha,df_k) i_\gamma \big( f_0 df_1 \wedge \ddd \wedge df_{k-1} \big) \\
& \qquad - (-1)^k g(\gamma,df_k) i_\alpha(f_0 df_1 \wedge \ddd \wedge df_{k-1}) \\
&= -i_\gamma i_\alpha(\beta)
\end{split}
\end{equation*}
which proves the result.
\end{proof}

\subsection{Second-order calculus: Hessian, covariant derivative, and Lie bracket}

To motivate the definition of the Hessian (Definition \ref{hessian def}), we derive the corresponding formula for the Hessian on a manifold.
We will use two standard facts from the Riemannian geometry of manifolds: that the Hessian is given by
\begin{equation}
\label{manifold hess def}
H(f)(X,Y) = g(\nabla_X(\nabla f), Y)
\end{equation}
and the Levi-Civita connection satisfies the Koszul formula
\begin{equation}
\label{koszul formula manifold}
\begin{split}
g(\nabla_X Y,Z) = \nabla Y (X,Z) = \frac{1}{2}\Big(
& X(g(Y,Z)) + Y(g(Z,X)) - Z(g(X,Y)) \\
&+ g([X,Y],Z) - g([Y,Z],X) + g([Z,X],Y) \Big).
\end{split}
\end{equation}
To avoid confusion we will stick with the carré du champ notation, so here $\Gamma(f,h) = g(\nabla f,\nabla h) = \nabla f(h)$.

\begin{prop}\label{manifold hessian}
The Hessian on a manifold satisfies
$$
H(f)(\nabla a,\nabla b) = \frac{1}{2}\big( \Gamma(a, \Gamma(f,b)) + \Gamma(b, \Gamma(f,a)) - \Gamma(f, \Gamma(a,b)) \big).
$$
\end{prop}
\begin{proof}
We directly invoke (\ref{manifold hess def}) and (\ref{koszul formula manifold}) to compute
\begin{equation*}
\begin{split}
H(f)(\nabla a,\nabla b)
&= g(\nabla_{\nabla a}(\nabla f), \nabla b) \\
&= \frac{1}{2}\Big( \nabla a(g(\nabla f,\nabla b)) + \nabla f(g(\nabla b,\nabla a)) - \nabla b(g(\nabla a,\nabla f)) \\
&\qquad + g([\nabla a,\nabla f],\nabla b) - g([\nabla f,\nabla b],\nabla a) + g([\nabla b,\nabla a],\nabla f) \Big)
\end{split}
\end{equation*}
and note that
\begin{equation*}
\begin{split}
g([\nabla a,\nabla f],\nabla b)
&= [\nabla a,\nabla f](b) \\
&= \nabla a (\nabla f(b)) - \nabla f (\nabla a(b)) \\
&= \Gamma(a, \Gamma(f,b)) - \Gamma(f, \Gamma(a,b)) \\
\end{split}
\end{equation*}
so obtain
\begin{equation*}
\begin{split}
H(f)(\nabla a,\nabla b)
&= \frac{1}{2}\Big( \Gamma(a, \Gamma(f,b)) + \Gamma(f, \Gamma(a,b)) - \Gamma(b, \Gamma(f,a)) \\
&\qquad + \Gamma(a, \Gamma(f,b)) - \Gamma(f, \Gamma(a,b)) - \Gamma(f, \Gamma(b,a)) \\
&\qquad + \Gamma(b, \Gamma(f,a)) + \Gamma(b, \Gamma(f,a)) - \Gamma(a, \Gamma(f,b)) \Big) \\
&= \frac{1}{2}\big( \Gamma(a, \Gamma(f,b)) + \Gamma(b, \Gamma(f,a)) - \Gamma(f, \Gamma(a,b)) \big)
\end{split}
\end{equation*}
for all $f,a,b$.
\end{proof}

With the definition for the Hessian offered here, we can derive the usual calculus rules.

\hessianleibnizrule*
\begin{proof}
We can compute
\begin{equation*}
\begin{split}
H(fh)(\nabla a, \nabla a)
&= \Gamma(a,\Gamma(fh,a)) - \frac{1}{2}\Gamma(fh,\Gamma(a,a)) \\
&= \Gamma(a,f\Gamma(h,a) + h\Gamma(f,a)) - \frac{1}{2}f\Gamma(h,\Gamma(a,a)) - \frac{1}{2}h\Gamma(f,\Gamma(a,a)) \\
&= f\Gamma(a,\Gamma(h,a)) + \Gamma(f,a)\Gamma(h,a) + h\Gamma(a,\Gamma(f,a)) + \Gamma(f,a)\Gamma(h,a) \\
& \qquad - \frac{1}{2}f\Gamma(h,\Gamma(a,a)) - \frac{1}{2}h\Gamma(f,\Gamma(a,a)) \\
&= f\big[ \Gamma(a,\Gamma(h,a)) - \frac{1}{2}\Gamma(h,\Gamma(a,a)) \big] + h\big[ \Gamma(a,\Gamma(f,a)) - \frac{1}{2}\Gamma(f,\Gamma(a,a)) \big] + 2\Gamma(f,a)\Gamma(h,a) \\
&= \Big[f H(h) + h H(f) + 2 (df \otimes dh) \Big] (\nabla a, \nabla a)
\end{split}
\end{equation*}
from which the result follows by polarisation and linearity.
\end{proof}

\hessianproductrule*
\begin{proof}
We can compute
\begin{equation*}
\begin{split}
H(f_1)(\nabla f_2, \nabla h) + H(f_2)(\nabla f_1, \nabla h)
&= \frac{1}{2} \Big( \Gamma(f_2, \Gamma(f_1,h)) + \Gamma(h, \Gamma(f_1,f_2)) - \Gamma(f_1, \Gamma(f_2,h)) \\
& \qquad \Gamma(f_1, \Gamma(f_2,h)) + \Gamma(h, \Gamma(f_1,f_2)) - \Gamma(f_2, \Gamma(f_1,h)) \Big)\\
&= \Gamma(h, \Gamma(f_1,f_2)) \\
&= d\big(\Gamma(f_1,f_2)\big)(\nabla h) \\
\end{split}
\end{equation*}
and extend $\nabla h$ to an arbitrary $X \in \mathfrak{X}(M)$ by linearity.
\end{proof}

We can derive the following formula for evaluating the Hessian.

\hessianevalformulae*
\begin{proof}
Suppose first that $X = a_0\nabla a_1$ and $Y = b_0 \nabla b_1$. We evalutate
\begin{equation*}
\begin{split}
g(X, [Y,\nabla f])
&= g( a_0\nabla a_1, [b_0 \nabla b_1,\nabla f]) \\
&= a_0 [b_0 \nabla b_1,\nabla f](a_1) \\
&= a_0 \big( b_0 [\nabla b_1,\nabla f] - \Gamma(f,b_0)\nabla b_1 \big) (a_1) \\
&= a_0 b_0 \Gamma\big( b_1, \Gamma(f, a_1) \big) - a_0b_0 \Gamma\big(f, \Gamma(a_1, b_1) \big) - a_0\Gamma(f,b_0)\Gamma(a_1, b_1)
\end{split}
\end{equation*}
and
$$
\Gamma\big(f, a_0b_0\Gamma(a_1,b_1)\big)
= a_0b_0\Gamma\big(f, \Gamma(a_1,b_1)\big)
+ a_0 \Gamma(f, b_0) \Gamma(a_1,b_1) 
+ b_0 \Gamma(f, a_0) \Gamma(a_1,b_1) 
$$
so
\begin{equation*}
\begin{split}
&\frac{1}{2}\Big(g(X, [Y,\nabla f]) + g(Y,[X,\nabla f]) + \Gamma(f, g(X,Y)) \Big) \\
=\ & \frac{1}{2}\Big[a_0 b_0 \Gamma\big( b_1, \Gamma(f, a_1) \big) - a_0b_0 \Gamma\big(f, \Gamma(a_1, b_1) \big) - a_0\Gamma(f,b_0)\Gamma(a_1, b_1) \\
&\qquad + a_0 b_0 \Gamma\big( a_1, \Gamma(f, b_1) \big) - a_0b_0 \Gamma\big(f, \Gamma(a_1, b_1) \big) - b_0\Gamma(f,a_0)\Gamma(a_1, b_1) \\
&\qquad + a_0b_0\Gamma\big(f, \Gamma(a_1,b_1)\big) + a_0 \Gamma(f, b_0) \Gamma(a_1,b_1) + b_0 \Gamma(f, a_0) \Gamma(a_1,b_1)  \Big] \\
=\ & \frac{1}{2}a_0 b_0\Big(\Gamma\big( a_1, \Gamma(f, b_1) \big) + \Gamma\big( b_1, \Gamma(f, a_1) \big) - \Gamma\big(f, \Gamma(a_1,b_1)\big) \Big) \\
=\ & a_0 b_0 H(f)(\nabla a_1,\nabla b_1) \\
=\ & H(f)(X,Y) \\
\end{split}
\end{equation*}
from which the result follows by linearity in $X$ and $Y$.
\end{proof}

We also check that the covariant derivative is well defined and satisfies the standard properties of the Levi-Civita connection.

\covariantwelldefined*
\begin{proof}
First notice that $\nabla X$ is indeed in $\Omega^1(M)^{\otimes2}$ because $h_i \in W^{2,\infty}(M)$ for all $i$, and so $H(h_i) \in \Omega^1(M)^{\otimes2}$.
We need to show that, if $\|X\| = 0$, then $\|\nabla X\| = 0$.
If $X = f\nabla h$, we can apply Proposition \ref{hessian product rule} to get
$$
d\big(f\Gamma(h,b)\big) = \Gamma(h,b) df + f H(h)(\nabla b, \cdot)  + f H(b)(\nabla h, \cdot)
$$
and so, by acting on $\nabla a$,
\begin{equation*}
\begin{split}
\Gamma \big(g(X,\nabla b), a\big) 
&= \Gamma(f,a)\Gamma(h,b) + f H(h)(\nabla b, \nabla a)  + f H(b)(\nabla h, \nabla a) \\
&= \nabla X(\nabla a, \nabla b) + H(b)(X, \nabla a), \\
\end{split}
\end{equation*}
which extends to all $X \in \mathfrak{X}^1(M)$ by linearity.
It then follows that
\begin{equation*}
\begin{split}
\inp{\nabla X}{f' da \otimes db}
&= \int f' \nabla X(\nabla a, \nabla b) d\mu \\
&= \int f' \Big[ \Gamma \big(g(X,\nabla b), a\big) - H(b)(X, \nabla a)\Big] d\mu.
\end{split}
\end{equation*}
If $\|X\| = 0$ then $g(X,X) = 0$ almost everywhere, so $g(X,\nabla b) = 0$ by Cauchy-Schwarz.
Likewise
\begin{equation*}
\begin{split}
H(b)(X, \nabla a)^2
&= g(H(b), X^\flat \otimes da)^2 \\
&\leq g(H(b),H(b)) g(X^\flat \otimes da, X^\flat \otimes da) \\
&= g(H(b),H(b)) g(X,X) \Gamma(a,a) \\
&= 0
\end{split}
\end{equation*}
so $\inp{\nabla X}{f' da \otimes db} = 0$.
The terms $f' da \otimes db$ span $\Omega^1(M)^{\otimes2}$, so we must have $\|\nabla X\| = 0$, and the map $X \mapsto \nabla X$ is well defined.
\end{proof}

\nablaaffineconnection*
\begin{proof}
The linearity claims follow directly from the definition. To prove the Leibniz rule, in the case $X = f' \nabla h$, we have
\begin{equation*}
\begin{split}
\nabla(fX)
&= d (ff') \otimes d h + ff' H(h_i) \\
&= \big(fd(f') + f'd(f)\big) \otimes d h + ff' H(h_i) \\
&= df \otimes X^\flat + f\nabla X,
\end{split}
\end{equation*}
from which the result follows by linearity.
\end{proof}

\nablapreservesmetric*
\begin{proof}
We need to check that $X(g(Y,Z)) = \nabla Y(X,Z) + \nabla Z(X,Y)$. If we set $Y = f_1 \nabla h_1$ and $Z = f_2 \nabla h_2$, we can evaluate
\begin{equation*}
\begin{split}
X\big( g(Y,Z) \big)
&= X\big(f_1f_2 \Gamma(h_1, h_2) \big) \\
&= X(f_1f_2) \Gamma(h_1, h_2) + f_1f_2 X \big( \Gamma(h_1,h_2) \big) \\
&= X(f_1)f_2 \Gamma(h_1, h_2) + f_1f_2 H(h_2)(X,\nabla h_2) \\
&\qquad + X(f_2)f_1 \Gamma(h_1, h_2) + f_1f_2 H(h_1)(X,\nabla h_1) \\
&= X(f_1)Z(h_1) + f_1 H(h_2)(X,Z) \\
&\qquad + X(f_2)Y(h_2) + f_2 H(h_1)(X,Y) \\
&= \nabla Y(X,Z) + \nabla Z(X,Y)
\end{split}
\end{equation*}
using the Leibniz rule for $X$ twice and the product rule for $\Gamma$ (Proposition \ref{hessian product rule}). The general case follows by linearity in $Y$ and $Z$.
\end{proof}

We also check that the Lie bracket is indeed the commutator of operators $[X,Y] = XY - YX$ on $\mathfrak{X}^1(M)$, and satisfies the Leibniz rule.

\liebracketcommutator*
\begin{proof}
We first consider $Y = f \nabla h$ and compute
\begin{equation*}
\begin{split}
X(Y(a))
&= X(f \Gamma(h,a)) \\
&= X(f) \Gamma(h,a) + f X(\Gamma(h,a)) \\
&= X(f) \Gamma(h,a) + f\big( H(h)(\nabla a, X) + H(a)(\nabla h, X) \big) \\
&= X(f) \Gamma(h,a) + f H(h)(\nabla a, X) + H(a)(Y, X)
\end{split}
\end{equation*}
using the product rule for $\Gamma$ (Proposition \ref{hessian product rule}).
Now notice that
\begin{equation*}
\begin{split}
\nabla_X Y (a)
&= g(\nabla_X Y, \nabla a) \\
&= \nabla Y (X, \nabla a) \\
&= X(f) \Gamma(h,a) + f H(h)(\nabla a, X)
\end{split}
\end{equation*}
so
$$
X(Y(a)) = \nabla_X Y (a) + H(a)(X, Y),
$$
which extends to arbitrary $Y$ by linearity.
We then see that
$$
X(Y(a)) - Y(X(a)) = [X,Y](a)
$$
by symmetry of $H(a)$.
\end{proof}

\liebracketleibniz*
\begin{proof}
We directly compute
\begin{equation*}
\begin{split}
[X,fY]
&= \nabla_{X}(fY) - \nabla_{fY}X \\
&= \nabla (fY) (X, \cdot)^\sharp - \nabla(X)(fY,\cdot)^\sharp \\
&= (df\otimes Y + f\nabla Y) (X, \cdot)^\sharp - f\nabla(X)(Y,\cdot)^\sharp \\
&= X(f)Y + f\nabla (Y) (X, \cdot)^\sharp - f\nabla(X)(Y,\cdot)^\sharp \\
&= X(f)Y + f[X,Y] \\
\end{split}
\end{equation*}
and likewise for $[fX,Y].$
\end{proof}

\subsection{Exterior derivative and codifferential}

We verify that the exterior derivative and codifferential are well defined on higher-order forms, and that the differential satisfies the Leibniz rule.

\extderivativeinvariantformula*
\begin{proof}
Let $T_\alpha:\mathfrak{X}(M)^{k+1} \rightarrow \A$ be given by
$$
T_\alpha(X_0,...,X_k) = \sum_{i=0}^k (-1)^i X_i(\alpha(...,\hat{X_i},...)) + \sum_{i<j} (-1)^{i+j} \alpha([X_i, X_j],...,\hat{X_i},...,\hat{X_j},...),
$$
so want to show that $T_\alpha = d\alpha$. 
We first show that, if $\beta = dh_1\wedge\ddd\wedge dh_k$, then $T_\beta = 0$.
Let $X_0,...,X_k \in \mathfrak{X}(M)$, and notice that $\beta(X_1,...,X_k) = \det\big(X_i(h_j)\big)$ so that
$$
\beta(...,\hat{X_i},...) = \sum_{\sigma \in S_k} \sign(\sigma) \prod_{l\neq i} X_l(h_{\sigma(a^i(l))})
$$
where
$$
a^i(l) =
\begin{cases}
l+1 & l < i \\
l & l > i
\end{cases}
$$
for $l = 0,...,i-1,i+1,...,k$.
We can compute
\begin{equation*}
\begin{split}
X_i(\beta(...,\hat{X_i},...))
&= X_i \big( \sum_{\sigma \in S_k} \sign(\sigma) \prod_{l\neq i} X_l(h_{\sigma(a^i(l))}) \big) \\
&= \sum_{\sigma \in S_k} \sign(\sigma) \sum_{j\neq i} X_i X_j (h_{\sigma(a^i(j))}) \prod_{l\neq i,j} X_l (h_{\sigma(a^i(l))}) \big), \\
\end{split}
\end{equation*}
with the Leibniz rule, and so
$$
\sum_{i=0}^k (-1)^i X_i(\beta(...,\hat{X_i},...)) = \sum_{i\neq j} \sum_{\sigma \in S_k} \sign(\sigma)(-1)^i X_i X_j (h_{\sigma(a^i(j))}) \prod_{l\neq i,j} X_l (h_{\sigma(a^i(l))}) \big).
$$

On the other hand, we can see that
$$
\beta([X_i, X_j],...,\hat{X_i},...,\hat{X_j},...) = \sum_{\sigma \in S_k} \sign(\sigma) [X_i, X_j](h_{\sigma(1)}) \prod_{l\neq i,j} X_l(h_{\sigma(b^{ij}(l))})
$$
where now
$$
b^{ij}(l) =
\begin{cases}
l+2 & l < \min(i,j) \\
l+1 & \min(i,j) < l < \max(i,j) \\
l & \max(i,j) < l \\
\end{cases}
$$
for $l \neq i,j$.
We then find
\begin{equation*}
\begin{split}
\sum_{i<j} (-1)^{i+j} \beta([X_i, X_j],...,\hat{X_i},...,\hat{X_j},...)
&= \sum_{i < j} (-1)^{i+j} \sum_{\sigma \in S_k} \sign(\sigma) [X_i, X_j](h_{\sigma(1)}) \prod_{l\neq i,j} X_l(h_{\sigma(b^{ij}(l))}) \\
&= \sum_{i < j} (-1)^{i+j} \sum_{\sigma \in S_k} \sign(\sigma) X_i X_j(h_{\sigma(1)}) \prod_{l\neq i,j} X_l(h_{\sigma(b^{ij}(l))}) \\
& \qquad - \sum_{i < j} (-1)^{i+j} \sum_{\sigma \in S_k} \sign(\sigma) X_j X_i(h_{\sigma(1)}) \prod_{l\neq i,j} X_l(h_{\sigma(b^{ij}(l))}) \\
&= \sum_{i < j} (-1)^{i+j} \sum_{\sigma \in S_k} \sign(\sigma) X_i X_j(h_{\sigma(1)}) \prod_{l\neq i,j} X_l(h_{\sigma(b^{ij}(l))}) \\
& \qquad - \sum_{i > j} (-1)^{i+j} \sum_{\sigma \in S_k} \sign(\sigma) X_i X_j (h_{\sigma(1)}) \prod_{l\neq i,j} X_l(h_{\sigma(b^{ij}(l))}). \\
\end{split}
\end{equation*}
We now want to reindex the sum over $S_k$, and do so separately for the cases $i<j$ and $i>j$.
If $i<j$, we define $\tau^{ij} \in S_k$ as
$$
\tau^{ij} = (j,1,2,...,j-1,j+1,...,k)
$$
which has sign $(-1)^j$.
If $i>j$, we define $\tau^{ij} \in S_k$ as
$$
\tau^{ij} = (j+1,1,2,...,j,j+2,...,k)
$$
which has sign $(-1)^{j+1}$.
An elementary but tedious calculation shows that, in both cases, $\tau^{ij}(1) = a^i(j)$ and $\tau^{ij}(b^{ij}(l)) = a^i(l)$.
$\tau^{ij}$ has a regular right action on $S_k$, so we can replace $\sigma$ with $\sigma\tau^{ij}$ in the summand, and obtain
\begin{equation*}
\begin{split}
\sum_{i<j} (-1)^{i+j} \beta([X_i, X_j],...,\hat{X_i},...,\hat{X_j},...)
&= \sum_{i < j} (-1)^{i+j} \sum_{\sigma \in S_k} \sign(\sigma\tau^{ij}) X_i X_j(h_{\sigma\tau^{ij}(1)}) \prod_{l\neq i,j} X_l(h_{\sigma\tau^{ij}(b^{ij}(l))}) \\
& \qquad - \sum_{i > j} (-1)^{i+j} \sum_{\sigma \in S_k} \sign(\sigma\tau^{ij}) X_i X_j(h_{\sigma\tau^{ij}(1)}) \prod_{l\neq i,j} X_l(h_{\sigma\tau^{ij}(b^{ij}(l))}) \\
&= \sum_{i < j} (-1)^{i+j} \sum_{\sigma \in S_k} \sign(\sigma)(-1)^{j + 1} X_i X_j(h_{\sigma(a^i(j))}) \prod_{l\neq i,j} X_l(h_{\sigma(a^i(l))}) \\
& \qquad - \sum_{i > j} (-1)^{i+j} \sum_{\sigma \in S_k} \sign(\sigma)(-1)^{j} X_i X_j(h_{\sigma(a^i(j))}) \prod_{l\neq i,j} X_l(h_{\sigma(a^i(l))}) \\
&= - \sum_{i \neq j} \sum_{\sigma \in S_k} \sign(\sigma)(-1)^{i} X_i X_j(h_{\sigma(a^i(j))}) \prod_{l\neq i,j} X_l(h_{\sigma(a^i(l))}) \\
&= - \sum_{i=0}^k (-1)^i X_i(\beta(...,\hat{X_i},...))
\end{split}
\end{equation*}
so $T_\beta = 0$.

We now let $\alpha = f\beta$ where $\beta = dh_1\wedge\ddd\wedge dh_k$, and find
\begin{equation*}
\begin{split}
T_\alpha(X_0,...,X_k) 
&= \sum_{i=0}^k (-1)^i X_i(f\beta(...,\hat{X_i},...)) + \sum_{i<j} (-1)^{i+j} f\beta([X_i, X_j],...,\hat{X_i},...,\hat{X_j},...) \\
&= \sum_{i=0}^k (-1)^i X_i(f)\beta(...,\hat{X_i},...) + f T_\beta(X_0,...,X_k) \\
&= df \wedge dh_1\wedge\ddd\wedge dh_k (X_0,...,X_k),
\end{split}
\end{equation*}
where the last step follows by expanding the determinant along the first column.
By linearity we see that $
T_\alpha(X_0,...,X_k) = d\alpha(X_0,...,X_k)$ for all $\alpha \in \Omega^k(M)$.

Now suppose that $\|\alpha\| = 0$, so $g(\alpha,\alpha)=0$ almost everywhere.
Then
\begin{equation*}
\begin{split}
\alpha(X_0,...,X_k)^2
&= g(\alpha, X_0^\flat \wedge \dots \wedge X_k^\flat)^2 \\
&\leq g(\alpha,\alpha) g(X_0^\flat \wedge \dots \wedge X_k^\flat, X_0^\flat \wedge \dots \wedge X_k^\flat) \\
&= 0
\end{split}
\end{equation*}
almost everywhere.
$T_\alpha(X_0,...,X_k) = 0$ where $\alpha(X_0,...,X_k) = 0$ so
\begin{equation*}
\begin{split}
\inp{d\alpha}{fdh_0\wedge\dots\wedge dh_k}
&= \int fd\alpha(\nabla h_0,...,\nabla h_k) d\mu \\
&= \int fT_\alpha(\nabla h_0,...,\nabla h_k) d\mu  \\
&= 0.
\end{split}
\end{equation*}
It follows by linearity that $\inp{d\alpha}{\beta} = 0$ for all $\beta\in\Omega^{k+1}(M)$, so $\|d\alpha\| = 0$.
This means that the map $\A \otimes \bigwedge^k(\A) \rightarrow \A \otimes \bigwedge^{k+1}(\A)$ given by
$$
fdh_1\wedge\dots\wedge dh_k \mapsto 1 \otimes df \wedge dh_1\wedge\dots\wedge dh_k
$$
descends to the quotient (by zero-norm tensors), and so $\alpha \mapsto d\alpha$ is a well defined map $\Omega^k(M)\rightarrow \Omega^{k+1}(M)$.
\end{proof}

\leibnizforextderivative*
\begin{proof}
We dealt with the case where $f, h \in \Omega^0(M) = \A$ in Proposition \ref{calculus_rules_d0}. 
If we now take $\alpha = f_0 df_1 \wedge \ddd \wedge df_k \in \Omega^k(M)$ and $\beta = h_0 dh_1 \wedge \ddd \wedge dh_l \in \Omega^l(M)$, we can calculate
\begin{equation*}
    \begin{split}
        d(\alpha \wedge \beta)
        &= d(f_0h_0) \wedge df_1 \wedge \ddd \wedge df_k \wedge dh_1 \wedge \ddd \wedge dh_l \\
        &= \big(h_0df_0 + f_0dh_0\big) \wedge df_1 \wedge \ddd \wedge df_k \wedge dh_1 \wedge \ddd \wedge dh_l \\
        &= h_0 df_0 \wedge df_1 \wedge \ddd \wedge df_k \wedge dh_1 \wedge \ddd \wedge dh_l \\
        & \quad + f_0 dh_0 \wedge df_1 \wedge \ddd \wedge df_k \wedge dh_1 \wedge \ddd \wedge dh_l \\
        &= \big( df_0 \wedge df_1 \wedge \ddd \wedge df_k \big) \wedge \big( h_0 dh_1 \wedge \ddd \wedge dh_l \big) \\
        & \quad + (-1)^k \big(f_0 df_1 \wedge \ddd \wedge df_k \big) \wedge \big( dh_0 \wedge dh_1 \wedge \ddd \wedge dh_l \big) \\
        &= d\alpha \wedge \beta + (-1)^k \alpha \wedge d\beta.
    \end{split}
\end{equation*}
The case for general $\alpha \in \Omega^k(\mu)$ and $\beta \in \Omega^l(\mu)$ follows by linearity of $\wedge$ and $d$.
\end{proof}

\codiffwelldefined*
\begin{proof}
For brevity, we will write
$$
\hat{dh_i} = dh_1 \wedge \dots \wedge \hat{dh_i} \wedge \dots \wedge dh_k
$$
and
$$
\hat{dh_i} \wedge \hat{dh_j} = dh_1 \wedge \dots \wedge \hat{dh_i} \wedge \dots \wedge \hat{dh_j} \wedge \dots \wedge dh_k.
$$
We consider a general form $a\beta \in \Omega^{k-1}(M)$, where $a \in \A$ and $\beta = dh_1' \wedge \dots \wedge dh_{k-1}'$.
We find that
\begin{equation*}
\begin{split}
\inp{\big( \Gamma(f, h_i) - f L(h_i) \big) \hat{dh_i}}{a \beta}
&= \int \Big[ \Gamma(f, h_i) g(\hat{dh_i}, a \beta) - f L(h_i) g(\hat{dh_i}, a \beta) \Big] d\mu \\
&= \int \Big[ \Gamma(f, h_i) a g(\hat{dh_i}, \beta) - \Gamma \big( h_i, fa g(\hat{dh_i}, \beta) \big) \Big] d\mu \\
&= - \int f \Gamma \big( h_i, ag(\hat{dh_i}, \beta) \big) d\mu \\
&= - \int fa \Gamma \big( h_i, g(\hat{dh_i}, \beta) \big) d\mu - \int f \Gamma (h_i, a) g(\hat{dh_i}, \beta) d\mu.
\end{split}
\end{equation*}
Notice that
\begin{equation*}
\begin{split}
- \int f \Big[ \sum_{i=1}^k (-1)^i \Gamma (h_i, a) g(\hat{dh_i}, \beta) \Big] d\mu 
& = \int f g(dh_1\wedge\dots\wedge dh_k, da \wedge \beta) d\mu \\
&= \inp{f dh_1\wedge\dots\wedge dh_k}{d(a\beta)}
\end{split}
\end{equation*}
by expanding the first column of the determinant, so
\begin{equation*}
\begin{split}
\inp{\partial(fdh_1\wedge\dots\wedge dh_k)}{a\beta}
&= \sum_{i=1}^k (-1)^i \inp{\big( \Gamma(f, h_i) - f L(h_i) \big) \hat{dh_i}}{a \beta} \\
& \qquad + \sum_{i<j} (-1)^{i+j} \inp{f [\nabla h_i, \nabla h_j]^\flat \wedge \hat{dh_i} \wedge \hat{dh_j}}{a\beta} \\
&= \inp{f dh_1\wedge\dots\wedge dh_k}{d(a\beta)} \\
& \qquad - \int fa \Big[ 
\sum_{i=1}^k (-1)^i \Gamma \big( h_i, g(\hat{dh_i},\beta) \big) \\
& \qquad \qquad - \sum_{i<j} (-1)^{i+j} g \big( [\nabla h_i, \nabla h_j]^\flat \wedge \hat{dh_i} \wedge \hat{dh_j}, \beta \big)
\Big] d\mu
\end{split}
\end{equation*}
We can apply Proposition \ref{ext derivative invariant formula} to find that
\begin{equation*}
\begin{split}
& \sum_{i=1}^k (-1)^i \Gamma \big( h_i, g(\hat{dh_i},\beta) \big) 
- \sum_{i<j} (-1)^{i+j} g \big( [\nabla h_i, \nabla h_j]^\flat \wedge \hat{dh_i} \wedge \hat{dh_j}, \beta \big) \\
=&\
\sum_{i=1}^k (-1)^i \nabla h_i(\beta(...,\hat{\nabla h_i},...)) - \sum_{i<j} (-1)^{i+j} \beta([\nabla h_i, \nabla h_j],...,\hat{\nabla h_i},...,\hat{\nabla h_j},...) \\
=&\
-d\beta (\nabla h_1,...,\nabla h_k) = 0
\end{split}
\end{equation*}
because $\beta$ is exact (hence closed), which proves the claim.
Now if $\|\alpha\| = 0$ then 
$$
\inp{\partial \alpha}{\beta} = \inp{\alpha}{d \beta} = 0
$$
for all $\beta \in \Omega^{k-1}(M)$, giving $\|\partial \alpha\| = 0$.
So the map $\A \otimes \bigwedge^k(\A) \rightarrow \A \otimes \bigwedge^{k-1}(\A)$ given by $\alpha \mapsto \partial\alpha$ descends to the quotient (by zero-norm tensors), meaning $\alpha \mapsto \partial\alpha$ is a well defined map $\Omega^k(M)\rightarrow \Omega^{k-1}(M)$.
\end{proof}

\newpage

\section{Appendix B: Implementation Details}\label{appendix implementation}

We implement the diffusion maps Laplacian exactly as in \cite{COIFMAN20065}, and outlined in Algorithm \ref{alg:DM}.
All the basic objects in diffusion geometry can be expressed with tensor operations on the eigenvalues and eigenfunctions of $\Delta$. These are easily calculated with the Einstein summation functions in NumPy, and we write the following derivations in that form.

All the forms here are represented in the eigenfunction frame discussed in Section \ref{section_estimation}, so functions are in the basis $\p{i}$, and 1-forms in the frame $\p{i}d\p{j}$ etc.
Where needed, we will project forms into an orthonormal basis for the positive definite subspace of the inner product.
However, it is more convenient here to work in the original eigenfunction frame, and so we will always include forms back in the frame after working in the orthonormal basis.

\subsection{Structure constants for the multiplicative algebra}
\label{structure constants appendix}
$$
c_{ijk} = \inp{\p{i} \p{j}}{\p{k}} = \frac{1}{n}\sum_{s = 1}^n \p{i}(s) \p{j}(s) \p{k}(s) D(s)
$$
Notice that, since $\p{0}$ is constant, we have
$$
c_{ij0} = \p{0} \inp{\p{i}}{\p{j}} = \p{0}\delta_{ij}.
$$
We will let $c_{i_1 \dots i_k}$ denote the corresponding $k$-fold products.

\subsection{Carré du champ}
\label{cdc appendix}
$$
\Gamma_{ijs} = \inp{\Gamma(\p{i}, \p{j})}{\p{s}} = \frac{1}{2}(\lambda_i + \lambda_j -  \lambda_s) c_{ijs}
$$

% \subsection{Metric on 1-forms}
% \begin{equation*}
% \begin{split}
% g(\p{i}d\p{j}, \p{k}d\p{l})
% &= \p{i}\p{k}\Gamma(\p{j}, \p{l}) \\
% &= \Big( \sum_u c_{iku} \p{u} \Big) \Big(\sum_v \Gamma_{jlv} \p{v} \Big) \\
% &= \sum_{u,v,w} c_{iku} \Gamma_{jlv} c_{uvw} \p{w} \\
% \end{split}
% \end{equation*}
% and so
% $$
% g^1_{ijkls} = \inp{g(\p{i}d\p{j}, \p{k}d\p{l})}{\p{s}} = \sum_{u,v} c_{iku} \Gamma_{jlv} c_{uvs}.
% $$

\subsection{Metric on k-forms}

To illustrate the general formula for $g^k$, we derive it here for $k=2$. Let us write the 2-forms
$$
\alpha_I = \p{i_0} d\p{i_1} \wedge d\p{i_2}
\qquad
\alpha_J = \p{j_0} d\p{j_1} \wedge d\p{j_2}
$$
where $I = (i_0, i_1, i_2)$ and $J = (j_0, j_1, j_2)$. Then, using Einstein summation notation,
\begin{equation*}
\begin{split}
g(\alpha_I,\alpha_J)
&= \p{i_0}\p{j_0} \det 
\begin{bmatrix}
\Gamma(\p{i_1}, \p{j_1}) & \Gamma(\p{i_1}, \p{j_2})\\
\Gamma(\p{i_2}, \p{j_1}) & \Gamma(\p{i_2}, \p{j_2})
\end{bmatrix}
\\
&= \p{i_0}\p{j_0} \det 
\begin{bmatrix}
\Gamma_{i_1 j_1 s}\p{s} & \Gamma_{i_1 j_2 t}\p{t}\\
\Gamma_{i_2 j_1 s}\p{s} & \Gamma_{i_2 j_2 t}\p{t}
\end{bmatrix}
\\
&= \p{i_0}\p{j_0}\p{s}\p{t} \det 
\begin{bmatrix}
\Gamma_{i_1 j_1 s} & \Gamma_{i_1 j_2 t}\\
\Gamma_{i_2 j_1 s} & \Gamma_{i_2 j_2 t}
\end{bmatrix}
\\
&= c_{i_0j_0st\ell} \det 
\begin{bmatrix}
\Gamma_{i_1 j_1 s} & \Gamma_{i_1 j_2 t}\\
\Gamma_{i_2 j_1 s} & \Gamma_{i_2 j_2 t}
\end{bmatrix}
 \p{\ell}
\\
&= c_{i_0j_0u}c_{stu\ell} \det 
\begin{bmatrix}
\Gamma_{i_1 j_1 s} & \Gamma_{i_1 j_2 t}\\
\Gamma_{i_2 j_1 s} & \Gamma_{i_2 j_2 t}
\end{bmatrix}
 \p{\ell}
\end{split}
\end{equation*}
where we are allowed to reuse the dummy indices $s,t$ because two entries in the same column are not multiplied in the determinant.
The general formula is derived analogously. Let $I = (i_0, ..., i_k)$ and $J = (j_0, ..., j_k)$, and take dummy indices $s_1,...,s_k$. Then
$$
g(\alpha_I,\alpha_J) = c_{i_0j_0s_1\ddd s_k \ell} \det \big((\Gamma_{i_n j_m s_m})_{n,m}\big) \p{\ell}.
$$
In the special case $k=1$, we find the formula
$$
g(\p{i}d\p{j}, \p{k}d\p{l}) = c_{ikst} \Gamma_{jls} \p{t},
$$
and when $k=0$ we retrieve $g(\p{i},\p{j}) = \p{i}\p{j} = c_{ijt} \p{t}$.

% \subsection{Inner product on 1-forms (Gram matrix)}

% \begin{equation*}
% \begin{split}
% G^1_{ijkl}
% &= \inp{\p{i}d\p{j}}{\p{k}d\p{l}} \\
% &= \int \p{i}\p{k} \Gamma(\p{j}, \p{l}) d\mu \\
% &= \inp{\p{i}\p{k}}{\Gamma(\p{j}, \p{l})} \\
% &= \inp{\sum_s c_{iks} \p{s}}{\sum_t \Gamma_{jlt}\p{t}} \\
% &= \sum_{s} c_{iks} \Gamma_{jls}. \\
% \end{split}
% \end{equation*}

\subsection{Inner product on k-forms (Gram matrix)} \label{gram matrix forms appendix}

To illustrate the general formula for $G^k$, we derive it here for $k=2$. Let us write the 2-forms
$$
\alpha_I = \p{i_0} d\p{i_1} \wedge d\p{i_2}
\qquad
\alpha_J = \p{j_0} d\p{j_1} \wedge d\p{j_2}
$$
where $I = (i_0, i_1, i_2)$ and $J = (j_0, j_1, j_2)$. Then, using Einstein summation notation,
\begin{equation*}
\begin{split}
\inp{\alpha_I}{\alpha_J}
&= \int \p{i_0}\p{j_0} \det 
\begin{bmatrix}
\Gamma(\p{i_1}, \p{j_1}) & \Gamma(\p{i_1}, \p{j_2})\\
\Gamma(\p{i_2}, \p{j_1}) & \Gamma(\p{i_2}, \p{j_2})
\end{bmatrix}
d\mu \\
&= \int \p{i_0}\p{j_0} \det 
\begin{bmatrix}
\Gamma_{i_1 j_1 s}\p{s} & \Gamma_{i_1 j_2 t}\p{t}\\
\Gamma_{i_2 j_1 s}\p{s} & \Gamma_{i_2 j_2 t}\p{t}
\end{bmatrix}
d\mu \\
&= \Big( \int \p{i_0}\p{j_0}\p{s}\p{t} d\mu \Big) \det 
\begin{bmatrix}
\Gamma_{i_1 j_1 s} & \Gamma_{i_1 j_2 t}\\
\Gamma_{i_2 j_1 s} & \Gamma_{i_2 j_2 t}
\end{bmatrix}
\\
&= c_{i_0j_0st} \det 
\begin{bmatrix}
\Gamma_{i_1 j_1 s} & \Gamma_{i_1 j_2 t}\\
\Gamma_{i_2 j_1 s} & \Gamma_{i_2 j_2 t}
\end{bmatrix}
\end{split}
\end{equation*}
where we are allowed to reuse the dummy indices $s,t$ because two entries in the same column are not multiplied in the determinant. Recall here that
$$
\Gamma_{ijs} = \frac{1}{2}(\lambda_i + \lambda_j -  \lambda_s) c_{ijs}.
$$
The general formula is derived analogously. Let $I = (i_0, ..., i_k)$ and $J = (j_0, ..., j_k)$, and take dummy indices $s_1,...,s_k$. Then
$$
\inp{\alpha_I}{\alpha_J} = c_{i_0j_0s_1\ddd s_k} \det \big((\Gamma_{i_n j_m s_m})_{n,m}\big).
$$
In the special case $k=1$, we obtain the formula
$$
G^1_{ijkl} = \inp{\p{i}d\p{j}}{\p{k}d\p{l}} = c_{iks} \Gamma_{jls},
$$
and when $k=0$ we retrieve $\inp{\p{i}}{\p{j}} = c_{ij} = \delta_{ij}$.

\subsection{Pseudo-inverse of the Gram matrix.}

A common computational problem here is the recovery of a form $v$ from its \q{weak} representation $Gv$. 
For example, we might have some expression for $\inp{w}{v} = w^TGv$ for all $w$ and would like to know $v$.
The Gram matrix we compute above is generally degenerate, and so is not invertible: the solution is to take the Moore-Penrose pseudo-inverse of $G$, where we
\begin{enumerate}
    \item project $v$ onto the positive definite subspace of $G$,
    \item invert $G$ on this subspace and compute $G\inv v$, and
    \item include $G\inv v$ in the total space.
\end{enumerate}
So while we cannot truly invert $G$ in the frame, we can invert it up to a form of zero norm.

We can obtain this representation by diagonalising $G$, and restricting to the eigenspaces with positive eigenvalue (in practice we take eigenvalues above some variable threshold, to control numerical instability).
We denote the pseudoinverse of $G$ by $G^+$ and can express it as
$$
G^+ = P^T (PGP^T)\inv P
$$
where $P$ is the orthonormal projection matrix onto the positive definite subspace of $G$. 
$P$ was obtained by diagonalising $G$, so the matrix $(PGP^T)$ is diagonal and easily inverted. 
We can verify that, if $w = P^Tv$ is in the positive definite subspace of $G$, then
$$
G^+Gw = P^T (PGP^T)\inv PGP^Tv = w
$$
and so we can recover $w$ from $Gw$ with $G^+$. Likewise
\begin{equation*}
\begin{split}
GG^+w &= GP^T (PGP^T)^{-1} PP^Tv \\
&= P^TP GP^T (PGP^T)^{-1} v \\
&= w. \\
\end{split}
\end{equation*}
using the facts that $PP^T = Id$ and $G = P^T D P = P^TPP^T D P = P^TPG$.

\subsection{Dual vector fields}
$$
\inp{\p{k}}{(\p{i}d\p{j})^\sharp \p{l}} = \int g(\p{i}d\p{j}, d\p{l}) \p{k} d\mu = G^1_{ijkl}
$$
So, if $v$ is a 1-form, $G^1v$ gives the coefficients of $v^\sharp$ as a linear operator. 
By reshaping $G^1v$ into a matrix we can compute its action on functions. 
In the following, we will write $G^1v$ to denote this matrix and write $\circ$ to denote matrix multiplication (i.e. composition of vector fields).

\subsection{Dual 1-forms}
By duality with the above, if $X$ is a vector field represented by $X_{ij} = \inp{\p{i}}{X(\p{j})}$, then $X = G^1 X^\flat$, and so we can recover $X^\flat = (G^1)^+X$.

\subsection{Wedge product}

Let $\alpha_I \in \Omega^k(M)$ and $\alpha_J \in \Omega^l(M)$ be given by
$$
\alpha_I = \p{i_0} d\p{i_1} \wedge\dots\wedge d\p{i_k}
\qquad
\alpha_J = \p{j_0} d\p{j_1} \wedge\dots\wedge d\p{j_l}
$$
where $I = (i_0, ..., i_k)$ and $J = (j_0, j_1, j_2)$. 
Then, using Einstein summation notation,
\begin{equation*}
\begin{split}
\alpha_I \wedge \alpha_J
&= (\p{i_0}\p{j_0}) d\p{i_1} \wedge\dots\wedge d\p{i_k} \wedge d\p{j_1} \wedge\dots\wedge d\p{j_l} \\
&= c_{i_0j_0s}\p{s} d\p{i_1} \wedge\dots\wedge d\p{i_k} \wedge d\p{j_1} \wedge\dots\wedge d\p{j_l}. \\
\end{split}
\end{equation*}

% \subsection{Derivative and Hodge Laplacian}

\subsection{Exterior derivative}
We can represent the exterior derivative weakly in the matrix
$$
(\tilde d_k)_{ij} = \inp{d_k(v_j)}{w_i}
$$
and recover the strong formulation as the matrix
$$
d_k = (G^k)^+(\tilde d_k).
$$
To illustrate the general formula, we derive it here for $k=1$. Let us write the 2 and 1-forms
$$
\alpha_I = \p{i_0} d\p{i_1} \wedge d\p{i_2}
\qquad
\alpha_J = \p{j_0} d\p{j_1}
$$
where $I = (i_0, i_1, i_2)$ and $J = (j_0, j_1)$. Then, using Einstein summation notation,
\begin{equation*}
\begin{split}
\inp{\alpha_I}{d \alpha_J}
&= \int \p{i_0} \det 
\begin{bmatrix}
\Gamma(\p{i_1}, \p{j_0}) & \Gamma(\p{i_1}, \p{j_1})\\
\Gamma(\p{i_2}, \p{j_0}) & \Gamma(\p{i_2}, \p{j_1})
\end{bmatrix}
d\mu \\
&= \int \p{i_0} \det 
\begin{bmatrix}
\Gamma_{i_1 j_0 s}\p{s} & \Gamma_{i_1 j_1 t}\p{t}\\
\Gamma_{i_2 j_0 s}\p{s} & \Gamma_{i_2 j_1 t}\p{t}
\end{bmatrix}
d\mu \\
&= \Big( \int \p{i_0}\p{s}\p{t} d\mu \Big) \det 
\begin{bmatrix}
\Gamma_{i_1 j_0 s} & \Gamma_{i_1 j_1 t}\\
\Gamma_{i_2 j_0 s} & \Gamma_{i_2 j_1 t}
\end{bmatrix}
\\
&= c_{i_0st} \det 
\begin{bmatrix}
\Gamma_{i_1 j_0 s} & \Gamma_{i_1 j_1 t}\\
\Gamma_{i_2 j_0 s} & \Gamma_{i_2 j_1 t}
\end{bmatrix}
\end{split}
\end{equation*}
where we are allowed to reuse the dummy indices $s,t$ because two entries in the same column are not multiplied in the determinant. The general formula is derived analogously. Let $I = (i_0, ..., i_{k+1})$ and $J = (j_0, ..., j_k)$, and take dummy indices $s_1,...,s_k$. Then
$$
\inp{\alpha_I}{d\alpha_J} = c_{i_0s_1\ddd s_k} \det \big((\Gamma_{i_{n+1} j_m s_m})_{n,m}\big).
$$
In the special case $k=0$, we obtain the formula
$$
\inp{\p{i_0}d\p{i_1}}{d\p{j}} =\Gamma_{i_1ji_0}.
$$
We can also recover the weak and strong forms of $d_k^*$ from $\tilde d_k$ since
$$
(\tilde d_k)_{ij} = \inp{v_j}{d_k^*(w_i)} = (\tilde d_k^*)_{ji}
$$
and so $\tilde d_k^* = \tilde d_k^T$ and $d_k^* = (G^{k-1})^+ \tilde d_k^T$.

\subsection{Hodge Laplacian}\label{appendix hodge laplacian}

We will represent the Hodge Laplacian $\Delta_k$ weakly as
$$
(\Tilde{\Delta_k})_{ij} = \inp{w_i}{\Delta_k(v_j)} = \inp{d_k(w_i)}{d_k(v_j)} + \inp{w_i}{d_{k-1}d^*_{k-1}(v_j)}.
$$
The \q{down} term $\inp{w_i}{d_{k-1}d^*_{k-1}(v_j)}$ is the weak formulation of the operator $d_{k-1}d^*_{k-1}$, which we can evaluate using $\tilde d_{k-1}$ as
$$
G^k d_{k-1} d_{k-1}^* = \tilde d_{k-1} (G^{k-1})^+ \tilde d_{k-1}^T.
$$
Although a similar process would work for the \q{up} term, this would involve the pseudo-inverse of $G^{k+1}$, which is more computationally complex. We can instead use the \q{kernel trick} and directly evaluate $\inp{d_k(w_i)}{d_k(v_j)}$ using the metric to get
$$
\inp{d_k(w_i)}{d_k(v_j)} = \int g(d_k(v_j), d_k(v_j)) d\mu.
$$
We can derive an explicit formula for this \q{up} energy in terms of the eigenfunction frame, which we illustrate here with the case $k=2$. Let us write the 2-forms
$$
\alpha_I = \p{i_0} d\p{i_1} \wedge d\p{i_2}
\qquad
\alpha_J = \p{j_0} d\p{j_1} \wedge d\p{j_2}
$$
where $I = (i_0, i_1, i_2)$ and $J = (j_0, j_1, j_2)$. Then, using Einstein summation notation,
\begin{equation*}
\begin{split}
\inp{d \alpha_I}{d\alpha_J}
&= \int \det 
\begin{bmatrix}
\Gamma(\p{i_0}, \p{j_0}) & \Gamma(\p{i_0}, \p{j_1}) & \Gamma(\p{i_0}, \p{j_2})\\
\Gamma(\p{i_1}, \p{j_0}) & \Gamma(\p{i_1}, \p{j_1}) & \Gamma(\p{i_1}, \p{j_2})\\
\Gamma(\p{i_2}, \p{j_0}) & \Gamma(\p{i_2}, \p{j_1}) & \Gamma(\p{i_2}, \p{j_2})\\
\end{bmatrix}
d\mu \\
&= \int \det 
\begin{bmatrix}
\Gamma_{i_0 j_0 s}\p{s} & \Gamma_{i_0 j_1 t}\p{t} & \Gamma_{i_0 j_2 u}\p{u}\\
\Gamma_{i_1 j_0 s}\p{s} & \Gamma_{i_1 j_1 t}\p{t} & \Gamma_{i_1 j_2 u}\p{u}\\
\Gamma_{i_2 j_0 s}\p{s} & \Gamma_{i_2 j_1 t}\p{t} & \Gamma_{i_2 j_2 u}\p{u}\\
\end{bmatrix}
d\mu \\
&= \Big( \int \p{s}\p{t}\p{u} d\mu \Big) \det 
\begin{bmatrix}
\Gamma_{i_0 j_0 s} & \Gamma_{i_0 j_1 t} & \Gamma_{i_0 j_2 u}\\
\Gamma_{i_1 j_0 s} & \Gamma_{i_1 j_1 t} & \Gamma_{i_1 j_2 u}\\
\Gamma_{i_2 j_0 s} & \Gamma_{i_2 j_1 t} & \Gamma_{i_2 j_2 u}\\
\end{bmatrix}
\\
&= c_{stu} \det 
\begin{bmatrix}
\Gamma_{i_0 j_0 s} & \Gamma_{i_0 j_1 t} & \Gamma_{i_0 j_2 u}\\
\Gamma_{i_1 j_0 s} & \Gamma_{i_1 j_1 t} & \Gamma_{i_1 j_2 u}\\
\Gamma_{i_2 j_0 s} & \Gamma_{i_2 j_1 t} & \Gamma_{i_2 j_2 u}\\
\end{bmatrix}
\end{split}
\end{equation*}
where we are allowed to reuse the dummy indices $s,t,u$ because two entries in the same column are not multiplied in the determinant. Recall here that
$$
\Gamma_{ijs} = \frac{1}{2}(\lambda_i + \lambda_j -  \lambda_s) c_{ijs}.
$$
The general formula is derived analogously. Let $I = (i_0, ..., i_k)$ and $J = (j_0, ..., j_k)$, and take dummy indices $S = (s_0,...,s_k)$. Then
$$
\inp{d \alpha_I}{d\alpha_J} = c_S \det \big((\Gamma_{i_n j_m s_m})_{n,m}\big).
$$
In the special case $k=1$, we have $c_{pq} = \delta_{pq}$ and so
$$
\inp{d\p{i_0}\wedge d\p{i_1}}{d\p{j_0}\wedge d\p{j_1}} =
\det 
\begin{bmatrix}
\Gamma_{i_0 j_0 s} & \Gamma_{i_0 j_1 s}\\
\Gamma_{i_1 j_0 s} & \Gamma_{i_1 j_1 s}\\
\end{bmatrix}
$$
In the very special case $k=0$, we have
$$
c_s = \int \p{s} d\mu = \frac{1}{\p{0}}\int \p{0}\p{s} d\mu = \frac{1}{\p{0}}\delta_{s0}
$$
and so
$$
\inp{d\p{i}}{d\p{j}} = \frac{1}{\p{0}} \Gamma_{i j 0}.
$$
Now notice that
$$
c_{ij0} = \int \p{i}\p{j}\p{0} d\mu = \p{0} \delta_{ij}
$$
so
$$
\Gamma_{ij0} = \frac{1}{2}(\lambda_i + \lambda_j -  \lambda_0) c_{ij0} = \p{0} \lambda_i \delta_{ij}
$$
giving $\inp{d\p{i}}{d\p{j}} = \lambda_i \delta_{ij}$. In particular, we recover that the Laplacian $\Delta$ is represented by a diagonal matrix with entries $\lambda_i$.

\subsection{Hodge Decomposition}
On a manifold, a 1-form $\alpha \in \Omega^1(M)$ can be decomposed uniquely as
$$
\alpha = h + df + d^*\beta
$$
where $h \in \ker \Delta$, $f\in\A$, and $\beta \in \Omega^2(M)$. In our case, the Hodge Laplacian on 1-forms does not \textit{actually} have a kernel, due to approximation error, and so the \q{harmonic} bit $h$ shows up in the $d^*\beta$ part. Although we have not proved this result for general diffusion Markov triples, the Hodge decomposition theorem can be shown (fairly straightforwardly) for finite-dimensional spaces, with which we are working here.

To compute the \q{gradient} part $df$, we can recover the function $f$ as
$$
f = \Delta^+ d^* \alpha,
$$
where $\Delta^+$ is the pseudo-inverse of $\Delta$. In our eigenfunction basis, this is given by
\[   
(\Delta^+f)_i = 
     \begin{cases}
       0 & \lambda_i = 0\\
       f_i / \lambda_i & \lambda_i > 0\\
     \end{cases}
\]
We can therefore write the projection onto the \q{gradient-only} part of $\Omega^1(M)$ as $d \Delta^+ d^*$.

\subsection{Lie bracket}
\begin{equation*}
\begin{split}
[v, w] 
&= \big( v^\sharp \circ w^\sharp - w^\sharp \circ v^\sharp \big)^\flat \\
&= (G^1)^+\big( G^1v \circ G^1 w - G^1w \circ G^1 v \big)
\end{split}
\end{equation*}

\subsection{Covariant derivative}
We use the Koszul formula
\begin{equation*}
\begin{split}
\inp{\nabla_X Y}{Z} = \frac{1}{2}\Big(
        & \frac{1}{\p{0}} \inp{X(g(Y,Z))}{\p{0}} + \frac{1}{\p{0}} \inp{Y(g(Z,X))}{\p{0}} - \frac{1}{\p{0}} \inp{Z(g(X,Y))}{\p{0}} \\
        &+ \inp{[X,Y]}{Z} - \inp{[Y,Z]}{X} + g\inp{[Z,X]}{Y} \Big).
\end{split}
\end{equation*}

\bibliographystyle{plain}
\bibliography{main}

\begin{thebibliography}{10}

\bibitem{ambrosio2018calculus}
Luigi Ambrosio.
\newblock Calculus, heat flow and curvature-dimension bounds in metric measure spaces.
\newblock In {\em Proceedings of the International Congress of Mathematicians: Rio de Janeiro 2018}, pages 301--340. World Scientific, 2018.

\bibitem{ambrosio2014calculus}
Luigi Ambrosio, Nicola Gigli, and Giuseppe Savar{\'e}.
\newblock Calculus and heat flow in metric measure spaces and applications to spaces with ricci bounds from below.
\newblock {\em Inventiones mathematicae}, 195(2):289--391, 2014.

\bibitem{ambrosio2014metric}
Luigi Ambrosio, Nicola Gigli, and Giuseppe Savar{\'e}.
\newblock Metric measure spaces with riemannian ricci curvature bounded from below.
\newblock 2014.

\bibitem{ambrosio2015bakry}
Luigi Ambrosio, Nicola Gigli, and Giuseppe Savar{\'e}.
\newblock Bakry--{\'e}mery curvature-dimension condition and riemannian ricci curvature bounds.
\newblock 2015.

\bibitem{arnold2001convex}
Anton Arnold, Peter Markowich, Giuseppe Toscani, and Andreas Unterreiter.
\newblock On convex sobolev inequalities and the rate of convergence to equilibrium for fokker-planck type equations.
\newblock 2001.

\bibitem{arnold2010finite}
Douglas Arnold, Richard Falk, and Ragnar Winther.
\newblock Finite element exterior calculus: from hodge theory to numerical stability.
\newblock {\em Bulletin of the American mathematical society}, 47(2):281--354, 2010.

\bibitem{arnold2006finite}
Douglas~N Arnold, Richard~S Falk, and Ragnar Winther.
\newblock Finite element exterior calculus, homological techniques, and applications.
\newblock {\em Acta numerica}, 15:1--155, 2006.

\bibitem{bakry1985seminaire}
Dominique Bakry and M~{\'E}mery.
\newblock S{\'e}minaire de probabilit{\'e}s xix 1983/84, 1985.

\bibitem{bakry2014analysis}
Dominique Bakry, Ivan Gentil, Michel Ledoux, et~al.
\newblock {\em Analysis and geometry of Markov diffusion operators}, volume 103.
\newblock Springer, 2014.

\bibitem{bartholdi2012hodge}
Laurent Bartholdi, Thomas Schick, Nat Smale, and Steve Smale.
\newblock Hodge theory on metric spaces.
\newblock {\em Foundations of Computational Mathematics}, 12:1--48, 2012.

\bibitem{bauer2021ripser}
Ulrich Bauer.
\newblock Ripser: efficient computation of vietoris--rips persistence barcodes.
\newblock {\em Journal of Applied and Computational Topology}, 5(3):391--423, 2021.

\bibitem{berry2020spectral}
Tyrus Berry and Dimitrios Giannakis.
\newblock Spectral exterior calculus.
\newblock {\em Communications on Pure and Applied Mathematics}, 73(4):689--770, 2020.

\bibitem{blumberg2022stability}
Andrew~J Blumberg and Michael Lesnick.
\newblock Stability of 2-parameter persistent homology.
\newblock {\em Foundations of Computational Mathematics}, pages 1--43, 2022.

\bibitem{bobrowski2017topological}
Omer Bobrowski, Sayan Mukherjee, and Jonathan~E Taylor.
\newblock Topological consistency via kernel estimation.
\newblock 2017.

\bibitem{bronstein2021geometric}
Michael~M Bronstein, Joan Bruna, Taco Cohen, and Petar Veli{\v{c}}kovi{\'c}.
\newblock Geometric deep learning: Grids, groups, graphs, geodesics, and gauges.
\newblock {\em arXiv preprint arXiv:2104.13478}, 2021.

\bibitem{bull2020combining}
Joshua~A Bull, Philip~S Macklin, Tom Quaiser, Franziska Braun, Sarah~L Waters, Chris~W Pugh, and Helen~M Byrne.
\newblock Combining multiple spatial statistics enhances the description of immune cell localisation within tumours.
\newblock {\em Scientific reports}, 10(1):18624, 2020.

\bibitem{carlsson2009computing}
Gunnar Carlsson, Gurjeet Singh, and Afra Zomorodian.
\newblock Computing multidimensional persistence.
\newblock In {\em Algorithms and Computation: 20th International Symposium, ISAAC 2009, Honolulu, Hawaii, USA, December 16-18, 2009. Proceedings 20}, pages 730--739. Springer, 2009.

\bibitem{carlsson2007theory}
Gunnar Carlsson and Afra Zomorodian.
\newblock The theory of multidimensional persistence.
\newblock In {\em Proceedings of the twenty-third annual symposium on Computational geometry}, pages 184--193, 2007.

\bibitem{COIFMAN20065}
Ronald~R. Coifman and Stéphane Lafon.
\newblock Diffusion maps.
\newblock {\em Applied and Computational Harmonic Analysis}, 21(1):5--30, 2006.
\newblock Special Issue: Diffusion Maps and Wavelets.

\bibitem{desbrun2005discrete}
Mathieu Desbrun, Anil~N Hirani, Melvin Leok, and Jerrold~E Marsden.
\newblock Discrete exterior calculus.
\newblock {\em arXiv preprint math/0508341}, 2005.

\bibitem{donoho2003hessian}
David~L Donoho and Carrie Grimes.
\newblock Hessian eigenmaps: Locally linear embedding techniques for high-dimensional data.
\newblock {\em Proceedings of the National Academy of Sciences}, 100(10):5591--5596, 2003.

\bibitem{edelsbrunner2002topological}
Edelsbrunner, Letscher, and Zomorodian.
\newblock Topological persistence and simplification.
\newblock {\em Discrete \& computational geometry}, 28:511--533, 2002.

\bibitem{gigli2018nonsmooth}
Nicola Gigli.
\newblock {\em Nonsmooth differential geometry--an approach tailored for spaces with Ricci curvature bounded from below}, volume 251.
\newblock American Mathematical Society, 2018.

\bibitem{gigli2020lectures}
Nicola Gigli, Enrico Pasqualetto, et~al.
\newblock {\em Lectures on nonsmooth differential geometry}.
\newblock Springer, 2020.

\bibitem{grigor2006heat}
Alexander Grigor’yan.
\newblock Heat kernels on weighted manifolds and applications.
\newblock {\em Cont. Math}, 398(2006):93--191, 2006.

\bibitem{harrington2019stratifying}
Heather~A Harrington, Nina Otter, Hal Schenck, and Ulrike Tillmann.
\newblock Stratifying multiparameter persistent homology.
\newblock {\em SIAM Journal on Applied Algebra and Geometry}, 3(3):439--471, 2019.

\bibitem{harris2020array}
Charles~R. Harris, K.~Jarrod Millman, St{\'{e}}fan~J. van~der Walt, Ralf Gommers, Pauli Virtanen, David Cournapeau, Eric Wieser, Julian Taylor, Sebastian Berg, Nathaniel~J. Smith, Robert Kern, Matti Picus, Stephan Hoyer, Marten~H. van Kerkwijk, Matthew Brett, Allan Haldane, Jaime~Fern{\'{a}}ndez del R{\'{i}}o, Mark Wiebe, Pearu Peterson, Pierre G{\'{e}}rard-Marchant, Kevin Sheppard, Tyler Reddy, Warren Weckesser, Hameer Abbasi, Christoph Gohlke, and Travis~E. Oliphant.
\newblock Array programming with {NumPy}.
\newblock {\em Nature}, 585(7825):357--362, September 2020.

\bibitem{kambhatla1997dimension}
Nandakishore Kambhatla and Todd~K Leen.
\newblock Dimension reduction by local principal component analysis.
\newblock {\em Neural computation}, 9(7):1493--1516, 1997.

\bibitem{lim2023hades}
Uzu Lim, Harald Oberhauser, and Vidit Nanda.
\newblock Hades: Fast singularity detection with local measure comparison.
\newblock {\em arXiv preprint arXiv:2311.04171}, 2023.

\bibitem{otter2017roadmap}
Nina Otter, Mason~A Porter, Ulrike Tillmann, Peter Grindrod, and Heather~A Harrington.
\newblock A roadmap for the computation of persistent homology.
\newblock {\em EPJ Data Science}, 6:1--38, 2017.

\bibitem{robins1999towards}
Vanessa Robins.
\newblock Towards computing homology from finite approximations.
\newblock {\em Topology proceedings}, 24(1):503--532, 1999.

\bibitem{shinconfidence}
Jaehyeok Shin, Jisu Kim, Alessandro Rinaldo, and Larry Wasserman.
\newblock Confidence sets for persistent homology of the kde filtration.

\bibitem{singer2012vector}
Amit Singer and H-T Wu.
\newblock Vector diffusion maps and the connection laplacian.
\newblock {\em Communications on pure and applied mathematics}, 65(8):1067--1144, 2012.

\bibitem{stolz2020geometric}
Bernadette~J Stolz, Jared Tanner, Heather~A Harrington, and Vidit Nanda.
\newblock Geometric anomaly detection in data.
\newblock {\em Proceedings of the national academy of sciences}, 117(33):19664--19669, 2020.

\bibitem{sturm2018ricci}
Karl-Theodor Sturm.
\newblock Ricci tensor for diffusion operators and curvature-dimension inequalities under conformal transformations and time changes.
\newblock {\em Journal of Functional Analysis}, 275(4):793--829, 2018.

\bibitem{turkevs2021noise}
Renata Turke{\v{s}}, Jannes Nys, Tim Verdonck, and Steven Latr{\'e}.
\newblock Noise robustness of persistent homology on greyscale images, across filtrations and signatures.
\newblock {\em Plos one}, 16(9):e0257215, 2021.

\bibitem{vipond2020multiparameter}
Oliver Vipond.
\newblock Multiparameter persistence landscapes.
\newblock {\em Journal of Machine Learning Research}, 21(61):1--38, 2020.

\bibitem{Vipond}
Oliver Vipond, Joshua~A. Bull, Philip~S. Macklin, Ulrike Tillmann, Christopher~W. Pugh, Helen~M. Byrne, and Heather~A. Harrington.
\newblock Multiparameter persistent homology landscapes identify immune cell spatial patterns in tumors.
\newblock {\em Proceedings of the National Academy of Sciences}, 118(41):e2102166118, 2021.

\bibitem{von2023topological}
Julius Von~Rohrscheidt and Bastian Rieck.
\newblock Topological singularity detection at multiple scales.
\newblock In {\em International Conference on Machine Learning}, pages 35175--35197. PMLR, 2023.

\bibitem{yang2006robust}
Yong-Liang Yang, Yu-Kun Lai, Shi-Min Hu, Helmut Pottmann, et~al.
\newblock Robust principal curvatures on multiple scales.
\newblock In {\em Symposium on Geometry processing}, pages 223--226, 2006.

\bibitem{zomorodian2004computing}
Afra Zomorodian and Gunnar Carlsson.
\newblock Computing persistent homology.
\newblock In {\em Proceedings of the twentieth annual symposium on Computational geometry}, pages 347--356, 2004.

\end{thebibliography}

\end{document}